\documentclass[11pt, a4paper]{article}
\usepackage[english]{babel}
\usepackage{amsmath,amssymb,graphicx,latexsym,amsthm,mathrsfs,stmaryrd}
\usepackage[margin=2.2cm]{geometry}
\usepackage[hidelinks]{hyperref}

\newcommand{\Longdownarrow}{{\mbox{\rotatebox[origin=c]{-90}{$\Longrightarrow$}}}}
\newcommand{\Longuparrow}{{\mbox{\rotatebox[origin=c]{90}{$\Longrightarrow$}}}}
\newcommand{\assign}{:=}
\newcommand{\cdummy}{\cdot}
\newcommand{\colons}{\,:\,}
\newcommand{\infixand}{\text{ and }}
\newcommand{\nin}{\not\in}
\newcommand{\tmmathbf}[1]{\ensuremath{\boldsymbol{#1}}}
\newcommand{\tmop}[1]{\ensuremath{\operatorname{#1}}}
\newcommand{\tmstrong}[1]{\textbf{#1}}
\newcommand{\tmtextbf}[1]{\text{{\bfseries{#1}}}}
\newcommand{\tmtextit}[1]{\text{{\itshape{#1}}}}

\DeclareMathOperator{\supp}{supp}

\newtheorem{theorem}{Theorem}[section]
\newtheorem{lemma}[theorem]{Lemma}
\newtheorem{proposition}[theorem]{Proposition}
\newtheorem{corollary}[theorem]{Corollary}
\newtheorem{definition}[theorem]{Definition}
\newtheorem{remark}[theorem]{Remark}

\newcommand{\R}{\ensuremath{\mathbb{R}}}
\newcommand{\AssumptionA}{\hyperref[assumption:A]{Assumption A}}
\newcommand{\AssumptionAMixed}{\hyperref[assumption:A-mixed]{Assumption A-mixed}}
\newcommand{\AssumptionB}{\hyperref[assumption:B]{Assumption B}}
\newcommand{\AssumptionC}{\hyperref[assumption:C]{Assumption C}}
\newcommand{\AssumptionP}{\hyperref[assumption:P]{Assumption P}}
\newcommand{\AssumptionPContinuous}{\hyperref[assumption:P-continuous]{Assumption P-continuous}}
\newcommand{\AssumptionBLip}{\hyperref[assumption:B-Lip]{Assumption B-Lip}}
\newcommand{\assA}{\hyperref[assumption:A]{A}}
\newcommand{\assAMixed}{\hyperref[assumption:A-mixed]{A-mixed}}
\newcommand{\assB}{\hyperref[assumption:B]{B}}
\newcommand{\assC}{\hyperref[assumption:C]{C}}
\newcommand{\assP}{\hyperref[assumption:P]{P}}
\newcommand{\assPContinuous}{\hyperref[assumption:P-continuous]{P-continuous}}
\newcommand{\assBLip}{\hyperref[assumption:B-Lip]{B-Lip}}
\newcommand{\para}{\varolessthan}
\newcommand{\arap}{\varogreaterthan}

\begin{document}

\title{Viscosity Solutions for Singular HJB Equations: \\BSDE Representations and Stochastic Control}

\author{%
Dirk Becherer%
\thanks{Institut f\"ur Mathematik,
Humboldt-Universit\"at zu Berlin, Berlin, Germany.
Email:
\href{mailto:becherer@hu-berlin.de}
{\nolinkurl{becherer@hu-berlin.de}}.},
Nicolas Perkowski%
\thanks{Institut f\"ur Mathematik,
Freie Universit\"at Berlin, Berlin, Germany;
and Max Planck Institute for Mathematics in the Sciences,
Leipzig, Germany.
Email:
\href{mailto:perkowski@math.fu-berlin.de}
{\nolinkurl{perkowski@math.fu-berlin.de}}.},
Yuchen Sun%
\thanks{Humboldt-Universit\"at zu Berlin
and Technische Universit\"at Berlin,
Berlin, Germany.
Email:
\href{mailto:suyuchen@hu-berlin.de}
{\nolinkurl{suyuchen@hu-berlin.de}}.}
\text{ }and Carlos Villanueva Mariz%
\thanks{Institut f\"ur Mathematik,
Freie Universit\"at Berlin, Berlin, Germany.
Email:
\href{mailto:cv7031fu@fu-berlin.de}
{\nolinkurl{cv7031fu@fu-berlin.de}}.}
}

\maketitle

\begin{abstract}

	We introduce a notion of viscosity solution for Hamilton--Jacobi--Bellman
(HJB) equations with distributional drift, based on paracontrolled test
functions and related through a Zvonkin transformation to classical viscosity
theory. The equations considered are of the form
\[
\left(\partial_t+\frac12\Delta+b\cdot\nabla\right)h(t,x)
=-H(t,x,h(t,x),\nabla h(t,x)),
\]
where $b$ is singular in the sense of \cite{paradistrib} and has regularity
$\mathcal C^{-\alpha}$ for $\alpha\in(1/2,2/3)$. Using
doubling-of-variables arguments, we derive a priori gradient estimates that
also cover Hamiltonians with slightly superquadratic growth in $\nabla h$.
We also obtain probabilistic representations through singular control
problems for convex $H$ and weak singular forward--backward SDEs for possibly
nonconvex Hamiltonians with at most quadratic growth.
\end{abstract}

\begin{quotation}
	
\end{quotation}

\section{Introduction}
In this paper, we study Hamilton-Jacobi-Bellman equations of the form
\begin{equation}
	\label{eq:singular-HJB} \left( \partial_t + \frac{1}{2} \Delta + b \nabla
	\right) h (t, x) = -H (t, x, h (t, x), \nabla h (t, x)), \qquad h
	(T, \cdot) = \bar{h} (\cdummy),
\end{equation}
where \(b\in C([0,T];\mathcal{C}^{-\alpha}(\mathbb{R}^d;\mathbb{R}^d))\) and \(\alpha\in(1/2,2/3)\). For this distributional drift, Schauder estimates suggest that \(h\in\mathcal{C}^{(2-\alpha)-}\), so \(\nabla h\in\mathcal{C}^{(1-\alpha)-}\) and even at the linear level the product \(b\cdot\nabla h\) is singular. In order to deal with the local solution theory, one can use regularity structures \cite{Hairer2014}, paracontrolled distributions \cite{paradistrib} or the flow approach \cite{Duch2025}. Important examples arise both from singular SPDEs and from diffusions in rough environments. In particular, when \(H\equiv \frac{|\nabla h|^2}{2}\) and \(d=1\), this equation is closely related to the KPZ equation on the torus, see \cite{hairer_solving_2013} and \cite{KPZreloaded}.

A global solution theory for singular HJB equations via paracontrolled calculus was established in \cite{zhang_singular_2022}. Therein, \(L^p\) PDE theory and Zvonkin's transform are used to give a priori bounds for \eqref{eq:singular-HJB}. \(H\) is allowed to be subquadratic, or quadratic if \(d=1\), and Cole-Hopf's transform is (mostly) avoided. Additionally, solutions in weighted spaces are considered, a difficulty which we will not treat.

{
Our approach relies on extending the notion of viscosity solution to the
singular equation \eqref{eq:singular-HJB}. The classical deterministic theory
was developed in \cite{CrandallLions1983,CIL1992}. Viscosity solution concepts
for stochastic partial differential equations were subsequently introduced in
\cite{lions_fully_1998,lions_nonsmooth_1998} and further developed in
\cite{buckdahn_ma_2001a,buckdahn_ma_2001b}. Rough-path formulations followed in
\cite{caruana_rough_2011,gubinelli_tindel_torrecilla_2014}, and were recently
applied to pathwise relaxed optimal control in
\cite{chakraborty_pathwise_2024}. In these stochastic and rough settings, the
additional driving signal has low temporal regularity and formally differentiating
it produces a distribution in time, so the solution need not possess a
classical time derivative. In our setting, the source of the singularity is
instead spatial, since \(b\) is distribution-valued in space
and \(b\cdot\nabla h\) is not classically defined. More precisely, for an
ordinary smooth test function \(\varphi\), the term \(b\cdot\nabla\varphi\) is
only a distribution and therefore has no pointwise value, so the classical viscosity inequality cannot be formulated directly.

Following the ideas of
\cite{gubinelli_tindel_torrecilla_2014,chakraborty_pathwise_2024}, we
incorporate the singular part of the equation directly into the test functions, which will be
paracontrolled solutions of
\[
\left(\partial_t+\frac{1}{2}\Delta+b\cdot\nabla\right)\varphi=g,
\]
for some continuous \(g\). Although still singular, this linear equation is well understood and its solutions have regularity \(\mathcal{C}^{(2-\alpha)-}\); see
\cite{cannizzaro_multidimensional_2018}. Consequently, if \(h-\varphi\) has a
local extremum at \((t_0,x_0)\), the viscosity inequality can be formulated
entirely in terms of the pointwise quantities \(g(t_0,x_0)\),
\(\varphi(t_0,x_0)\), and \(\nabla\varphi(t_0,x_0)\). Our definition reduces to the
classical one when \(b\) is sufficiently regular, and it is also consistent with the usual notion of
paracontrolled solution, in the sense that every paracontrolled solution will be a
viscosity solution.

A recurring technique in stochastic and rough viscosity theories is the use of flow transformations which remove the irregular temporal term
from the equation (\cite{buckdahn_ma_2001a}, \cite{gubinelli_tindel_torrecilla_2014}). In a similar spirit, we will use Zvonkin's transformation
\cite{Zvonkin1974} to remove the spatially singular
term \(b\cdot\nabla h\). We prove that there is an appropriate time-dependent
change of variables \(\Psi\) such that, if \(h\) is a viscosity solution of
\eqref{eq:singular-HJB}, then
\(h(\cdot,\Psi(\cdot,\cdot))\) is a standard viscosity solution of a uniformly
elliptic PDE with Hölder continuous coefficients. Zvonkin's transformation was
already used in the paracontrolled analysis of \cite{zhang_singular_2022}, but our framework makes the tools of classical viscosity theory
available for the singular equation.

Following ideas from \cite{PorrettaPriola2013} and \cite{LeyNguyen2017}, we apply a doubling of variables argument to the transformed
equation and obtain a priori Lipschitz estimates.
In the convex setting, these estimates also apply to Hamiltonians whose growth
in the gradient variable is slightly above quadratic. Undoing the transformation
yields Lipschitz estimates for \eqref{eq:singular-HJB}, which in turn can be used to establish global well-posedness of paracontrolled solutions.

	Another aspect of the paper concerns stochastic representations of
	\eqref{eq:singular-HJB}. When the Hamiltonian is convex, one can express the
	paracontrolled solution as the value function of a singular stochastic control
	problem. This was already done for KPZ in \cite{KPZreloaded} in order to obtain
	an \(L^\infty\) maximum principle which, combined with the Cole--Hopf
	transform, yielded global well-posedness of the equation. It was also employed
	in \cite{perkowski_coming_2025} to study ``coming up from \(-\infty\)''
	properties (of KPZ). In the present work, the control formulation: provides a
	priori \(L^\infty\) bounds, which are later used in the Lipschitz estimates;
	shows that paracontrolled (in our setting, ``classical'') solutions are
	viscosity solutions; and yields existence and uniqueness of viscosity solutions for terminal data of
	low regularity, leveraging stability under uniform limits.

    The control formulation provides a natural framework for stochastic control in singular random environments. For instance, distributional drifts arise in one-dimensional diffusions in random potentials, such as Brox diffusion, and in tracer dynamics driven by stationary Burgers fields related to KPZ. In such examples, \eqref{eq:singular-HJB} describes the quenched value function of a controlled diffusion in a fixed rough environment. This opens the possibility of studying questions of stochastic homogenization, effective Hamiltonians and quenched large deviations for diffusions with distributional drift.

	Without convexity, we replace the variational representation by a singular
	BSDE representation. This leads to analogous viscosity and stability results
	for Hamiltonians that may depend on the solution variable and have at most
	quadratic growth in the gradient. Since the weak formulation needed for this
	representation is less standard, we discuss it in more detail next, before
	summarising our main results.
}

{For regular coefficients, the stochastic representation of semilinear
parabolic equations by Markovian backward stochastic differential equations is
known as the nonlinear Feynman--Kac formula and goes back to
{\cite{pardoux_peng_1992}}. We establish an analogue of this correspondence
for the singular HJB equation \eqref{eq:singular-HJB}. The associated
forward--backward system is formally given by
\begin{align}
	dX_t^{s,x}
	&=b(t,X_t^{s,x})\,dt+dW_t,
	\qquad X_s^{s,x}=x,
	\label{eq:singular-SDE}\\
	-dY_t^{s,x}
	&=H(t,X_t^{s,x},Y_t^{s,x},Z_t^{s,x})\,dt
	-Z_t^{s,x}\,dW_t,
	\qquad Y_T^{s,x}=\bar h(X_T^{s,x}).
	\label{eq:quadratic-BSDE}
\end{align}
Because $b$ is distribution-valued, the forward equation cannot be read in the
classical It\^o sense; its rigorous interpretation will be part of the weak
formulation described below. We show that a paracontrolled solution of the
singular HJB equation, evaluated along this singular diffusion, yields a
solution of \eqref{eq:singular-SDE}--\eqref{eq:quadratic-BSDE}. In the converse
direction, the FBSDE gives rise to the value field
\[
	u(t,x):=Y_t^{t,x},
\]
Although the forward equation is defined only weakly, this quantity is
deterministic and independent of the chosen natural-filtration rough weak
solution. Once $u$ is shown to be a continuous
universal decoupling field, it is a viscosity solution of the singular HJB
equation. This second direction is particularly useful for bounded continuous
terminal conditions, which are usually not regular enough to study paracontrolled solutions. 
For bounded continuous terminal data, continuity of the
value field is recovered through monotone approximation by smooth terminal
conditions and stability of quadratic BSDEs.

The BSDE representation complements the control formulation: it does not
require convexity and allows the Hamiltonian to depend on the solution variable
$y$, at the price of restricting its growth in $z$ to at most quadratic and
imposing a slightly stronger local Lipschitz condition. Besides yielding the
usual upper bound, it shows how, under a quadratic lower-growth condition, a
``coming up from \(-\infty\)'' estimate analogous to
{\cite{perkowski_coming_2025}} can be obtained.

Before proving this correspondence, we need a solution theory for the singular
FBSDE which does not presuppose a solution of the nonlinear HJB equation. The
difficulty already appears in the forward component: for the distributional
drifts considered here, \eqref{eq:singular-SDE} is available only through the
martingale-problem formulation of {\cite{cannizzaro_multidimensional_2018}} or
as a rough weak solution in the sense of {\cite{kremp_rough_2025}}. It is
therefore natural to formulate the entire forward--backward system weakly. We
work on the completed, right-continuous natural filtration of the forward
process, where the martingale representation property needed for the backward
equation is available, and prove well-posedness both for a fixed forward process
and in law across different realizations of the rough weak solution.

Weak formulations of FBSDEs have previously been studied, for instance, in
{\cite{antonelli_weak_2003,ma_weak_2008}}. In those works the weak formulation
is motivated by the coupling between the forward and backward equations. Our
system is decoupled: the need for a weak formulation comes instead from the
distributional drift in the forward equation.

Similar FBSDEs with distributional coefficients were studied by Issoglio and
Jing {\cite{issoglio_jing_2020}} in negative Sobolev spaces, building on the
singular SDE theory of {\cite{flandoli_issoglio_russo_2017}}. Here we work
instead in the rougher paracontrolled regime of enhanced distributional drifts,
using the singular SDE theory of
{\cite{cannizzaro_multidimensional_2018,kremp_rough_2025}}, and allow the
backward generator to have quadratic growth in the control variable.

A related probabilistic representation for the KPZ equation was obtained
earlier in {\cite{almada_budhiraja_2014}}. Since this work predates the
development of regularity structures and paracontrolled distributions, the
singular KPZ equation and its connection to FBSDEs are understood there through
mollified and renormalized approximations, rather than via an intrinsic
singular FBSDE formulation.

}

Taken together, the control and BSDE representations lead to the same basic
analytic picture. For the reader's convenience, we summarise its main
consequences. {Let \(\delta\in(0,T]\), and assume that either
Assumptions \assAMixed{} and \assB{} or Assumptions \assA{} and \assC{} hold,
together with Assumptions \assP{} and \assPContinuous{}.}
\begin{theorem}\label{thm:main1}If \(h\) is a paracontrolled solution of \eqref{eq:singular-HJB} defined on \([T-\delta,T]\), then it is a viscosity solution on that interval. Furthermore, there exists a constant \(C_1(T,\|\bar{h}\|_{L^\infty},H)>0\) independent of \(\delta\) such that
\[
\|h\|_{L^\infty_{[T-\delta,T]}L^{\infty}}\leqslant C_1(T,\|\bar{h}\|_{L^\infty},H)
\]
\end{theorem}
\begin{theorem}\label{thm:main2}
If \(h\) is a bounded viscosity solution of \eqref{eq:singular-HJB}  on \([T-\delta,T]\) and \(\bar{h}\in W^{1,\infty}\), there exists a constant \[C_2(T,\|h\|_{L^\infty_{[T-\delta,T]}L^{\infty}},\|\bar{h}\|_{W^{1,\infty}},b,H)>0\] such that
\[
\|h\|_{L^\infty_{[T-\delta,T]}W^{1,\infty}}\leqslant C_2(T,\|h\|_{L^\infty_{[T-\delta,T]}L^{\infty}},\|\bar{h}\|_{W^{1,\infty}},b,H).
\]
\end{theorem}
\begin{corollary}\label{corol:main1}
    For smooth enough \(\bar{h}\), there exists a global paracontrolled solution, which is also the unique bounded viscosity solution of \eqref{eq:singular-HJB}. Furthermore, it satisfies
    \[
    \|h\|_{L^\infty_{[0,T]}W^{1,\infty}}\leqslant C_3(T,\|\bar{h}\|_{W^{1,\infty}},b,H),
    \]
	for some \(C_3>0\).
\end{corollary}
\begin{corollary}\label{corol:main2}
    If \(\bar{h}\) bounded and uniformly continuous, there exists a unique bounded viscosity solution \(h\) of \eqref{eq:singular-HJB} on \([0,T]\). Furthermore, it satisfies
    \[
\|h\|_{L^\infty_{[0,T]}L^{\infty}}\leqslant C_1(T,\|\bar{h}\|_{L^\infty},H),
    \]
    and for some \(\nu\equiv\nu(H,b)\in(0,1)\) and \(C_4>0\)
    \[
    \|\nabla h(t,\cdot)\|_{L^{\infty}}\leqslant\frac{C_4(T,\|\bar{h}\|_{L^\infty},b,H)}{(T-t)^\frac{1}{2-\nu}},\quad t\in[0,T).
    \]
\end{corollary}

The remainder of the paper is organized as follows. In
Section~\ref{sec:prelims}, we recall the required facts from paracontrolled
calculus and collect the assumptions on the drift and the Hamiltonian. In
Section~\ref{sec:maindef}, we introduce the notion of viscosity solution and
establish its basic properties. Section~\ref{sec:control} treats the convex case
through the singular stochastic control problem. In Section~\ref{sec::zvonkin},
we develop the Zvonkin transformation and derive the a priori Lipschitz
estimates by doubling of variables arguments. Section~\ref{sec:BSDE} develops
the weak theory of singular quadratic FBSDEs and proves the nonlinear
Feynman--Kac correspondence in the non-convex setting. The paracontrolled
well-posedness results used throughout are collected in the Appendix.

\section{Preliminaries}\label{sec:prelims}

Let $\mathcal{S} (\mathbb{R}^d)$ denote the Schwartz space on $\mathbb{R}^d$
and $\mathcal{S}' (\mathbb{R}^d)$ its dual, the space of tempered
distributions. Let $(\rho_j)_{j \ge - 1} \subset C_c^{\infty} (\mathbb{R}^d)$
be non-negative functions such that
\begin{itemize}
	\item the support of $\rho_{- 1}$ is contained in a ball and the support of
	      $\rho_0$ is contained in an annulus,

	\item $\rho_j (\xi) = \rho_0  (2^{- j} \xi)$ for $j \ge 0$ and $\xi \in
	      \mathbb{R}^d$,

	\item $\sum_{j \ge - 1} \rho_j (\xi) = 1$, $\xi \in \mathbb{R}^d$,

	\item $\supp (\rho_i) \cap \supp (\rho_j) = \varnothing$ whenever $|i - j| >
	      1$.
\end{itemize}
Such a family $(\rho_j)_{j \ge - 1}$ is called a dyadic partition of unity.
Let $\mathcal{F}$ denote the Fourier transform
\[ \mathcal{F}f (z) = \int_{\mathbb{R}^d} e^{- 2 \pi ix \cdot z} f (x)
	\hspace{0.17em} dx, \qquad z \in \mathbb{R}^d, \]
and let $\mathcal{F}^{- 1}$ denote the inverse Fourier transform. For $u \in
	\mathcal{S}' (\mathbb{R}^d)$ we define the Littlewood--Paley blocks by
\[ \Delta_j u =\mathcal{F}^{- 1}  (\rho_j \mathcal{F}u), \qquad j \ge - 1. \]

\begin{definition}
	\label{def:holder-besov}For $\beta \in \mathbb{R}$ the H{\"o}lder--Besov
	space $\mathcal{C}^{\beta} (\mathbb{R}^d) \equiv B_{\infty,
			\infty}^{\beta} (\mathbb{R}^d)$ is defined by
	\[ \mathcal{C}^{\beta} (\mathbb{R}^d) = \{ u \in \mathcal{S}' (\mathbb{R}^d)
		: \|u\|_{\beta} \assign \sup_{j \ge - 1} 2^{j \beta} \| \Delta_j
		u\|_{L^{\infty} (\mathbb{R}^d)} < \infty \} . \]
	Such a notation is justified by the fact that for $\beta > 0$, $\beta \nin
		\mathbb{N}$, the space $\mathcal{C}^{\beta} (\mathbb{R}^d)$ coincides with
	the classical space of $\beta$-H{\"o}lder continuous functions. We also use
	the notation $\mathcal{C}^{\beta} (\mathbb{R}^d ; \mathbb{R}^d) \assign
		(\mathcal{C}^{\beta} (\mathbb{R}^d))^d$. Choosing a different dyadic partition of unity yields an equivalent norm.
\end{definition}

We can formally decompose the product of two tempered distributions $f, g \in
	\mathcal{S}' (\mathbb{R}^d)$ as
\[ f \cdummy g = f \para g + f \arap g + f \odot g, \]
where, denoting $S_j \equiv \sum_{i = - 1}^{j - 1} \Delta_i$,
\[ f \para g \assign \sum_{j \geqslant - 1} S_{j - 1} f \Delta_j g \quad
	\tmop{and} \quad f \arap g \assign \sum_{j \geqslant - 1} S_{j - 1} g
	\Delta_j f \]
are called \tmtextit{paraproducts}, and
\[ f \odot g \assign \sum_{j \geqslant - 1} \sum_{| i - j | \leqslant 1}
	\Delta_i f \Delta_j g \]
is the \tmtextit{resonant product}.

\begin{proposition}
	\label{paraestimates}$\tmmathbf{(\tmop{Paraproduct} \tmop{estimates})}$ For
	all $\beta \in \mathbb{R}$ and $f, g \in \mathcal{S}' (\mathbb{R}^d)$ we
	have
	\[ \| f \para g \|_{\mathcal{C}^{\beta}} \lesssim \| f \|_{L^{\infty}} \| g
		\|_{\mathcal{C}^{\beta}} . \]
	If $\alpha < 0$,
	\[ \| f \para g \|_{\mathcal{C}^{\alpha + \beta}} \lesssim \| f
		\|_{\mathcal{C}^{\alpha}} \| g \|_{\mathcal{C}^{\beta}}, \]
	and if furthermore $\alpha + \beta > 0$,
	\[ \| f \odot g \|_{\mathcal{C}^{\alpha + \beta}} \lesssim \| f
		\|_{\mathcal{C}^{\alpha}} \| g \|_{\mathcal{C}^{\beta}} . \]
\end{proposition}

For a more detailed study of Besov spaces and paraproducts, the interested
reader may consult Chapter~2 in {\cite{bahouri_fourier_2011}}. From now on, we
fix $\alpha \in (1 / 2, 2 / 3)$ and $T > 0$.

\begin{definition}[Enhanced drift]
	\label{Def:enhanced-drift}\label{enhanceddrift}We say that $b \equiv
		(b^1, \ldots, b^d) \in C_T \mathcal{C}^{- \alpha} (\mathbb{R}^d ;
		\mathbb{R}^d)$ {is enhanced if there exists
	$B\equiv(B^{i,j})_{i,j=1,\ldots,d}\in
	C_T\mathcal C^{1-2\alpha}(\mathbb R^d;\mathbb R^{d\times d})$
	such that $(b,B)\in\mathcal X^\alpha$, where
	\[
		\mathcal X^\alpha
		\assign\mathrm{cl}_{C_T\mathcal C^{-\alpha}\times
		C_T\mathcal C^{1-2\alpha}}
		\left\{\left(\eta,
		\bigl(J^T(\partial_{x_j}\eta^i)\odot\eta^j\bigr)_{i,j=1,\ldots,d}
		\right):
			\eta\in C_T\mathcal C^\infty(\mathbb R^d;\mathbb R^d)\right\}.
	\]
	Here $\mathrm{cl}$ means closure in $C_T \mathcal{C}^{-\alpha}
		(\mathbb{R}^d ; \mathbb{R}^d) \times C_T \mathcal{C}^{1-2\alpha}
		(\mathbb{R}^d ; \mathbb{R}^{d \times d})$, $J^T (u)$ is the solution of
	\[ \left( \partial_t + \frac{1}{2} \Delta \right) J^T (u) = -u, \qquad J^T
		(u) (T) = 0, \]}
	and $\odot$ denotes the resonant product of two distributions; see
	Chapter~2 of {\cite{bahouri_fourier_2011}} for more details.
\end{definition}

We are now ready to state the definition of the martingale solution to
\eqref{eq:singular-SDE}.

\begin{definition}[Martingale problem]
	\label{def:martingale-problem}Let $T > 0$ and $b \in C ([0, T] ;
		\mathcal{C}^{- \alpha} (\mathbb{R}^d ; \mathbb{R}^d))$ be a distributional
	drift admitting an enhancement $(b, B) \in \mathcal{X}^{\alpha}$.
	Consider the filtered probability space $(\Omega^c,
		\mathcal{F}^c, (\mathcal{F}_t^c)_{t \in [0, T]})$ with $\Omega^c = C ([0,
		T] ; \mathbb{R}^d)$, where $\mathcal{F}^c$ is the Borel
	$\sigma$-algebra on $\Omega^c$ and $(\mathcal{F}_t^c)$ is the canonical
	filtration.

	For $s \in [0, T]$ and $x \in \mathbb{R}^d$, a probability measure
	$\mathbb{P}^{s, x}$ on $(\Omega^c, \mathcal{F}^c, (\mathcal{F}_t^c)_{t \in
			[s, T]})$ is said to solve the martingale problem associated to the singular
	SDE \eqref{eq:singular-SDE} if the canonical process $X_t (\omega) = \omega
		(t)$ satisfies the following two properties under $\mathbb{P}^{s, x}$:
	\begin{enumerate}
		\item $\mathbb{P}^{s, x}  (X_s = x) = 1$;

		\item for all $f \in C_T L^{\infty} (\mathbb{R}^d)$ and every $u_T \in
		      \mathcal{C}^{2-\alpha} (\mathbb{R}^d)$, the process
		      \begin{equation}
			      \label{eq:martingale-problem-process} \left\{ u (t, X_t) - u (s, x) -
			      \int_s^t f (r, X_r) \hspace{0.17em} dr \right\}_{t \in [s, T]}
		      \end{equation}
		      is a square-integrable martingale, where $u$ is the unique solution of the
		      Kolmogorov backward equation
		      \begin{equation}
			      \label{eq:Kolmogorov-backward} \partial_t u + \frac{1}{2} \Delta u + b
			      \cdot \nabla u = f, \qquad u (T, \cdot) = u_T .
		      \end{equation}
	\end{enumerate}
\end{definition}

The existence and uniqueness, together with the Markovianity, of the
martingale solution is guaranteed by
{\cite[Theorem~4.3]{cannizzaro_multidimensional_2018}}.

\begin{theorem}[{\cite{cannizzaro_multidimensional_2018}}, Theorem~4.3]
	\label{thm:martingale-problem-wellposedness}Let $T > 0$ and $b \in C ([0,
		T] ; \mathcal{C}^{- \alpha} (\mathbb{R}^d ; \mathbb{R}^d))$ be a
	distributional drift admitting an enhancement $(b, B) \in
		\mathcal{X}^{\alpha}$. Then for each $s \in [0, T]$
	and $x \in \mathbb{R}^d$ there exists a unique probability measure
	$\mathbb{P}^{s, x}$ on $(\Omega^c, \mathcal{F}^c, (\mathcal{F}_t^c)_{t \in
			[s, T]})$ which solves the martingale problem associated to the singular SDE
	\eqref{eq:singular-SDE}. Moreover, the canonical process $(X_t)_{t \in [s,
			T]}$ under $\mathbb{P}^{s, x}$ is strong Markov.
\end{theorem}

{
\subsection{Summary of the main assumptions}

We collect here the common assumption on the singular drift and the three main
sets of assumptions on the Hamiltonian used below. The latter are not mutually
exclusive: the control representation uses the control assumptions, the
intrinsic FBSDE theory uses the BSDE assumptions, and the identification with
viscosity solutions combines the BSDE assumptions with the PDE assumptions.
Conditions on the terminal datum will be stated separately in each result.

\phantomsection\label{assumption:A}
\paragraph{Assumption A (enhanced drift).}
For some $\alpha\in(1/2,2/3)$, let
$b\in C_T\mathcal C^{-\alpha}(\mathbb R^d;\mathbb R^d)$ be a distributional
drift admitting an enhancement $(b,B)\in\mathcal X^\alpha$. \\

Throughout the remainder of the paper, \AssumptionA{} is in force and will not
be repeated in individual statements. {Later, when working with rough weak solutions in
Section~\ref{sec:BSDE}, we require the slightly stronger condition that the
mixed resonant products associated with the enhanced drift converge.}

{\phantomsection\label{assumption:A-mixed}
\paragraph{Assumption A-mixed (mixed enhanced drift).}
Assumption A holds. Moreover, there exists a fixed sequence
$b_n\in C_T(C_b^\infty(\mathbb R^d;\mathbb R^d))$ such that
\[
	\lim_{m,n\to\infty}\Bigg(
	\|b_n-b\|_{C_T\mathcal C^{-\alpha}}
	+\Big\|
		\bigl(J^T(\partial_{x_j}b_m^i)\odot b_n^j\bigr)_{i,j=1,\ldots,d}
		-B
	\Big\|_{C_T\mathcal C^{1-2\alpha}}
	\Bigg)=0.
\]
}

\phantomsection\label{assumption:C}
\paragraph{Assumption C (Control).}
The Hamiltonian is independent of the solution variable and is of the form
$H=H(t,x,p)$. There exists a jointly continuous function
$L:[0,T]\times\mathbb R^d\times\mathbb R^d\to\mathbb R$ such that
\begin{equation}\label{eq:legendre}
	H(t,x,p)=\sup_{v\in\mathbb R^d}\{p\cdot v-L(t,x,v)\}.
\end{equation}

Moreover, $L$ is uniformly superlinear in $v$, in the sense that
\[
	\lim_{R\to\infty}
	\inf_{\substack{t\in[0,T],\,x\in\mathbb R^d\\ |v|\geq R}}
	\frac{L(t,x,v)}{|v|}=+\infty.
\]
\[
	\sup_{t\in[0,T],\,x\in\mathbb R^d}
	\bigl(|H(t,x,0)|+|L(t,x,0)|\bigr)<\infty.
\]

\phantomsection\label{assumption:B}
\paragraph{Assumption B (Quadratic BSDE).}
The generator
$H:[0,T]\times\mathbb R^d\times\mathbb R\times\mathbb R^d\to\mathbb R$
is Borel measurable and continuous in $(y,z)$. There exists $C>0$ such that,
for every $t,x,y,y',z,z'$,
\[
	|H(t,x,y,z)-H(t,x,y',z')|
	\leq C\left[|y-y'|
	+(1+|y|+|y'|+|z|+|z'|)|z-z'|\right].
\]
\begin{equation}
	|H(t,x,y,z)|\leq C(1+|y|+|z|^2).
	\label{assump:quadratic growth}
\end{equation}

\phantomsection\label{assumption:P}
\paragraph{Assumption P (PDE).}
The Hamiltonian
$H:[0,T]\times\mathbb R^d\times\mathbb R\times\mathbb R^d\to\mathbb R$
is continuous. There exist $\beta\in(0,1]$, $C>0$,
\[
	\kappa\in\bigl[1,2+((\gamma-1)\wedge\beta)\bigr),
	\qquad q\in[0,2),
\]
such that, for all $t,x,x',y,y',p,p'$,
\[
\begin{aligned}
	|H(t,x,y,p)-H(t,x',y,p)|
	&\leq C|x-x'|^\beta(1+|p|^\kappa),\\
	|H(t,x,y,p)-H(t,x,y,p')|
	&\leq C\bigl(1+(|p|\vee|p'|)^{\kappa-1}\bigr)|p-p'|,\\
	|H(t,x,y,p)-H(t,x,y',p)|
	&\leq C(1+|p|^q)|y-y'|.
\end{aligned}
\]

\phantomsection\label{assumption:P-continuous}
\paragraph{Assumption P-continuous.}
The time dependence of the Hamiltonian is continuous in $L^\infty$ locally
uniformly in $(y,p)$: for every $R>0$,
\[
	\lim_{s\to t}
	\sup_{\substack{x\in\mathbb R^d\\|y|+|p|\leq R}}
	|H(s,x,y,p)-H(t,x,y,p)|=0.
\]
}

\section{Definition of viscosity solutions}\label{sec:maindef}

We are interested in defining viscosity solutions for singular HJB equations
of the form
\begin{equation}
	\left( \partial_t + \frac{1}{2} \Delta + b \nabla \right) h (t, x) = - H (t,
	x, h (t, x), \nabla h (t, x)), \quad h (T, \cdummy) =
	\overline{h} (\cdummy), \label{singularHJB}
\end{equation}
where $(t, x) \in [0, T) \times \mathbb{R}^d$ and $b$ is enhanced in the sense of Definition
\ref{enhanceddrift}. Suppose for a moment that $b \in C_T
	\mathcal{C}^{\infty}_b$, and assume $\varphi \in C^{1, 2} ([0, T) \times
	\mathbb{R}^d)$ is such that
\begin{enumerate}
	\item $\left( \partial_t + \frac{1}{2} \Delta + b (t, x) \nabla \right)
		      \varphi (t, x) = g (t, x)$ for every $(t, x) \in [0, T) \times
		      \mathbb{R}^d$, for some $g \in C_T \mathcal{C}_b$;

	\item $h - \varphi$ has a strict local maximum at $(t_0, x_0) \in [0, T)
		      \times \mathbb{R}^d$.
\end{enumerate}
In that case $h$ and $\varphi$ are classical solutions of their respective
PDEs, so it is clear (since $\Delta (h - \varphi) (t_0, x_0) \leqslant 0$ and
$\nabla (h - \varphi) (t_0, x_0) = 0$) that
\begin{equation}
	g (t_0, x_0) \geqslant - H (t_0, x_0, h (t_0, x_0), \nabla \varphi (t_0,
	x_0) ) . \label{viscineq}
\end{equation}
Notice that the last expression involves only derivatives of $\varphi$. Going
back to the singular case, i.e. enhanced $b$, this suggests that our test
functions should be solutions of (linear) paracontrolled PDEs
\[ \left( \partial_t + \frac{\Delta}{2} + b \nabla \right) \varphi = g, \quad
	\varphi (T, \cdummy) = \varphi^T . \]
Such equations have unique solutions in the paracontrolled sense: this means roughly that for a subspace of functions depending on the enhancement, one can make sense of the product \(b\star\nabla\varphi\) and look for mild solutions of the form
\[ \varphi (t, \cdummy) = e^{(T - t) \frac{\Delta}{2}} \left( \varphi^T \right)
	(\cdummy) + \int_t^T e^{(s - t) \frac{\Delta}{2}} (b \star \nabla \varphi (s,
	\cdummy)- g (s, \cdummy)) d s, \]
see the Appendix for more details. Such solutions have regularity $C_T \mathcal{C}^{(2 - \alpha) -} \cap C^{\left( \frac{2 -
				\alpha}{2} \right) -} L^{\infty}$, so in
particular they are differentiable in space, and
\eqref{viscineq} still makes sense. This prompts the following

\begin{definition}
	\tmtextbf{(Test functions)} \label{testfunctdef}For any $4 / 3 < \gamma < 2
		- \alpha$, we say that $\varphi \in C_T \mathcal{C}^{\gamma} (\mathbb{R}^d)
		\cap C_T^{\frac{\gamma}{2}} L^{\infty} (\mathbb{R}^d)$ is a test function,
	and denote it by $\varphi \in \mathcal{T}^b$, if there exists some $g \in
		C_T \mathcal{C}_b (\mathbb{R}^d)$ and $\varphi^T \in \mathcal{C}^{\gamma}
		(\mathbb{R}^d)$ such that $\varphi$ solves the paracontrolled PDE
	\begin{equation}
		\left( \partial_t + \frac{\Delta}{2} + b \nabla \right) \varphi = g\quad
		\text{in}\quad [0, T) \times \mathbb{R}^d, \quad \varphi (T, \cdummy) =
		\varphi^T (\cdummy) . \label{testeqdef}
	\end{equation}
\end{definition}

We now have all the ingredients necessary to define viscosity
(sub/super)solutions.

\begin{definition}
	\label{viscositydef}We say that $h \in \tmop{USC} ([0, T] \times
		\mathbb{R}^d)  (\tmop{respectively}, \tmop{LSC})$ is a
	$\tmop{subsolution}$ (resp. $\tmop{supersolution}$) of \eqref{singularHJB}
	if
	\begin{enumerate}
		\item For any $\varphi \in \mathcal{T}^b$ such that $h - \varphi$ has a
		      strict local maximum (minimum) at $(t_0, x_0) \in [0, T) \times
			      \mathbb{R}^d$,
		      \[ g (t_0, x_0) \geqslant (\leqslant) - H (t_0, x_0, h (t_0, x_0), \nabla
			      \varphi (t_0, x_0) ) ; \]
		\item $h (T, \cdummy) \leqslant (\geqslant) \overline{h} (\cdummy)$.
	\end{enumerate}
	We say that $h$ is a solution if it is both a sub and supersolution (so in
	particular $h \in C ([0, T] \times \mathbb{R}^d)$ and $h (T, \cdummy) =
		\overline{h}$).
\end{definition}

By adding or subtracting a constant, we may always assume that $h (t_0, x_0)
	= \varphi (t_0, x_0)$ at the local extrema. If $b$ is smooth, our notion of
viscosity solution coincides with the usual one: if $\varphi$ is an arbitrary
$C_b^{1, 2} ([0, T] \times \mathbb{R}^d)$ function, we can just take $g$ to be
\[ \partial_t \varphi + \frac{\Delta}{2} \varphi + b \nabla \varphi \in C_T
	\mathcal{C}_b (\mathbb{R}^d) . \]
In the usual viscosity setting, considering non strict local extrema
yields an equivalent definition. This is also the case here, but the proof requires Zvonkin's transform, see Corollary \ref{touching} later. Another result that follows from that Section, and we will need presently, is the following
\begin{lemma}[Chain rule]\label{zvonkin_chain_rule}
	There exists a (time dependent) \(C^1\) diffeomorphism of \(\mathbb{R}^d\), \(\Phi\), and bounded Lipschitz functions \(c\equiv(c_i)\) such that: if \(q\in C_b^{1,2}([0,T]\times\mathbb{R}^d)\) and we let
	\[
		\psi(t,x):=q(t,\Phi(t,x)),
	\]
	then $\psi\in\mathcal T^b$ and
	\[
		\left(\partial_t+\frac{\Delta}{2}+b\nabla\right)\psi=g,
		\qquad \psi(T,\cdot)=q(T,\cdot),
	\]
	where
	\begin{equation}
		g(t,x):=(\partial_tq)(t,\Phi(t,x))
		+\frac{1}{2}(\nabla\Phi\nabla\Phi^T)(t,x)\colons
		\nabla^2q(t,\Phi(t,x))
		+c(t,x)\cdot\nabla q(t,\Phi(t,x)).
		\label{eq:zvonkin-chain-force}
	\end{equation}
	In particular, there exists a constant $C$, depending only on \(\Phi\) and \(c\), such that
	\[
		\|\nabla\psi\|_\infty+\|g\|_\infty
		\leq C\|q\|_{C_b^{1,2}}.
	\]

\end{lemma}

Again let $4 / 3 < \gamma < 2
		- \alpha$. A "classical" paracontrolled solution of the singular HJB
equation \eqref{singularHJB} is a function $h \in C_T \mathcal{C}^{\gamma} 
		\cap C_T^{\frac{\gamma}{2}} L^{\infty}$
satisfying
\[ h (t, \cdummy) = e^{(T - t) \frac{\Delta}{2}} \left( \overline{h} \right)
	(\cdummy) + \int_t^T e^{(s - t) \frac{\Delta}{2}} \left(b \star \nabla h (s,
	\cdummy) + H (s, \cdummy, \nabla h (s, \cdummy))\right) d s. \]
In our language, this means that \(h\) is a test function with RHS equal to
\(-H (s, x, \nabla h (s, x))\). Contrary to the usual setting, it is not
immediately clear that \(h\) is a viscosity solution in the sense of
Definition~\ref{viscositydef}. This is established under \AssumptionC{} in
Theorem~\ref{paravicosity} and under \AssumptionB{} in
Corollary~\ref{cor:paracontrolled-viscosity-B}. However, we do have the following
\begin{proposition}
	\label{smoothcomparison}Assume there exists $h$, a paracontrolled solution
	of \eqref{singularHJB} with terminal condition $h (T, \cdummy) =
		\overline{h} \in \mathcal{C}^{\gamma}$ for some $\gamma \in (4/3,
		2-\alpha)$. {Assume moreover that, for every $R,S>0$, there exist
		$C_{R,S}>0$ and a modulus of continuity $\omega_{R,S}$ such that, for all
		$t,x,y,y',p,p'$ with $|y|,|y'|\leq R$ and $|p|,|p'|\leq S$,
		\[
			|H(t,x,y,p)-H(t,x,y',p')|
			\leq C_{R,S}|y-y'|+\omega_{R,S}(|p-p'|).
		\]}
	Then, if $v$ is a bounded subsolution (supersolution) with $v ({T}, \cdummy)
		\leqslant (\geqslant) \overline{h}$ in $\mathbb{R}^d$, we have that
	\[ v \leqslant (\geqslant) h \text{ in }  [0, T] \times \mathbb{R}^d . \]
	In particular, if there exists a bounded viscosity solution $v$ with $v (T, \cdummy)
		= \overline{h}$, necessarily $v = h$.

	\begin{proof}
		Let $v$ be a bounded subsolution. Choose
		$R>\|v\|_{L^\infty}\vee\|h\|_{L^\infty}$,
		$S>\|\nabla h\|_{L^\infty}+1$ and $\lambda>C_{R,S}$. For
		$\varepsilon>0$, set
		\[
			h_{\varepsilon}(t,x):=h(t,x)+\varepsilon e^{\lambda(T-t)}.
		\]
		Suppose that
		\[
			M_{\varepsilon}:=\sup_{(t,x)\in[0,T]\times\mathbb{R}^d}
			e^{\lambda(t-T)}\bigl(v(t,x)-h_{\varepsilon}(t,x)\bigr)>0.
		\]
		Fix a smooth function $\rho:\mathbb{R}^d\to[0,2]$ such that $\rho(0)=0$
		and $\rho(y)=2$ for $|y|\geq1$. For
		$0<\delta<M_{\varepsilon}/2$, choose $(t_{\delta},x_{\delta})$ such that
		\[
			e^{\lambda(t_{\delta}-T)}\bigl(v(t_{\delta},x_{\delta})
			-h_{\varepsilon}(t_{\delta},x_{\delta})\bigr)>M_{\varepsilon}-\delta
		\]
		and, using Lemma \ref{zvonkin_chain_rule}, set
		\[
			\psi_{\delta}(t,x):=\delta\rho\bigl(\Phi(t,x)
			-\Phi(t_{\delta},x_{\delta})\bigr).
		\]
		The function
		\[
			(t,x)\longmapsto e^{\lambda(t-T)}\bigl(v(t,x)-h_{\varepsilon}(t,x)\bigr)
			-\psi_{\delta}(t,x)
		\]
		attains its maximum at some $(s_{\delta},z_{\delta})$. Indeed, at
		$(t_{\delta},x_{\delta})$ its value is larger than
		$M_{\varepsilon}-\delta$, while on
		$|\Phi(t,x)-\Phi(t_{\delta},x_{\delta})|\geq1$ it is at most
		$M_{\varepsilon}-2\delta$. The maximum is therefore attained in a compact
		set. Its value, denoted by $R_{\delta}$, is positive, and the terminal
		condition implies $s_{\delta}<T$.

		Lemma~\ref{zvonkin_chain_rule} shows that this change of variables gives
		$\psi_{\delta}\in\mathcal T^b$ and, uniformly in the centre,
		\[
			\|\nabla\psi_{\delta}\|_{L^\infty}
			+\left\|\left(\partial_t+\frac{\Delta}{2}+b\nabla\right)
			\psi_{\delta}\right\|_{L^\infty}\leq C\delta.
		\]
		Consequently,
		\[
			\varphi_{\delta}(t,x):=h_{\varepsilon}(t,x)
			+e^{\lambda(T-t)}\bigl(R_{\delta}+\psi_{\delta}(t,x)\bigr)
		\]
		is a test function touching $v$ from above at
		$(s_{\delta},z_{\delta})$. At the contact point,
		\[
			v-h=e^{\lambda(T-s_{\delta})}
			\bigl(\varepsilon+R_{\delta}+\psi_{\delta}\bigr)>\varepsilon.
		\]
		The viscosity inequality gives, with all functions on the right-hand side
		evaluated at $(s_{\delta},z_{\delta})$,
		\[
			\lambda(v-h)\leq H\bigl(s_{\delta},z_{\delta},v,\nabla h
			+e^{\lambda(T-s_{\delta})}\nabla\psi_{\delta}\bigr)
			-H\bigl(s_{\delta},z_{\delta},h,\nabla h\bigr)
			+e^{\lambda(T-s_{\delta})}
			\left(\partial_t+\frac{\Delta}{2}+b\nabla\right)\psi_{\delta}.
		\]
		By the assumptions on $H$ and the preceding bounds, this yields
		\[
			(\lambda-C_{R,S})(v-h)(s_{\delta},z_{\delta})\leq o_{\delta}(1).
		\]
		Since the left-hand side is bounded below by
		$(\lambda-C_{R,S})\varepsilon>0$, letting $\delta\to0$ gives a contradiction.
		Thus $v\leq h_{\varepsilon}$, and then $v\leq h$ by letting
		$\varepsilon\to0$. The lower comparison follows in the same way by taking
		$h_{\varepsilon}=h-\varepsilon e^{\lambda(T-t)}$ and penalizing
		$e^{\lambda(t-T)}(h_{\varepsilon}-v)$.
	\end{proof}
\end{proposition}
The previous Proposition tells us that whenever a paracontrolled solution
$h$ exists, then it is the only reasonable candidate for a (bounded) viscosity solution.

Our strategy in what remains can be summarised as follows
\[ \begin{array}{ccc}
     h (\tmop{local}) \tmop{paracontrolled} &  & \tmop{Global}
     \tmop{paracontrolled}\\
     \tmop{Control} \Longdownarrow \tmop{BSDE} &  & \Longuparrow\\
     h \tmop{viscosity} + L^{\infty} \tmop{bound} &
     \overset{\tmop{Zvonkin}}{\Longrightarrow} & a \tmop{priori} W^{1,
     \infty}_x
   \end{array} \]

\section{Convex Hamiltonian and singular control problem}\label{sec:control}
Throughout this section, we will work under \AssumptionC{}. We will prove that whenever a paracontrolled solution $h$ exists, then
\begin{equation}
	h (t, x) = \underset{v}{\sup} \mathbb{E}_{t, x} \left[ \overline{h}
		(\gamma^{t, x, v}_T) - \int_t^T L
		(s, \gamma_s^{t, x, v}, v_s) d s \right] = \underset{v}{\sup} J (t, x ; u)
	\label{valuefunc},
\end{equation}
where
\begin{equation}
	\gamma^{t, x, v}_s = x + \int_t^s (v_u + b (u, \gamma^{t, x, v}_s)) d u +
	(W_s - W_t), \label{sde} \quad s \in [t, T] ,
\end{equation}
the precise meaning of which is explained bellow. Once this is established, we prove that the
Dynamic Programming Principle holds (Proposition \ref{strongDPP}), which we
can use in turn to see that any paracontrolled solution is a viscosity
solution. To summarize,
\[ \begin{array}{ccc}
		h \tmop{paracontrolled}        &                 & h \tmop{viscosity} \tmop{solution} \\
		\Longdownarrow                 &                 & \Longuparrow                       \\
		h \tmop{value} \tmop{function} & \Longrightarrow & \tmop{DPP}
	\end{array} \]
    
With the stochastic representation we also obtain, in a straightforward manner, global \(L^\infty\) bounds. These will allow us to establish a priori Lipschitz bounds via viscosity techniques, see Subsection \ref{sec::zvonkin}. 

The well posedness of the paracontrolled solution on the whole interval \([0,T]\) is not clear in
principle: in \cite{zhang_singular_2022}, the best result in this direction,
the Hamiltonian is allowed to grow subquadratically in the gradient, or
quadratically if \(d=1\), while \AssumptionC{} also permits superquadratic
Hamiltonians. Global well posedness can be established a posteriori
using the Lipschitz bounds (so to be fully rigourous, one should work in this
section up to a blow-up time \([T-T^*,T]\), and deduce afterwards that we can
take \(T^*\rightarrow0\)).

The first step is defining the controlled singular diffusions, after which we prove a verification result.
\begin{definition}\label{controlledSDE}
Let $(\Omega,\mathcal F,(\mathcal F_t)_{t\in[0,T]},\mathbb P)$ be a filtered probability space
and let $v$ be a progressively measurable process. For an enhanced drif
$b\in C_T C^{-\alpha}$, we say that an adapted
and continuous stochastic process $\gamma$ is a \emph{controlled martingale solution} of
\[
\gamma_s
=
x+\int_t^s \bigl(b(u,\gamma_u)+v_u\bigr)\,du+W_s-W_t,
\qquad t\le s\le T,
\]
if $\mathbb P(\gamma_t=x)=1$ and whenever
$\varphi_T\in C^\gamma$, $\gamma\in(4/3,2-\alpha)$, $f\in C_T L^\infty$, and $\varphi$ solves the
paracontrolled PDE
\[
\left(\partial_t+\frac12\Delta+b\partial_x\right)\varphi=f,
\qquad
\varphi(T)=\varphi_T
\]
on $[0,T]$, then
\[
\varphi(s,\gamma_s)-\varphi(t,x)
-
\int_t^s
\Bigl(
f(u,\gamma_u)+\partial_x\varphi(u,\gamma_u)v_u
\Bigr)\,du,
\qquad s\in[t,T],
\]
is a $\mathbb{P}$-martingale. We say that $(\Omega,\mathcal F,(\mathcal F_t)_{t\in[0,T]},\mathbb{P},v)$ (or for convenience, $v$) is \emph{admissible} if there exists a $\gamma$ on this space solving the martingale problem.
\end{definition}
Notice that the underlying filtered space may vary, and that for an admissible $v$ we make no claims about the uniqueness of $\gamma$. If $v=0$ and we fix the canonical space, we recover Definition \ref{def:martingale-problem}, and uniqueness does hold.

For the payoff functional we keep the notation $J(t,x;u)$ for simplicity, but rigorously it is a function of the sextuple $(\Omega,\mathcal F,(\mathcal F_t)_{t\in[0,T]},\mathbb{P},v,\gamma)$. 
\begin{proposition}
	$\tmmathbf{(\tmop{Verification})}$\label{verification}
	Let $h$ be a
	paracontrolled solution of \eqref{singularHJB}. Then
	\begin{enumerate}
		\item $h (t, x) \geqslant J (t, x ; u)$ for any admissible drift $u$;

		\item If there exists admissible $(\Omega^{\star}, \{
			      \mathcal{F}_s^{\star} \}, \mathbb{P}^{\star}, \gamma^{\star}, u^{\star})$
		      such that
		      \begin{equation}
			      \begin{array}{lll}
				      u^{\star} (s) & \in & \underset{v}{\arg \max} [v \cdummy \nabla h (s,
					      \gamma^{\star}_s) - L (s, \gamma_s^\star, v)]
			      \end{array} \label{optdrift}
		      \end{equation}
		      for $\tmop{Leb} \times \mathbb{P}^{\star}$ a.e. $(s, \omega) \in [t, T]
			      \times \Omega^{\star}$, then $h (t, x) = J (t, x ; u^{\star})$, and
		      therefore
		      \[ h (t, x) = \underset{v}{\sup} \mathbb{E}_{t, x} \left[ \overline{h}
				      (\gamma^{t, x, v}_T) -
				      \int_t^T L (s, \gamma_s^{t, x, v}, v_s) d s \right] , \]
            where the supremum is taken over the class of all admissible $v$.
	\end{enumerate}
	\begin{proof}
		The proof is similar to Theorem 7.13 in {\cite{KPZreloaded}}. For any base
		$(\Omega, \{ \mathcal{F}_r \}, \mathbb{P}, \gamma, u)$, and since $h$ is a
		paracontrolled solution
		\[ h (s, \gamma_s) - h (t, x) - \int_t^s (- H (r, \gamma_r, \nabla h
			(r, \gamma_r)))d r - \int_t^s u_r \cdummy
			\nabla h (r, \gamma_r) d r, \quad s \in [t, T], \]
		is a martingale by Definition \ref{controlledSDE}, and therefore
		(taking s=T, and by definition of $H$)
		\[ \begin{array}{lll}
				h (t, x) & =         & \mathbb{E}_{\mathbb{P}} \left[ \overline{h} (\gamma_T)
					 + \int_t^T (- u_s \cdummy \nabla h (s,
				\gamma_s) + H (s, \gamma_s, \nabla h (s, \gamma_s))) d s \right]              \\
				         & \geqslant & \mathbb{E}_{\mathbb{P}} \left[ \overline{h} (\gamma_T)
					- \int_t^T L (s, \gamma_s, u_s) d s
					\right] = J (t, x, u),
			\end{array} \]
		so we have that $h (t, x) \geqslant V (t, x)$. If there exists some base
		$(\Omega^{\star}, \{ \mathcal{F}_s^{\star} \}, \mathbb{P}^{\star},
			\gamma^{\star}, u^{\star})$ satisfying \eqref{optdrift}, the last
		inequality turns to an equality, that is
		\[ \begin{array}{lll}
				h (t, x) & = & \mathbb{E}_{\mathbb{P}^{\star}} \left[ \overline{h}
					(\gamma^{\star}_T) - \int_t^T
					L (s, \gamma^{\star}_s, u^{\star}_s) d s \right] = J (t, x ;
				u^{\star}) .
			\end{array} \]
	\end{proof}
\end{proposition}
Regarding the existence of such an optimal control, we have the following
\begin{lemma}
{For every paracontrolled solution
$h$ as in Proposition~\ref{verification}, there exists an optimal
$(\Omega^{\star}, \{\mathcal{F}_s^{\star}\}, \mathbb{P}^{\star},
\gamma^{\star}, u^{\star})$.}
                  
\begin{proof}
By the coercivity assumption, and since $\nabla h\in C_TC_b$, we can use a measurable selection theorem (for instance Corollary 14.6 in \cite{rockafellar_wets}) to guarantee the existence of $u^\star:[0,T]\times\mathbb{R}^d\xrightarrow{}\mathbb{R}^d$ such that
\begin{enumerate}
    \item $u^\star$ is Borel measurable and bounded;
    \item For $d t\times \text{Leb}$ almost every $(s,y)$,
    \begin{equation}\label{control_selector}
		\begin{array}{lll}
			u^{\star} (s,y) & \in & \underset{v}{\arg \max} [v \cdummy \nabla h (s,y) - L (s, y,v)].
		\end{array}
	\end{equation}
\end{enumerate}
Let \((\Omega^\star,\{\mathcal{F}_s^\star\},\gamma^\star)\) be the canonical path space, and \(\mathbb{P}_{t,x}\) the unique martingale solution of
\[
\gamma_s
=
x+\int_t^s b(u,\gamma_u)\,du+W_s-W_t,
\qquad t\le s\le T.
\]
Following the Girsanov argument of
\cite[Lemma~2.11]{perkowski_coming_2025}, and letting
\(\tilde{W_\cdot}:=W_\cdot-W_t\), define
\[
\frac{\mathrm{d}\mathbb{Q}^\star}{\mathrm{d}\mathbb{P}_{t,x}}:=\operatorname{exp}\left(\int_t^{T}u^\star(s,\gamma_s)d\tilde{W}_s-\frac{1}{2}\int_t^{T}|u^\star(s,\gamma_s)|^2ds\right),
\]
Then, under \(\mathbb{Q}^\star\), \(\gamma\) solves the controlled martingale
problem associated with \(v^\star\) (which is bounded, so Novikov's condition
holds). There remains only to check that \(dt\times\mathbb{Q}^\star\) a.s.
\[
		\begin{array}{lll}
			u^{\star} (s,\gamma_s) & \in & \underset{v}{\arg \max} [v \cdummy \nabla h (s,\gamma_s) - L (s, \gamma_s,v)].
		\end{array}
\]
This follows from the fact that \(\mathbb{Q}^\star\ll\mathbb{P}_{t,x}\), and that under \(\mathbb{P}_{t,x}\) the process \(\gamma\) has Gaussian-type densities (Corollary 2.15 in \cite{perkowski_coming_2025}). Indeed, letting \(\mathcal{N}\) be the null set of \([0,T]\times\mathbb{R}^d\) under which \eqref{control_selector} holds,
\[
\int_t^T\mathbb{P}_{t,x}[(s,\gamma_s)\in\mathcal{N}]d s=0.
\]
\end{proof}
\end{lemma}
\begin{remark}
	In the previous result, we made no assertions about uniqueness in law of the (optimal) controlled process. If for example $H (\nabla h) = \frac{1}{p} | \nabla h |^p =
		\underset{v \in \mathbb{R}^d}{\sup} \left\{ v \cdummy \nabla h -
		\frac{1}{p'} | v |^{p'} \right\}$ for $p>1$ and $\frac{1}{p}
		+ \frac{1}{p'} = 1$, then
	\[ u_s^{\star} = | \nabla h (s, \gamma^{\star}_s) |^{p-2}
		\nabla h (s, \gamma^{\star}_s) . \]
	For a paracontrolled solution $h$ of the corresponding HJB equation,
	\[ | \nabla h (\cdummy, \cdummy) |^{p-2} \nabla h (\cdummy,
		\cdummy) \in C_T \mathcal{C}^{\varepsilon}, \quad \tmop{for} \tmop{some}
		\varepsilon > 0, \]
	so $\tilde{b} = b + | \nabla h |^{p-2} \nabla h$ is also an enhanced drift, and the singular
	SDE
	\[ \gamma_s = x + \int_t^s \tilde{b} (r, \gamma_r) d u + (W_s - W_t), \quad
		s \in [t, T], \]
	has a unique solution (in the martingale sense).
\end{remark}

An immediate consequence of the stochastic representation is the following maximum principle.
\begin{corollary}\label{maxpple}
    Let \(h\) be a paracontrolled solution of
    \eqref{singularHJB}. Then,
    \[
\|h\|_\infty\leqslant\|\overline{h}\|_\infty+T\left(\sup_{t,x}L(t,x,0)\vee H(t,x,0)\right).
    \]
\end{corollary}
\begin{proof}
    By Proposition \ref{verification}, for any \((t,x)\)
    \begin{align*}
        h(t,x)
        &\leqslant\|\overline{h}\|_\infty-\mathbb{E}_{t,x}\left[\int_t^T\inf_{v\in\mathbb{R}^d}L(s,\gamma_s,v)d s\right]\\
        &\leqslant\|\overline{h}\|_\infty+\mathbb{E}_{t,x}\left[\int_t^TH(s,\gamma_s,0)d s\right]\\
        &\leqslant\|\overline{h}\|_\infty+(T-t)\sup_{t,x}H(t,x,0).
    \end{align*}
    The lower bound is proved similarly by choosing the control \(v=0\).
\end{proof}

In a similar way to Proposition \ref{verification}, we can show that any paracontrolled solution satisfies the
Dynamic Programming Principle.

\begin{proposition}
	\label{strongDPP}Let $h$ be a paracontrolled solution of
	$\eqref{singularHJB}$. Then
	\begin{enumerate}
		\item For any base $(\Omega, \{ \mathcal{F}_s \}, \mathbb{P}, \gamma, u)$
		      and $\mathcal{F}_s$-stopping time $\theta$ such that $t \leqslant \theta
			      \leqslant T$
		      \[ h (t, x) \geqslant \mathbb{E}_{\mathbb{P}} \left[ h (\theta,
				      \gamma_{\theta}) -
				      \int_t^{\theta} L (s, \gamma_s, u_s) d s \right] \]
		\item If $(\Omega^{\star}, \{ \mathcal{F}_s^{\star} \},
			      \mathbb{P}^{\star}, \gamma^{\star}, u^{\star})$ is as in the Proposition
		      \ref{verification} and $\theta^{\star}$ is a $\mathcal{F}_s$-stopping time
		      with $t \leqslant \theta^{\star} \leqslant T$, then
		      \[ h (t, x) =\mathbb{E}_{\mathbb{P}^{\star}} \left[ h (\theta^{\star},
				      \gamma_{\theta^{\star}}) - \int_t^{\theta^{\star}} L (s, \gamma^{\star}_s,
				      u^{\star}_s) d s \right], \]
		      so in particular (by the first point),
		      \[ h (t, x) = \underset{v}{\sup} \mathbb{E}_{t, x} \left[ h (\theta,
				      \gamma_{\theta}) -
				      \int_t^{\theta} L (s, \gamma_s^{t, x, v}, v_s) d s \right] . \]
	\end{enumerate}
	\begin{proof}
		For any such base and stopping time, we have again that
		\[ h (s \wedge \theta, \gamma_{s \wedge \theta}) - h (t, x) - \int_t^{s
			\wedge \theta} (- H (r, \gamma_r, \nabla h (r, \gamma_r)))  d r - \int_t^{s \wedge \theta} u_r \cdummy \nabla h
			(r, \gamma_r) d r, \quad s \in [t, T], \]
		is a martingale. It follows as before, since $\theta \wedge T = \theta$,
		that
		\[ \begin{array}{lll}
				h (t, x) & \geqslant & \mathbb{E}_{\mathbb{P}} \left[ h (\theta,
					\gamma_{\theta}) -
					\int_t^{\theta} L (s, \gamma_s, u_s) d s \right] .
			\end{array} \]
		If instead $u^{\star}$ is an optimal control with base $(\Omega^{\star},
			\{ \mathcal{F}_s^{\star} \}, \mathbb{P}^{\star}, \gamma^{\star},
			u^{\star})$, the previous inequality is an equality.
	\end{proof}
\end{proposition}

We now have all the ingredients necessary to see that paracontrolled solutions
are viscosity solutions. The proof roughly follows Section 3.2 of
	{\cite{huyen}}.

\begin{theorem}
	\label{paravicosity}Let $h$ be a paracontrolled solution of
	\eqref{singularHJB}. Then $h$ is a viscosity solution.
\end{theorem}

\begin{proof}
	The continuity and terminal data properties are clear ($h \in C_T C^{(2
				- \alpha) -}$), so we need only check the sub/supersolution properties.

	\tmtextbf{(Supersolution)} Let $\varphi \in \mathcal{T}^b$ with
	\[ \left( \partial_t + \frac{\Delta}{2} + b \nabla \right) \varphi = g, \]
	and $(t_0, x_0) \in [0, T) \times \mathbb{R}^d$ such that $h - \varphi$ has
	a strict local minimum at $(t_0, x_0)$ (remember we may assume wlog that $h
		(t_0, x_0) = \varphi (t_0, x_0)$). In particular there exists a set
	\[ B_{\eta_1, \eta_2} (t_0, x_0) \assign \left\{ (s, x) : \quad t_0
		\leqslant s < t_0 + \eta_1, \quad | x - x_0 | < \eta_2 \right\} \subset
		[0, T) \times \mathbb{R}^d, \]
	such that
	\[ h (s, x) > \varphi (s, x), \quad \tmop{for} \tmop{every} \quad (s, x) \in
		B_{\eta_1, \eta_2} (t_0, x_0) \backslash \{ (t_0, x_0) \} . \]
	Define the stopping time
	\[ \tau = \inf \{ s \geqslant t_0 : (s, \gamma_s^{t_0, x_0}) \notin
		B_{\eta_1, \eta_2} (t_0, x_0) \} . \]
	Letting $(\varepsilon_m)$ be a sequence such that
	\[ \varepsilon_m < \eta_1 \tmop{for} \tmop{every} n \quad \tmop{and} \quad
		\varepsilon_m \overset{m \rightarrow \infty}{\rightarrow} 0, \]
	we can consider $\theta_m = \tau \wedge (t_0 + \varepsilon_m)$. Fix a
	constant, deterministic control $v \in \mathbb{R}^d$; by Proposition
	\ref{strongDPP},
	\[ h (t_0, x_0) \geqslant \mathbb{E} \left[ h (\theta_m,
			\gamma_{\theta_m}^{t_0, x_0, v})  - \int_{t_0}^{\theta_m} L (s, \gamma_s^{t_0,
				x_0, v}, v) d s \right], \]
	and by the definition of the stopping time,
	\[ \varphi (t_0, x_0) = h (t_0, x_0) \geqslant \mathbb{E} \left[ \varphi
			(\theta_m, \gamma_{\theta_m}^{t_0, x_0, v})  - \int_{t_0}^{\theta_m} L (s, \gamma_s^{t_0,
				x_0, v}, v) d s \right] . \]
	On the other hand (again by Definition \ref{controlledSDE})
	\[ \varphi (s \wedge \theta_m, \gamma_{s \wedge \theta_m}^{t_0, x_0, v}) -
		\varphi (t_0, x_0) - \int^{s \wedge \theta_m}_{t_0} (g + v \cdummy \nabla
		\varphi) (u \wedge \theta_m, \gamma_{u \wedge \theta_m}^{t_0, x_0, v}) d
		u = \text{ Martingale}, \]
	and then
	\[ 0 \geqslant \mathbb{E} \left[ \int^{\theta_m}_{t_0} (g + v \cdummy \nabla
			\varphi) (s, \gamma_s^{t_0, x_0, v}) d s - \int_{t_0}^{\theta_m} L (s, \gamma_s^{t_0,
				x_0, v}, v) d s \right], \]
	so dividing both sides by $\varepsilon_m$,
	\[ 0 \geqslant \frac{1}{\varepsilon_m} \mathbb{E} \left[
			\int^{\theta_m}_{t_0} (g + v \cdummy \nabla \varphi) (s, \gamma_s^{t_0,
				x_0, v}) d s -
			\int_{t_0}^{\theta_m} L (s, \gamma_s^{t_0, x_0, v}, v) d s \right] . \]
	Since for each $\omega$, $\theta_m (\omega) = t_0 + \varepsilon_m$ for big
	enough $m$ (by continuity of the sample paths), we can take $m \rightarrow
		\infty$ and use the mean value theorem to conclude that
	\[ g (t_0, x_0) \leqslant - (v \cdummy \nabla \varphi (t_0,
		x_0) - L (t_0, x_0, v)), \]
	and since this holds for any arbitrary $v \in \mathbb{R}^d$,
	\[ g (t_0, x_0) \leqslant  - H (t_0, x_0, \nabla \varphi (t_0,
		x_0)) . \]
	\tmtextbf{(Subsolution)} Assume now that $h - \varphi$ has a strict local
	maximum at $(t_0, x_0) \in (0, T) \times \mathbb{R}^d$ (again wlog $(h -
		\varphi) (t_0, x_0) = 0 ; \tmop{the} t_0 = 0$ case is similar); let us
	suppose that
	\[ g (t_0, x_0) <  - H (t_0, x_0, \nabla \varphi (t_0, x_0)),
	\]
	and we will arrive at a contradiction. By continuity of the functions
	involved, there exists a neighbourhood
	\[ \tilde{B}_{\eta_1, \eta_2} (t_0, x_0) : = \left\{ (s, x) : \quad | s -
		t_0 | < \eta_1, \quad | x - x_0 | < \eta_2 \right\} \subset [0, T) \times
		\mathbb{R} \]
	such that
	\begin{enumerate}
		\item $g (t , x ) <  - H (t, x, \nabla \varphi (t, x))$ for
		      every $(t, x) \in \tilde{B}_{\eta_1, \eta_2} (t_0, x_0) ;$

		\item $\underset{\overline{\tilde{B}_{\eta_1, \eta_2}} (t_0, x_0)}{\max}
			      (h - \varphi) (t, x) - \underset{\partial (\tilde{B}_{\eta_1, \eta_2}
				      (t_0, x_0))}{\sup} (h-\varphi) (t, x) = (h - \varphi) (t_0, x_0) -
			      \underset{\partial (\tilde{B}_{\eta_1, \eta_2} (t_0, x_0))}{\sup}(h-\varphi) (t, x)
			      \assign \delta > 0,$
	\end{enumerate}
	where $2$ is possible because we have a strict local maximum at $(t_0, x_0)$
	(here $\overline{B}$ and $\partial B$ denote the closure and boundary as
	subsets of $\mathbb{R} \times \mathbb{R}^d$). Similarly to before, choose
	the stopping time
	\[ \theta = \inf \left\{ s \geqslant t_0 : \quad (s, \gamma_s^{t_0, x_0})
		\notin \tilde{B}_{\eta_1, \eta_2} (t_0, x_0) \right\} . \]
	Fixing $\varepsilon = \delta / 2$, by Proposition \ref{strongDPP} there exists a
	control $v^{\varepsilon}$ such that
	\[ h (t_0, x_0) - \varepsilon \leqslant \mathbb{E} \left[ h (\theta,
		\gamma_{\theta}^{t_0, x_0, v^{\varepsilon}}) - \int_{t_0}^{\theta} L (s,
		\gamma_s^{t_0, x_0, v^{\varepsilon}}, v^{\varepsilon}_s) d s \right] . \]
	Now, since $(\theta, \gamma_{\theta}^{t_0, x_0, v^{\varepsilon}}) \in
		\partial (\tilde{B}_{\eta_1, \eta_2} (t_0, x_0))$ and on this set
	\[ (h - \varphi) + \delta \leqslant (h - \varphi) (t_0, x_0), \]
	it follows that
	\[ - \varepsilon \leqslant \mathbb{E} \left[ \varphi (\theta,
			\gamma_{\theta}^{t_0, x_0, v^{\varepsilon}}) - \varphi (t_0, x_0) -
			\int_{t_0}^{\theta} L (s, \gamma_s^{t_0, x_0, v^{\varepsilon}},
			v^{\varepsilon}_s) d s \right] - \delta . \]
	As before, by definition of the martingale problem
	\[ \mathbb{E} \left[ \int_{t_0}^{\theta} (- g - v^{\varepsilon}_s \cdummy
			\nabla \varphi) (s, \gamma_s^{t_0, x_0, v^{\varepsilon}})+
			\int_{t_0}^{\theta} L (s, \gamma_s^{t_0, x_0, v^{\varepsilon}},
			v^{\varepsilon}_s) d s + \right] \leqslant - \delta / 2, \]
	which implies (remembering that $H (t, x, p) = \underset{v \in
			\mathbb{R}^d}{\sup} \{ v \cdummy p - L (t, x, v) \}$)
	\[ \mathbb{E} \left[ \int_{t_0}^{\theta} - g (s, \gamma_s^{t_0, x_0,
				v^{\varepsilon}}) - \int_{t_0}^{\theta} H (s, \gamma_s^{t_0, x_0,
				v^{\varepsilon}}, \nabla \varphi (s, \gamma_s)) d s + \right] \leqslant -
		\delta / 2 < 0, \]
	contradicting the fact that
	\[ 0 < - g (t, x)- H (t, x, \nabla \varphi (t, x)) \quad
		\tmop{for} \tmop{every} (t, x) \in \tilde{B}_{\eta_1, \eta_2} (t_0, x_0)
		. \]
\end{proof}

Another application of the previous results is solving, in the viscosity
sense, singular HJB equations with uniformly continuous and bounded initial conditions.
This is beyond reach for {\cite{zhang_singular_2022}}, where the initial condition is
much smoother, and for {\cite{KPZreloaded}}, where working with modified
parabolic spaces (that allow blow-ups at the origin) this smoothness can be
reduced to H{\"o}lder continuity of any positive index.

\begin{proposition}
	\label{prop:control-bounded-terminal}
		{{Under Assumptions \assC{}, \assP{}, and \assPContinuous{},}
	{let $\overline{h} \in \operatorname{BUC} (\mathbb{R}^d)$, with
	$\operatorname{BUC} (\mathbb{R}^d)$ denoting the space of bounded uniformly
	continuous functions,}}. Define
	\[ h (t, x) \assign \underset{v}{\sup} \mathbb{E} \left[ \overline{h}
			(\gamma^{t, x, v}_T)  - \int_t^T
			L (s, \gamma_s^{t, x, v}, v_s) d s \right] . \]
	Then $h$ is the unique bounded viscosity solution of \eqref{singularHJB}.

	\begin{proof}
		Consider two sequences of smooth bounded approximations
		$\left( \overline{h}^+_n \right)_{n \in \mathbb{N}}$ and $\left(
			\overline{h}^-_n \right)_{n \in \mathbb{N}}$ of $\overline{h}$ such that
		\[ \overline{h}^-_1 \leqslant \ldots \leqslant \overline{h}^-_n \leqslant
			\ldots \leqslant \overline{h} \leqslant \ldots \leqslant
			\overline{h}^+_n \leqslant \ldots \leqslant \overline{h}^+_1  \infixand
			\quad \left\| \overline{h}_n^+ - \overline{h}^-_n \right\|_{\infty}
			\overset{n \rightarrow \infty}{\rightarrow} 0, \]
		for example by putting $\overline{h}^+_n = \rho_{\varepsilon_n} \ast
			\left( \overline{h} + 1 / n \right)$ ($\overline{h}^-_n =
			\rho_{\varepsilon_n} \ast \left( \overline{h} - 1 / n \right)$) with small
			enough $\varepsilon_n > 0$, and define
		\[ h^{\pm}_n (t, x) \assign \underset{v}{\sup} \mathbb{E}  \left[
				\overline{h}^{\pm}_n (\gamma^{t, x, v}_T) - \int_t^T L (s, \gamma_s^{t, x, v}, v_s) d s \right] . \]
		{By Theorem~\ref{thm:quadratic-paracontrolled-global}, there
		exists a global paracontrolled solution with terminal condition
		$\overline{h}^{\pm}_n$. Proposition~\ref{verification} identifies this
		solution with the control value function $h^{\pm}_n$ defined above, and
		Theorem~\ref{paravicosity} shows that it is a viscosity solution.} Let us
		see that
		\begin{equation}
			h^-_1 \leqslant \ldots \leqslant h^-_n \leqslant \ldots \leqslant h
			\leqslant \ldots \leqslant h^+_n \leqslant \ldots \leqslant h^+_1
			\infixand \quad \| h_n^+ - h^-_n \|_{\infty} \overset{n \rightarrow
				\infty}{\rightarrow} 0. \label{valueconv}
		\end{equation}
		The first property is clear from the choice of initial conditions and the
		definition of the value functions; for the second, let $(t, x) \in [0, T]
			\times \mathbb{R}^d$ and $v^{t, x}_n \equiv v^n$ such that
		\[ h_n^+ (t, x) \leqslant \mathbb{E}  \left[ \overline{h}^+_n (\gamma^{t,
					x, v^n}_T) - \int_t^T L (s,
				\gamma_s^{t, x, v^n}, v^n_s) d s \right] + \frac{1}{n} . \]
		In that case
		\[ \begin{array}{lll}
				| h^+_n (t, x) - h_n^- (t, x) | & =         & h^+_n (t, x) - h_n^- (t, x)                                   \\
				                                & \leqslant & \mathbb{E}  \left[ \overline{h}^+_n (\gamma^{t, x,
						v^n}_T)  - \int_t^T L (s,
				\gamma_s^{t, x, v^n}, v^n_s) d s \right] + \frac{1}{n}                                                      \\
				                                &           & -\mathbb{E}  \left[ \overline{h}^-_n (\gamma^{t, x, v^n}_T) - \int_t^T L (s,
				\gamma_s^{t, x, v^n}, v^n_s) d s \right]                                                                    \\
				                                & =         & \mathbb{E} \left[ \overline{h}^+_n (\gamma^{t, x, v^n}_T) -
			\overline{h}^-_n (\gamma^{t, x, v^n}_T) \right] + \frac{1}{n}                                               \\
				                                & \leqslant & \left\| \overline{h}^+_n - \overline{h}^-_n
				\right\|_{\infty} + \frac{1}{n},
			\end{array} \]
		and the last term goes to zero when $n \rightarrow \infty$.

		From \eqref{valueconv} it follows that $h \in C_T C_b (\mathbb{R}^d)$
		(clearly also $h (T, \cdummy) = \overline{h}$); for the subsolution
		property, let $\varphi \in \mathcal{T}^b$ such that $h - \varphi$ has a
		strict local maximum at $(t_0, x_0) \in [0, T) \times \mathbb{R}^d$. By
		uniform convergence, there exists $\{ (t_n, x_n) \} \subset [0, T] \times
			\mathbb{R}^d$ such that
		\begin{enumerate}
			\item $(t_n, x_n) \overset{n \rightarrow \infty}{\rightarrow} (t_0,
				      x_0)$, so wlog we may assume $t_n < T$;

			\item $h^+_n - \varphi$ has a local maximum at $(t_n, x_n)$.
		\end{enumerate}
		We therefore have (the local maximum may not be strict anymore, but this
		is not problem; see Corollary \ref{touching} in the next Section) that
		\[ g (t_n, x_n) \geqslant - H (t_n, x_n, \nabla \varphi (t_n, x_n) ), \]
		and taking the limit $n \rightarrow \infty$
		\[ g (t_0, x_0) \geqslant - H (t_0, x_0, \nabla \varphi (t_0, x_0) ). \]
		The supersolution property is similar. Uniqueness follows from the fact
		that, for any other bounded solution $v$ with $v (T, \cdummy) \equiv
			\overline{h}$, we can use Proposition \ref{smoothcomparison} to conclude
		\[ h_n^- \leqslant v \leqslant h_n^+, \quad \tmop{for} \tmop{every} n \in
			\mathbb{N}, \]
		from which $v = h$.
	\end{proof}
\end{proposition}

\section{Zvonkin's transform}\label{sec::zvonkin}
\subsection{Change of variables and transformed equation}
We would like to find a change of variables \(\Phi\equiv(\Phi^1,\dots,\Phi^d)\) so that
\[ h (t, x) = v (t, \Phi (t, x)), \]
where $v$ solves a different PDE, more amenable to viscosity techniques.

Assume that all the functions involved are smooth enough, and let us derive
the equation. First, using Einstein's convention,
\[ \begin{array}{lll}
		\partial_t h (t, x) & = & (\partial_t v) (t, \Phi (t, x)) +
		\partial_{x_j} v (t, \Phi (t, x)) \partial_t \Phi^j (t, x)                                  \\
		                    & = & \left( - \frac{\Delta}{2} h - H (\cdummy, h, \nabla h) - b \nabla
		h \right) (t, x) .
	\end{array} \]
On the other hand,
\[ \begin{array}{lll}
		\partial_{x_i} h (t, x)     & = & \partial_{x_j} v (t, \Phi (t, x))
		\partial_{x_i} \Phi^j (t, x) \quad (i.e. \nabla h (t, x) = (\nabla
		\Phi)^T (t, x) \nabla v (t, \Phi (t, x))),                            \\
		\partial_{x_i x_i} h (t, x) & = & \partial_{x_k} \partial_{x_j} v (t,
		\Phi (t, x)) (\partial_{x_i} \Phi^k \partial_{x_i} \Phi^j) (t, x) +
		\partial_{x_j} v (t, \Phi (t, x)) \partial_{x_i x_i} \Phi^j .
	\end{array} \]
Putting everything together (remember that $\nabla \Phi$ is a $d \times d$
matrix and $\Delta \Phi \equiv (\Delta \Phi^1, \ldots, \Delta \Phi^d)$)
\[ \begin{array}{lll}
		(\partial_t v) (t, \Phi (t, x)) + \nabla v (t, \Phi (t, x)) \partial_t
		\Phi (t, x) & = & - \frac{(\nabla \Phi \nabla \Phi^T)}{2} (t, x) \colons
		\nabla^2 v (t, \Phi (t, x)) - \nabla v (t, \Phi (t, x)) \frac{\Delta}{2}
		\Phi (t, x)                                                                           \\
		            &   & - H (t, x, v (t, \Phi (t, x)), (\nabla \Phi (t, x))^T \nabla v (t,
		\Phi (t, x)))                                                                         \\
		            &   & - b (t, x) (\nabla \Phi (t, x))^T \nabla v (t, \Phi (t, x)) .
	\end{array} \]
Now, if our diffeomorphism $\Phi$ solves a system of the form
\[ \left( \partial_t + \frac{\Delta}{2} + b \nabla \right) \Phi^i (t, x) = -
	c^i (t, x), \quad i = 1, \ldots, d \]
for some continuous field $c$, we conclude that
\[ \begin{array}{lll}
		(\partial_t v) (t, \Phi (t, x)) + \frac{(\nabla \Phi \nabla \Phi^T)}{2}
		(t, x) \colons \nabla^2 v (t, \Phi (t, x)) & = & - H (t, x, v (t, \Phi
		(t, x)), (\nabla \Phi (t, x))^T \nabla v (t, \Phi (t, x)))                                        \\
		                                           &   & + c (t, x) \nabla v (t, \Phi (t, x)),
	\end{array} \]
and putting $x \equiv x (t) = \Phi^{- 1} (t, y)$
\[ \begin{array}{lll}
		(\partial_t v) (t, y) + \frac{(\nabla \Phi \nabla \Phi^T)}{2} (t, \Phi^{-
		1} (t, y)) \colons \nabla^2 v (t, y) & = & - H (t, \Phi^{- 1} (t, y), v
		(t, y), (\nabla \Phi (t, \Phi^{- 1} (t, y)))^T \nabla v (t, y))                                              \\
		                                     &   & + c (t, \Phi^{- 1} (t, y)) \nabla v (t, y) .
	\end{array} \]
Provided $\Phi$ is taken appropriately, there are no singular terms anymore. Our
choice will be
\[ \Phi (t, x) = x + u_{\lambda} (t, x), \]
where $u_{\lambda}$ solves the system
\begin{equation}
	\left( \partial_t + \frac{\Delta}{2} + b \nabla \right) u^i_{\lambda} = -
	b^i + \lambda u_{\lambda}^i, \quad u_{\lambda}^i (T, \cdummy) = 0, \quad i =
	1, \ldots, d. \label{linearZvonkin}
\end{equation}
The following result guarantees that we can fix $\lambda_0 > 0$ big enough
such that
\[ \underset{i, j = 1, \ldots, d}{\sup} \| \partial_{x_i} u^j_{\lambda_0}
	\|_{C ([0, T] \times \mathbb{R}^d ; \mathbb{R}^{d \times d})} \]
is as small as we want, so that $\Phi (t, \cdummy)$ is indeed a
diffeomorphism of $\mathbb{R}^d$ for each $t \in [0, T]$.

\begin{proposition}
	$\left( \cite{zhang_singular_2022}, \tmop{Lemma} 3.4 \right) .$\label{equl}
	Under \AssumptionA{}, let $4 / 3
		< \gamma < 2 - \alpha$ and $u_{\lambda}$ be as in \eqref{linearZvonkin}.
	There exists $\tilde{\lambda}  > 0$ such that for every $\lambda \geqslant
		\tilde{\lambda}$,
	\[ \| u_{\lambda} \|_{C_T \mathcal{C}^{\gamma}} + \| u_{\lambda}
		\|_{C^{\gamma / 2}_T {L^{\infty}} } \leqslant (1 \vee
		\lambda)^{\frac{\gamma + \alpha}{2} - 1} c (\alpha, \gamma, \| b \|_{C_T
		\mathcal{C}^{- \alpha}}, \| B \|_{C_T \mathcal{C}^{1 - 2 \alpha}}), \]
	where $(b, B) \in \mathcal{X}^{\alpha}$ is an enhanced pair (remember
	Definition \ref{enhanceddrift}).
\end{proposition}

Fix from now on $\lambda > 0$ such that $\underset{(t, x) \in [0, T] \times
		\mathbb{R}^d}{\sup} \| \nabla u_{\lambda} (t, x) \|_{\mathcal{L}
	(\mathbb{R}^d)} < \frac{1}{2} $; for our choice of $\Phi_{\lambda}$, the
transformed equation is
\begin{equation}
	\begin{array}{lll}
		\partial_t v + \frac{(\nabla \Phi \nabla \Phi^T)}{2} (t, \Phi^{- 1} (t,
		y)) \colons \nabla^2 v & = & - H (t, \Phi^{- 1} (t, y), v, (\nabla \Phi
		(t, \Phi^{- 1} (t, y)))^T \nabla v)                                                        \\
		                       &   & {-} \lambda u_{\lambda} (t, \Phi^{- 1} (t, y)) \nabla v .
	\end{array} \label{transPDE}
\end{equation}
Even in the rough case (i.e. \(b\) enhanced), its coefficients are H{\"o}lder continuous and the diffusion matrix is uniformly elliptic ($\nabla
	\Phi_{\lambda} = I_{d \times d} + \nabla u_{\lambda}$). This is a very favorable structure, which will allow us to establish, in the next Subsection, a priori Lipschitz bounds for viscosity solutions of \eqref{transPDE}. In order to transfer these back to \(h\), we need the following

\begin{theorem}
	\label{thm:uniqueness}
	Let $h$ be a viscosity solution of
	\eqref{singularHJB}. Then
	\[ v (t, x) \assign h (t, \Phi^{- 1} (t, x)) \]
	is a viscosity solution of \eqref{transPDE} (with the same terminal
	condition).

	\begin{proof}
		It is clear that $v \in C ([0, T] \times \mathbb{R}^d)$ and
		\[ v (T, \cdummy) = h (T, (\cdummy + u_{\lambda} (T, \cdummy))^{- 1})
			\overset{u_{\lambda} (T, \cdummy) = 0}{=} h (T, \cdummy) = \overline{h}
			. \]
		Let us check the subsolution property. Similarly to what was remarked
		right after Definition \ref{viscositydef}, for equation \eqref{transPDE} our
		notion of viscosity solution (i.e. taking test functions that solve linear
		versions of \eqref{transPDE}) coincides with the usual one. With this in
		mind, let $\varphi$ be such that
		\begin{equation}
			\left( \partial_t + \frac{(\nabla \Phi \nabla \Phi^T) \circ \Phi^{-
						1}}{2} \colons \nabla^2 \right) \varphi = g \circ \Phi^{- 1} - \lambda
			u_{\lambda} \circ \Phi^{- 1} \nabla \varphi, \label{testtrans1}
		\end{equation}
		and $(v - \varphi)$ has a strict local maximum, wlog equal to zero, at
		$(t_0, x_0) \in [0, T) \times \mathbb{R}^d$. We need to see that
		\begin{equation}
			\begin{array}{lll}
				g (t_0, \Phi^{- 1} (t_{0,} x_0)) & \geqslant & -H (t_0, \Phi^{- 1}
				(t_0, x_0), \varphi (t_0, x_0), (\nabla \Phi (t_0, \Phi^{- 1} (t_0,
				x_0)))^T \nabla \varphi (t_0, x_0))         ,
			\end{array} \label{uniqineq}
		\end{equation}
		If we define $\psi
			(t, x) \assign \varphi (t, \Phi (t, x))$ and $(t_0, y_0) \assign (t_0,
			\Phi^{- 1} (t_{0,} x_0))$, inequality \eqref{uniqineq} is equivalent to
		\begin{equation}
			\begin{array}{lll}
				g (t_0, y_0) & \geqslant & -H (t_0, y_0, \psi (t_0, y_0), \nabla \psi
				(t_0, y_0)).
			\end{array} \label{uniqineq2}
		\end{equation}
		Since $(h - \psi)$ has a strict local maximum at $(t_0, y_0)$ ($\Phi (t,
			\cdummy)$ is a diffeomorphism for each $t$), by definition of viscosity
		solution \ref{uniqineq} is guaranteed if we can prove that $\psi \in
			\mathcal{T}_b$ and
		\begin{equation}
			\left( \partial_t + \frac{\Delta}{2} + b \nabla \right) \psi = g.
			\label{testtrans2}
		\end{equation}
		In order to do that, consider a smooth approximation of the enhanced
		drift, i.e. $b^n, B^n \in C_T \mathcal{C}^{\infty}$ satisfying
		\[ \begin{array}{ll}
				b_n \rightarrow b, & \tmop{in} C_T \mathcal{C}^{- \alpha},      \\
				\mathcal{J}^T (\partial_{x_j} b_n^i) \odot b_n^j \rightarrow B_{i,
				j},                & \tmop{in} C_T \mathcal{C}^{1 - 2 \alpha} .
			\end{array} \]
		If we let $u^n_{\lambda}$ solve
		\[ \left( \partial_t + \frac{\Delta}{2} + b^n \nabla \right) u^n_{\lambda}
			= - b^n + \lambda u^n_{\lambda}, \quad u_{\lambda}^n (T, \cdummy) = 0,
		\]
		we can make analogous definitions for $\varphi_n$ and $\Phi_n$ so that
		\[ \left( \partial_t + \frac{\left( \nabla \Phi^n {\nabla \Phi^n}^T
				\right) \circ (\Phi^n)^{- 1}}{2} \colons \nabla^2 \right) \varphi^n = g
			\circ (\Phi^n)^{- 1} - \lambda u^n_{\lambda} \circ (\Phi^n)^{- 1}
			\nabla \varphi^n . \]
		By direct computation (like at the beginning of the section, since all the
		terms involved are smooth) one can check that $\psi_n (t, x) \assign
			\varphi_n (t, \Phi_n (t, x))$ solves
		\[ \left( \partial_t + \frac{\Delta}{2} + b^n \nabla \right) \psi^n = g.
		\]
		We can apply continuity with respect to the enhanced drift (see Theorem
		\ref{solmap} in the Appendix) to conclude that $\psi_n$ converges in $C_T
			\mathcal{C}^{(2 - \alpha) -}$ to the unique paracontrolled solution of
		\eqref{testtrans2}. On the other hand, by continuity wrt the coefficients
		of \eqref{testtrans1} and Proposition
		\ref{equl},
		\[ \varphi^n (t, \Phi^n (t, x)) \rightarrow \varphi (t, \Phi (t, x)) =
			\psi (t, x), \]
		i.e. $\psi$ is the unique solution of \eqref{testtrans2}, as we wanted.
	\end{proof}
\end{theorem}

The crucial step in the previous proof was checking that
\[ \varphi (\cdummy, \cdummy) \in C^{1, 2}_b ([0, T] \times \mathbb{R}^d)
	\Longrightarrow \varphi (\cdummy, \Phi (\cdummy, \cdummy)) \in
	\mathcal{T}^b. \]
We can deduce other interesting facts from this.

\begin{corollary}
	\label{touching}\tmtextbf{(Touching by test functions)}
	For any $(t_0, x_0)
		\in [0, T) \times \mathbb{R}^d$, there is a test function $\psi \in
		\mathcal{T}^b$ such that
	\begin{enumerate}
		\item $\psi (t, x) > 0 = \psi (t_0, x_0), \quad \tmop{for} \tmop{every}
			      (t, x) \in [0, T] \times \mathbb{R}^d \backslash \{ (t_0, x_0) \}$.

		\item $g (t_0, x_0) = 0$, where $g \in C_T \mathcal{C}_b$ is the ``right-hand side'' of $\psi$, i.e.
		      \[ \left( \partial_t + \frac{\Delta}{2} + b \nabla \right) \psi = g
			      \text{ in } [0, T) \times \mathbb{R}^d . \]
	\end{enumerate}
	In particular, admitting non-strict local extrema in Definition
	\ref{viscositydef} yields an equivalent definition.

	\begin{proof}
		If we let $(t_0, y_0) \assign (t_0, \Phi (t_0, x_0))$, it is clear
		(choosing for example $| t - t_0 |^2 + | x - \Phi (t_0, x_0) |^4$ locally
		near that point) that there exists $\varphi \in C_b^{1, 2} ([0, T] \times
			\mathbb{R}^d)$ satisfying
		\begin{enumerate}
			\item $\varphi (t, y) > 0 = \varphi (t_0, y_0)$ for every $\tmop{for}
				      \tmop{every} (t, y) \in [0, T] \times \mathbb{R}^d \backslash \{ (t_0,
				      y_0) \}$,

			\item $\partial_t \varphi (t_0, y_0) = \partial_i \varphi (t_0, y_0) =
				      \partial_{i j} \varphi (t_0, y_0) = 0$ for each $i, j = 1, \ldots, d$.
		\end{enumerate}
		In that case, $\psi (t, x) \assign \varphi (t, \Phi (t, x))$ satisfies the
		desired conditions, since $g$ is given precisely by
		\[ g (\cdummy, \cdummy) = \left[ \left( \partial_t + \frac{(\nabla \Phi
					\nabla \Phi^T) \circ \Phi^{- 1}}{2} \colons \nabla^2 + \lambda
				u_{\lambda} \circ \Phi^{- 1} \nabla \right) \varphi \right] (\cdummy,
			\Phi (\cdummy, \cdummy)) \in {C_T \mathcal{C}_b (\mathbb{R}^d)},
		\]
		and therefore $g (t_0, x_0) = 0$,

		The last assertion follows from the fact that if $\xi \in \mathcal{T}^b$
		and $h - \xi$ has a local maximum (resp. min.) at $(t_0, x_0)$, then $h -
			(\xi + \psi)$ has a strict local maximum at that point (resp $h - (\xi -
			\psi)$ strict loc. min.), while the viscosity inequality \eqref{viscineq}
		``does not see'' $\psi$.
	\end{proof}
\end{corollary}

\subsubsection{A priori Lipschitz estimates}
{
Under \AssumptionP{}, it is easy to see that the transformed Hamiltonian
\begin{align*}
  \hat H(t,y,r,p) := H&\bigl(t,\Phi^{-1}(t,y),r,
  (\nabla\Phi)^T(t,\Phi^{-1}(t,y))p\bigr)\\
  &+\lambda u_\lambda(t,\Phi^{-1}(t,y))\cdot p
\end{align*}
satisfies
\begin{equation}
  \begin{aligned}
    |\hat H(t,x,r,p)-\hat H(t,y,r,p)|
    &\leq C |x-y|^{\beta'}(1+|p|^\kappa)+C,\\
    |\hat H(t,y,r,p)-\hat H(t,y,r,p')|
    &\leq C\bigl(1+(|p|\vee|p'|)^{\kappa-1}\bigr)|p-p'|,\\
    |\hat H(t,y,r,p)-\hat H(t,y,r',p)|
    &\leq C(1+|p|^q)|r-r'|.
  \end{aligned}
  \label{eq:apriori-hhat}
\end{equation}
where $\beta'=(\gamma-1)\wedge\beta$, and $\gamma\in(4/3,2-\alpha)$ is such that
$\Phi\in C^\gamma$ and \(\kappa\in[1,2+(\gamma-1)\wedge\beta)\).
}

Fix
$0<\nu<\alpha\wedge\beta'\wedge(2-q)\wedge(2+\beta'-\kappa)$ and let
$\omega:\mathbb{R}_+\to\mathbb{R}_+$
be such that
\begin{enumerate}
  \item For $r\in[0,1)$
  \[
    \omega(r)=r-\frac{1}{2(\nu+1)}r^{\nu+1};
  \]

  \item $\omega$ is linear outside of $[0,1)$ (with the appropriate slope).
\end{enumerate}
In particular, near zero
\[
  \omega'(r)=1-\frac{1}{2}r^\nu,
  \qquad
  \omega''(r)=-\frac{\nu}{2}r^{\nu-1},
\]
and globally
\[
  cr\leq\omega(r)\leq Cr,
\]
for positive constants \(c,C>0\). When applying Ishii's Lemma (Theorem 8.3 in \cite{CIL1992}), the uniform ellipticity of the diffusion matrix in the transformed equation will be crucial for exploiting the negative sign arising from \(\omega''\). A similar mechanism can be found in \cite{LeyNguyen2017} and \cite{PorrettaPriola2013}: the results in the latter are quite general, but for completeness we include a simpler proof, better fitted to our assumptions on the transformed Hamiltonian.

In the following Theorem, it is important to assume that the viscosity solution is bounded, as otherwise there is no guarantee that the doubled variable function achieves a maximum (unless we work in \(\mathbb{T}^d\) instead of \(\mathbb{R}^d\)).
{In the control setting, this boundedness follows from
Corollary \ref{maxpple}; in the general setting considered here, it is imposed
directly.}

In what follows, we denote \(a(t,x):=(\nabla\Phi\nabla\Phi^T)\circ\Phi^{-1}(t,x)\)
\begin{theorem}\label{thm:viscosity-lipschitz-estimate}
{Under \AssumptionP{}, let} $v$ be a bounded viscosity solution of
\eqref{transPDE} with Lipschitz terminal data \(\overline{v}\). Then,
for $K$ large enough and any $x,y\in\mathbb{R}^d$, $t\in(0,T]$,
\[
  v(t,x)-v(t,y)\leq \{K\omega(|x-y|)\}\wedge2\|v\|_\infty;
\]
In particular (exchanging $x$ and $y$),
\[
  |v(t,x)-v(t,y)|\leq C(|x-y|\wedge\|v\|_\infty),
\]
where \(C>0\) depends only on the data in Assumptions \assA{} and \assP{},
\(\|v\|_\infty\), and the Lipschitz constant of
\(\overline{v}\).
\end{theorem}

\begin{proof}
Let $\mu\in(0,1)$ (to be taken $\to0$ later), and consider the doubled
variable function
\[
  \Psi(t,x,y):=v(t,x)-v(t,y)-K\omega(|x-y|)
  -\frac{\mu}{2}(|x|^2+|y|^2).
\]
Suppose that $\sup\Psi>0$; since $v$ is continuous and bounded, the supremum
is attained at some point $(t^*,x^*,y^*)$. By taking $K$ big enough, we may
assume (since $\bar v=v(T,\cdot)$ is Lipschitz) that
\[
  |x^*-y^*|\leq1,\qquad t^*\in[0,T),\qquad |x^*|\vee|y^*|\leq M_\mu.
\]
Letting $r^*:=|x^*-y^*|$, we have that
\[
  cKr^*\leq K\omega(r^*)+\frac{\mu}{2}(|x^*|^2+|y^*|^2)<2\|v\|_\infty,
\]
so $r^*\leq\frac{2\|v\|_\infty}{cK}$, and since the supremum is strictly
positive, $r>0$. Letting
\[
  p:=K\omega'(r^*)\frac{x^*-y^*}{r^*},
\]
we can apply Ishii's Lemma (Theorem 8.3 in \cite{CIL1992}) to conclude that there exist
$b_1,b_2\in\mathbb{R}$ and \(d\times d\) symmetric matrices $X,Y$ such that
\[
  b_1-b_2=0=\partial_t\left(K\omega(|x^*-y^*|)
+\frac{\mu}{2}(|x^*|^2+|y^*|^2)\right)
\]
and
\begin{equation}
  {
  \begin{aligned}
    b_1+\frac{1}{2}a(t^*,x^*)::X
    +\hat H(t^*,x^*,v(t^*,x^*),p+\mu x^*)&\geq0,\\
    b_2+\frac{1}{2}a(t^*,y^*)::Y
    +\hat H(t^*,y^*,v(t^*,y^*),p-\mu y^*)&\leq0.
  \end{aligned}
  }
  \label{eq:apriori-jets}
\end{equation}
Furthermore, $X$ and $Y$ satisfy the usual matrix identities (see below).
Subtracting the equations in \eqref{eq:apriori-jets},
\begin{equation}
  {
  \frac{1}{2}\bigl[a(t^*,x^*)::X-a(t^*,y^*)::Y\bigr]
  +\bigl[\hat H(t^*,x^*,v(t^*,x^*),p+\mu x^*)
  -\hat H(t^*,y^*,v(t^*,y^*),p-\mu y^*)\bigr]\geq0.
  }
  \label{eq:apriori-subtracted}
\end{equation}
We will fix $K$ large enough contradicting this inequality.

\noindent\textbf{Step 1:}
$\frac{1}{2}\bigl[a(t^*,x^*)::X-a(t^*,y^*)::Y\bigr]
\leq-\widetilde C K(r^*)^{\nu-1}$, for some $\widetilde C>0$ independent of
$K$.

Outside of the diagonal, we can compute the (two variable) Hessian
\[
  D^2_{x,y}\left(K\omega(|x-y|)+\frac{\mu}{2}(|x|^2+|y|^2)\right)
  =
  \begin{pmatrix}
    A+\mu I & -A\\
    -A & A+\mu I
  \end{pmatrix}
  =:\mathcal{A},
\]
where
\[
  A=K\left[
    \omega''(r)\frac{x-y}{r}\otimes\frac{x-y}{r}
    +\frac{\omega'(r)}{r}\left(
      I-\frac{x-y}{r}\otimes\frac{x-y}{r}
    \right)
  \right].
\]
From the Ishii Lemma, we know that for any $\varepsilon>0$ (to be fixed
later), $X$ and $Y$ can be chosen to satisfy
\begin{equation}
  \begin{pmatrix}
    X&0\\
    0&-Y
  \end{pmatrix}
  \leq\mathcal{A}+\varepsilon\mathcal{A}^2.
  \label{eq:apriori-ishii-matrix}
\end{equation}
Remembering that
\[
  a=(I+\nabla u_\lambda)(I+\nabla u_\lambda)^t\circ\Phi^{-1},
\]
we can split (taking larger \(\lambda\) if necessary)
\[
  a(t,x)=\frac{1}{2}I+\widetilde a(t,x),
  \qquad\text{where}\qquad
  \widetilde a(t,x)\geq\frac{1}{4}I,
\]
and the term we need to control becomes
\begin{equation}
  \frac{1}{4}\operatorname{tr}(X-Y)
  +\frac{1}{2}\bigl[\widetilde a(t^*,x^*)::X
  -\widetilde a(t^*,y^*)::Y\bigr].
  \label{eq:apriori-diffusion-split}
\end{equation}
From \eqref{eq:apriori-ishii-matrix}, one can see that $X-Y\leq 2(\varepsilon\mu^2+\mu)I$, and in
particular,
\[
  \begin{aligned}
  \operatorname{tr}(X-Y)
  &\leq
  \left(\frac{x-y}{r}\right)(X-Y)
  \left(\frac{x-y}{r}\right)^t+2(d-1)(\varepsilon\mu^2+\mu)\\
  &=
  \left(\frac{x-y}{r},\frac{y-x}{r}\right)
  \begin{pmatrix}
    X&0\\
    0&-Y
  \end{pmatrix}
  \left(\frac{x-y}{r},\frac{y-x}{r}\right)^t+2(d-1)(\varepsilon\mu^2+\mu).
  \end{aligned}
\]
On the other hand (using that
$\frac{x-y}{r}\otimes\frac{x-y}{r}$ is the projection onto
$\operatorname{span}\{\frac{x-y}{r}\}$) one can check that
\[
  \begin{aligned}
  &\left(\frac{x-y}{r},\frac{y-x}{r}\right)
  \mathcal{A}
  \left(\frac{x-y}{r},\frac{y-x}{r}\right)^t\\
  &\qquad=
  4\left(\frac{x-y}{r}\right)A
  \left(\frac{x-y}{r}\right)^t+2\mu
  =4K\omega''(r)+2\mu,
  \end{aligned}
\]
and
\[
  \begin{aligned}
  &\left(\frac{x-y}{r},\frac{y-x}{r}\right)
  \mathcal{A}^2
  \left(\frac{x-y}{r},\frac{y-x}{r}\right)^t\\
  &\qquad=
  8\left(\frac{x-y}{r}\right)A^2
  \left(\frac{x-y}{r}\right)^t
  +8\mu\left(\frac{x-y}{r}\right)A
  \left(\frac{x-y}{r}\right)^t+2\mu^2\\
  &\qquad=
  8(K\omega''(r))^2+8\mu K\omega''(r)+2\mu^2.
  \end{aligned}
\]
By \eqref{eq:apriori-ishii-matrix} and the previous identities,
\begin{align*}
  \frac{1}{4}\operatorname{tr}(X-Y)
  \leq&\frac{1}{4}\{
    4K\omega''(r)+2\mu
    +\varepsilon\left[8(K\omega''(r))^2
    +8\mu K\omega''(r)+2\mu^2\right]\\&+2(d-1)(\frac{r}{K}\mu^2+\mu)\},
\end{align*}
so choosing $\varepsilon=\frac{r}{K}$ and plugging
$\omega''(r)=-\frac{\nu}{2}r^{\nu-1}$ yields
\begin{equation}
  \begin{aligned}
  \frac{1}{4}\operatorname{tr}(X-Y)
  \leq\frac{1}{4}\biggl(
    &-2\nu Kr^{\nu-1}+2\mu
    +2\nu^2Kr^{2\nu-1}\\
    &-4\mu\nu r^\nu
    +2\frac{\mu^2}{K}r+2(d-1)(\frac{r}{K}\mu^2+\mu)
  \biggr).
  \end{aligned}
  \label{eq:apriori-trace}
\end{equation}
The right hand side of which is dominated by the negative term
$-\frac{1}{2}\nu Kr^{\nu-1}$. For the second term in
\eqref{eq:apriori-diffusion-split}, denote
\[
  \tilde{\sigma}(\cdot,\cdot):=\widetilde a^{1/2}(\cdot,\cdot);
\]
for the square root matrix. Since $\widetilde a\geq\frac{1}{4}I$, the operator norms satisfy (see Lemma 2.2 in \cite{Sch92})
\[
  \|\tilde{\sigma}(t,x)-\tilde{\sigma}(t,y)\|_{\operatorname{op}}
  \leq C\|\widetilde a(t,x)-\widetilde a(t,y)\|_{\operatorname{op}}
  \leq C'|x-y|^{\gamma-1}.
\]
We have that
\[
  \begin{aligned}
  &\frac{1}{2}\bigl[\widetilde a(t^*,x^*)::X
  -\widetilde a(t^*,y^*)::Y\bigr]\\
  &\qquad=
  \frac{1}{2}\sum_j
  \left(\widetilde\sigma_j(t^*,x^*),
        \widetilde\sigma_j(t^*,y^*)\right)
  \begin{pmatrix}
    X&0\\
    0&-Y
  \end{pmatrix}
  \begin{pmatrix}
    \widetilde\sigma_j(t^*,x^*)\\
    \widetilde\sigma_j(t^*,y^*)
  \end{pmatrix},
  \end{aligned}
\]
where $\widetilde\sigma_j$ is the $j$-th column of $\widetilde\sigma$, so
again using \eqref{eq:apriori-ishii-matrix} and writing
\[
  \mathcal{A}=
  \begin{pmatrix}
    A&-A\\
    -A&A
  \end{pmatrix}
  +\mu
  \begin{pmatrix}
    I&0\\
    0&I
  \end{pmatrix},
\]
\[
  \begin{aligned}
  &\frac{1}{2}\bigl[\widetilde a(t^*,x^*)::X
  -\widetilde a(t^*,y^*)::Y\bigr]\\
  &\quad\leq
  \frac{1}{2}\sum_j
  \left(\widetilde\sigma_j(t^*,x^*)
        -\widetilde\sigma_j(t^*,y^*)\right)
  (A+2\varepsilon A^2)
  \left(\widetilde\sigma_j(t^*,x^*)
        -\widetilde\sigma_j(t^*,y^*)\right)^t\\
  &\qquad
  +\varepsilon\mu\sum_j
  \left(\widetilde\sigma_j(t^*,x^*)
        -\widetilde\sigma_j(t^*,y^*)\right)
  A
  \left(\widetilde\sigma_j(t^*,x^*)
        -\widetilde\sigma_j(t^*,y^*)\right)^t\\
  &\qquad
  +(\mu+\varepsilon\mu^2)\frac{1}{2}\sum_j
  \left(
    |\widetilde\sigma_j(t^*,x^*)|^2
    +|\widetilde\sigma_j(t^*,y^*)|^2
  \right).
  \end{aligned}
\]
As $\mu\in(0,1)$ the last term is uniformly bounded. Since
\[
  A=K\left[
    \omega''(r)\frac{x-y}{r}\otimes\frac{x-y}{r}
    +\frac{\omega'(r)}{r}\left(
      I-\frac{x-y}{r}\otimes\frac{x-y}{r}
    \right)
  \right]
\]
is a sum of projection matrices and $\omega''\leq0$,
\begin{equation}
  \begin{aligned}
  &\frac{1+2\varepsilon\mu}{2}\sum_j
  \left(\widetilde\sigma_j(t^*,x^*)
        -\widetilde\sigma_j(t^*,y^*)\right)
  A
  \left(\widetilde\sigma_j(t^*,x^*)
        -\widetilde\sigma_j(t^*,y^*)\right)^t\\
  &\qquad\leq
  (1+2\varepsilon\mu)\sum_j K\frac{\omega'(r)}{r}
  |\widetilde\sigma_j(t^*,x^*)
   -\widetilde\sigma_j(t^*,y^*)|^2\\
  &\qquad\leq CK(1+2\varepsilon\mu)r^{2\gamma-3},
  \end{aligned}
  \label{eq:apriori-A}
\end{equation}
where we used that $\widetilde\sigma\in C^{\gamma-1}$ and $\omega'\leq1$.
On the other hand, $\|A\|_{\mathrm{op}}\leq C\frac{K}{r}$ and therefore
\begin{equation}
  \begin{aligned}
  &\frac{\varepsilon}{2}\sum_j
  \left(\widetilde\sigma_j(t^*,x^*)
        -\widetilde\sigma_j(t^*,y^*)\right)
  A^2
  \left(\widetilde\sigma_j(t^*,x^*)
        -\widetilde\sigma_j(t^*,y^*)\right)^t\\
  &\qquad\leq
  \frac{\varepsilon}{2}\sum_j\|A\|_{\mathrm{op}}^2
  |\widetilde\sigma_j(t^*,x^*)
   -\widetilde\sigma_j(t^*,y^*)|^2\\
  &\qquad\leq C\varepsilon K^2r^{2\gamma-4}\\
  &\qquad\leq CKr^{2\gamma-3},
  \end{aligned}
  \label{eq:apriori-A2}
\end{equation}
where in the last line we chose again
$\varepsilon=\frac{r}{K}$. Combining
\eqref{eq:apriori-diffusion-split}, \eqref{eq:apriori-trace},
\eqref{eq:apriori-A} and \eqref{eq:apriori-A2}, we get
\[
  \begin{aligned}
  \frac{1}{2}\bigl[a(t^*,x^*)::X-a(t^*,y^*)::Y\bigr]
  \leq
  &\frac{1}{4}\left(
    -2\nu Kr^{\nu-1}+2\mu
    +2\nu^2Kr^{2\nu-1}
    -4\mu\nu r^\nu
    +2\frac{\mu^2}{K}r
  \right)
  \\&+C(Kr^{2\gamma-3}+r^{2\gamma-2}+1).
  \end{aligned}
\]
Since
\[
  \left.
  \begin{array}{c}
    0<\nu<\alpha,\qquad \alpha\in(1/2,2/3)\\
    4/3<\gamma<2-\alpha
  \end{array}
  \right\}
  \Longrightarrow \nu-1<2\gamma-3,
\]
we get that for big enough \(K\) (remember that
$r^*\leq\frac{2\|v\|_\infty}{cK}$), the negative term
$-\frac{1}{4}\nu Kr^{\nu-1}$ dominates and therefore
\[
  \frac{1}{2}\bigl[a(t^*,x^*)::X-a(t^*,y^*)::Y\bigr]
  \leq-\widetilde C K(r^*)^{\nu-1},
\]
for some $\widetilde C>0$ independent of $K$ (and uniformly in \(\mu\in(0,1)\)).
\end{proof}

\noindent\textbf{Step 2:}
For big enough $K$, we show that
\begin{align*}
\hat H(t^*,x^*,v(t^*,x^*),p+\mu x^*)
-\hat H(t^*,y^*,v(t^*,y^*),p-\mu y^*) \leq\frac{\widetilde C}{2}K(r^*)^{\nu-1},
\end{align*}
where $\widetilde C>0$ is as in the previous step.

Write $v_x^*:=v(t^*,x^*)$ and $v_y^*:=v(t^*,y^*)$. By
\eqref{eq:apriori-hhat}, since
$p=K\omega'(r^*)\frac{x^*-y^*}{r^*}$ and
$\mu(|x^*|^2+|y^*|^2)<2\|v\|_\infty$,
\begin{align*}
&\hat H(t^*,x^*,v_x^*,p+\mu x^*)
-\hat H(t^*,y^*,v_y^*,p-\mu y^*)\\
&=\hat H(t^*,x^*,v_x^*,p+\mu x^*)
-\hat H(t^*,y^*,v_x^*,p+\mu x^*)\\
&\quad{+\hat H(t^*,y^*,v_x^*,p+\mu x^*)
-\hat H(t^*,y^*,v_y^*,p+\mu x^*)}\\
&\quad+\hat H(t^*,y^*,v_y^*,p+\mu x^*)
-\hat H(t^*,y^*,v_y^*,p-\mu y^*)\\
&\leq C(r^*)^{\beta'}(1+|p+\mu x^*|^\kappa)
{+C(1+|p+\mu x^*|^q)|v_x^*-v_y^*|}\\
&\quad+C(1+|p+\theta\mu x^*-(1-\theta)\mu y^*|^{\kappa-1})
\mu|x^*+y^*|\\
&\leq C(r^*)^{1+\beta'-\nu}(r^*)^{\nu-1}
\left\{1+(2\mu\|v\|_\infty)^{\frac{\kappa}{2}}+K^\kappa\right\}\\
&\quad+C\left(1+K^{\kappa-1}
+(2\mu\|v\|_\infty)^{\frac{\kappa-1}{2}}\right)
{+C(1+K^q)}\\
&\leq C\left(\frac{2\|v\|_\infty}{cK}\right)^{1+\beta'-\nu}
\left(1+(2\mu\|v\|_\infty)^{\frac{\kappa}{2}}+K^\kappa\right)
(r^*)^{\nu-1}\\
&\quad+o\bigl(K(r^*)^{\nu-1}\bigr){+C(1+K^q)}\\
&=C'K^{(\kappa-1)-(1+\beta'-\nu)}K(r^*)^{\nu-1}
+o\bigl(K(r^*)^{\nu-1}\bigr).
\end{align*}
{Indeed, $r^*\leq2\|v\|_\infty/(cK)$ implies
$K(r^*)^{\nu-1}\gtrsim K^{2-\nu}$, so
$1+K^q=o(K(r^*)^{\nu-1})$ by $q<2-\nu$.} The result now follows by taking
$K$ big enough and choosing $\nu$ also so that
\begin{align*}
\nu<2+\beta'-\kappa.
\end{align*}

Combining Steps 1 and 2, and by \eqref{eq:apriori-subtracted},
\[
  {
  \begin{aligned}
  0\leq{}&\frac{1}{2}\bigl[a(t^*,x^*)::X-a(t^*,y^*)::Y\bigr]\\
  &+\bigl[\hat H(t^*,x^*,v(t^*,x^*),p+\mu x^*)
  -\hat H(t^*,y^*,v(t^*,y^*),p-\mu y^*)\bigr]
  \leq-\frac{\widetilde C}{2}K(r^*)^{\nu-1},
  \end{aligned}
  }
\]
which is a contradiction. It follows that
\[
  v(t,x)-v(t,y)\leq K\omega(|x-y|)
  +\frac{\mu}{2}(|x|^2+|y|^2),
\]
where $K$ is big enough and does not depend on $\mu$. Then, taking
$\mu\to0$,
\[
  v(t,x)-v(t,y)\leq K\omega(|x-y|),
\]
as we wanted.
\begin{remark}
By taking the time dependent modulus of continuity
\[
  \frac{K}{(T-t)^{\frac{1}{2-\nu}}}\omega(|x-y|),
\]
one can prove the previous theorem without assuming that the terminal
condition is Lipschitz, yielding the interior Lipschitz estimate
\[
  |v(t,x)-v(t,y)|\leq
  \frac{K}{(T-t)^{\frac{1}{2-\nu}}}\omega(|x-y|),
\]
for any viscosity solution. Let us justify the exponent: repeating the proof
above would add the extra time derivative term in
\eqref{eq:apriori-subtracted}
\[
  \frac{d}{dt}\left(
    \frac{K}{(T-t)^\eta}\omega(|x-y|)
  \right)
  =\frac{\eta}{(T-t)}\frac{K}{(T-t)^\eta}\omega(|x-y|).
\]
At the maximum point,
\[
  \frac{K}{(T-t^*)^\eta}\omega(|x^*-y^*|)\leq2\|v\|_\infty,
\]
so
\begin{equation}
  \left.
  \frac{d}{dt}\left(
    \frac{K}{(T-t)^\eta}\omega(|x-y|)
  \right)
  \right|_{(t^*,x^*,y^*)}
  \leq\frac{2\eta\|v\|_\infty}{(T-t)}.
  \label{eq:apriori-time-derivative}
\end{equation}
On the other hand, since
$r^*\leq\frac{4\|v\|_\infty}{K}(T-t)^\eta$ and $\nu<1$,
\[
  \begin{aligned}
  \frac{K}{(T-t)^\eta}(r^*)^{\nu-1}
  &\geq\frac{K}{(T-t)^\eta}
  \left(\frac{4\|v\|_\infty}{K}(T-t)^\eta\right)^{\nu-1}\\
  &=(4\|v\|_\infty)^{\nu-1}K^{2-\nu}
  \left(\frac{1}{T-t}\right)^{\eta(2-\nu)}.
  \end{aligned}
\]
We can choose $\eta=\frac{1}{2-\nu}$, which together with \eqref{eq:apriori-time-derivative} yields
\begin{equation*}
  \left.
  \frac{d}{dt}\left(
    \frac{K}{(T-t)^\eta}\omega(|x-y|)
  \right)
  \right|_{(t^*,x^*,y^*)}
  {\leq C_\nu\|v\|_\infty^{2-\nu}K^{\nu-2}
  \frac{K}{(T-t)^\eta}(r^*)^{\nu-1}.}
\end{equation*}
{Thus, by taking $K$ sufficiently large (depending
on $\|v\|_\infty$) this term can be absorbed into the one coming from the
second derivative of the concave term, and we conclude as before.}
\end{remark}

\begin{corollary}\label{cor:singular-viscosity-lipschitz-estimate}
{Under \AssumptionP{}, let \(h\) be a bounded viscosity
solution of \eqref{singularHJB}.} If \(\overline{h}\) is Lipschitz continuous,
\[
  {\|h\|_{L_T^\infty W^{1,\infty}}
  {\leq C}};
\]
otherwise, the interior estimate
\[
{\|\nabla h(t,\cdot)\|_{L^\infty}
{\leqslant \frac{C}
{(T-t)^\frac{1}{2-\nu}}}}
\]
holds. {In both cases, $C>0$ depends only on the data in
Assumptions \assA{} and \assP{} and on $\|h\|_\infty$ (in the first
case, also on $\|\overline h\|_{\mathrm{Lip}}$).}
\end{corollary}
\begin{proof}
    Follows from Theorem \ref{thm:uniqueness} and the previous estimates, since
    \({\Phi}\) is globally Lipschitz in \(x\).
\end{proof}

\section{Singular quadratic FBSDEs}
\label{sec:BSDE}
\label{section:singular-FBSDE}

We study the singular forward--backward system, formally written as
\begin{align*}
  X_t^{s,x}
  &=x+\int_s^t b(r,X_r^{s,x})\,dr+W_t-W_s,\\
  Y_t^{s,x}
  &=\bar h(X_T^{s,x})
    +\int_t^T H(r,X_r^{s,x},Y_r^{s,x},Z_r^{s,x})\,dr
    -\int_t^T Z_r^{s,x}\,dW_r,
  \qquad s\leq t\leq T.
\end{align*}
The forward drift $b$ is distribution-valued, while the backward generator
$H$ may have quadratic growth in $z$; we refer to this system as a singular
quadratic FBSDE. We begin with the singular forward equation. The
martingale-problem formulation of Definition~\ref{def:martingale-problem}
specifies a probability law on the canonical path space, but does not
intrinsically provide a Brownian motion on an underlying filtered probability
space. Since the backward equation must be driven by the same Brownian motion
as the forward equation, we work instead with the notion of rough weak
solution introduced in {\cite{kremp_rough_2025}}. Such a solution is given by
a tuple
\[
  (\Omega^{s,x},\mathcal F^{s,x},
  (\mathcal F_t^{s,x})_{t\in[s,T]},\mathbb P^{s,x},
  X^{s,x},G^{s,x},W^{s,x}),
\]
where $X^{s,x}$ is the forward process, $W^{s,x}$ is a Brownian motion, and
$G^{s,x}$ heuristically represents the singular drift integral
$\int_s^\cdot b(r,X_r^{s,x})\,dr$; all these processes are adapted to the
same filtration. In Subsection~\ref{subsec:singular-forward-equation}, we
recall the relevant results on rough weak solutions and their relation to the
martingale formulation, and establish an It{\^o} formula for such solutions in
Proposition~\ref{prop:ito-formula}.

Consequently, the singular FBSDE must itself be understood in a weak sense.
In Subsection~\ref{subsec:FBSDE-weak-wellposedness}, we introduce the
corresponding notion of weak solution and establish well-posedness under
relatively weak assumptions.

Fix $(s,x)$ and a rough weak solution as above. At negative regularity,
pathwise uniqueness may fail {\cite{hess_childs_rowan_2026}}, so the filtration
of a weak solution need not be generated by $W^{s,x}$. Hence the classical
martingale representation theorem for Brownian filtrations is not directly
available. However, {\cite[Proposition~2.8]{perkowski_coming_2025}} shows that
every square-integrable martingale on the completed, right-continuous natural
filtration $(\mathcal F_t^{X^{s,x}})_{t\in[s,T]}$ admits a representation with
respect to $W^{s,x}$. This is sufficient for the BSDE arguments below. We
therefore restrict throughout this section to rough weak solutions equipped
with this filtration. Such rough weak solutions exist by
Theorem~\ref{thm:weak-martingale-equivalence} and it may be viewed as the
canonical choice, since it is precisely the rough weak realization obtained
from the canonical martingale solution.

Once $(s,x)$ and such a forward rough weak solution are fixed, the backward
equation falls essentially within the classical theory of quadratic BSDEs
with bounded terminal conditions
{\cite{kobylanski_backward_2000,tevzadze_solvability_2008}}.
In particular, the standard a priori estimates, comparison results and
monotone-limit arguments carry over by exactly the same proofs as in the
classical case; we refer to
{\cite[Chapter~7]{zhang_backward_2017}} for a compact account of this theory.
Consequently, when working with a fixed $(s,x)$ and a fixed forward rough weak
solution, we shall invoke these results directly rather than restating or
reproving them. One exception is the proof of
Theorem~\ref{thm:FBSDE-wellposedness-bounded}, where we briefly recall the
well-posedness construction in order to make clear that the Brownian
filtration used in the classical formulation can be replaced by the
martingale representation property available here. The main technical
difficulty in this section is therefore to obtain uniqueness and estimates for the singular
FBSDE that are uniform in the initial condition $(s,x)$ and invariant under
the choice of natural-filtration rough weak solution.

We also establish a nonlinear Feynman--Kac correspondence between the
singular FBSDE and the associated semilinear PDE. We call
\begin{equation}
  \left(\partial_t+\frac12\Delta+b\cdot\nabla\right)u
  =-H(t,x,u,\nabla u),
  \qquad u(T,\cdot)=\bar h,
  \label{eq:section4-singular-HJB}
\end{equation}
the singular HJB equation. The first direction starts from a paracontrolled
solution $u$ of this equation. Formally applying the It\^o formula suggests that
\[
  Y_t^{s,x}=u(t,X_t^{s,x}),\qquad
  Z_t^{s,x}=\nabla u(t,X_t^{s,x})
\]
solves the backward equation; this implication is made rigorous in
Proposition~\ref{prop:converse-FK}. Conversely, the FBSDE suggests defining
\[
  u(t,x)\assign Y_t^{t,x},
\]
  as a candidate solution of the singular HJB equation. We prove in
  Theorem~\ref{thm:nonlinear-FK} that, when $H$ is jointly continuous, a bounded
  continuous universal decoupling field is a viscosity solution of the singular
  HJB equation.

For classical Markovian FBSDEs, continuity of the value field is usually a
direct consequence of stability estimates for the forward and backward
equations; see, for example, {\cite{el_karoui_backward_1997}}. The situation
here is more delicate: changing the initial condition $(t,x)$ also changes
the probability measure and potentially the stochastic basis on which the
weak solution of the singular forward equation is realized. Under the general
assumptions of Theorem~\ref{thm:FBSDE-wellposedness-bounded}, we do not
presently obtain the continuity of $u$. We give two ways to overcome this
difficulty. With additional regularity of the terminal condition and suitable
assumptions on $H$, Proposition~\ref{prop:FBSDE-lipschitz-regularity} identifies
$u$ directly with the paracontrolled solution of the HJB equation. For
bounded continuous terminal data, Theorem~
\ref{thm:FBSDE-quadratic-continuous-terminal} instead obtains continuity
through monotone approximation and quadratic BSDE stability.

\subsection{The singular forward equation}
\label{subsec:singular-forward-equation}

We begin by recalling the solution theory for the singular forward equation
that will be used throughout this section. A notion of rough weak solution has been proposed in
{\cite{kremp_rough_2025}}. We recall their definition and see in the next
theorem how its equivalence to the martingale solution strengthens our
understanding of the singular SDE.

We start by introducing the following metric that we need for the definition of
weak solution. For an adapted process $A$ and a two-parameter process
$\mathbb A=(\mathbb A_{u,v})_{s\leq u\leq v\leq T}$, set
\[
\begin{aligned}
	\|A\|_{\rho,p}
	&\assign \sup_{s\leq u<v\leq T}
		\frac{\|A_v-A_u\|_{L^p}}{|v-u|^\rho},
&	\|A\mid\mathcal F\|_{\rho,\infty}
	&\assign \sup_{s\leq u<v\leq T}
		\frac{\bigl\|\mathbb E[A_v-A_u\mid\mathcal F_u]\bigr\|_{L^\infty}}
		{|v-u|^\rho},
\\
	\|\mathbb A\|_{\rho,p}
	&\assign \sup_{s\leq u<v\leq T}
		\frac{\|\mathbb A_{u,v}\|_{L^p}}{|v-u|^\rho},
&	\|\mathbb A\mid\mathcal F\|_{\rho,\infty}
	&\assign \sup_{s\leq u<v\leq T}
		\frac{\bigl\|\mathbb E[\mathbb A_{u,v}\mid\mathcal F_u]\bigr\|_{L^\infty}}
		{|v-u|^\rho}.
\end{aligned}
\]

\begin{definition}[{\cite{kremp_rough_2025}}, Definition~3.3]
	\label{def:rough-weak-solution}
	{Suppose that \AssumptionAMixed{} holds} and fix
	$(s,x)\in[0,T]\times\mathbb R^d$. A rough weak solution of
	\[
		dX_t=b(t,X_t)\,dt+dW_t,\qquad X_s=x,
	\]
	starting from $(s,x)$ is a tuple
	\[
		\bigl(\Omega,\mathcal F,(\mathcal F_t)_{t\in[s,T]},
		\mathbb P,W,X,G,\mathbb G^b\bigr)
	\]
	satisfying the following properties.\\
	The process $W$ is a $d$-dimensional $(\mathcal F_t)$-Brownian motion,
	$G$ is continuous and adapted with $G_s=0$, and
	\[
		X_t=x+W_t-W_s+G_t,\qquad s\leq t\leq T.
	\]
	The process
	$\mathbb G^b=(\mathbb G^b_{u,v})_{s\leq u\leq v\leq T}$ is a continuous
	two-parameter process with values in $\mathbb R^{d\times d}$, where
	$\mathbb G^b_{u,v}$ is $\mathcal F_v$-measurable, and
	\[
	\begin{aligned}
		&\|G\|_{1-\alpha/2,2}
		+\|G\mid\mathcal F\|_{1-\alpha/2,\infty}
		+\|\mathbb G^b\|_{1-\alpha/2,2}
		+\|\mathbb G^b\mid\mathcal F\|_{3/2-\alpha,\infty}<\infty.
	\end{aligned}
	\]
	{Let $(b_n)$ be the fixed smooth approximation from
	\AssumptionAMixed{}. We require in addition that the following two
	conditions hold:}
	\begin{enumerate}
		\item With $G_t^n\assign\int_s^t b_n(r,X_r)\,dr$, one has
		      \[
			      \lim_{n\to\infty}\Bigl(
			      \|G^n-G\|_{1-\alpha/2,2}
			      +\|G^n-G\mid\mathcal F\|_{1-\alpha/2,\infty}
			      \Bigr)=0.
		      \]
		\item For $m,n\in\mathbb N$ and $s\leq u\leq v\leq T$, set $\mathbb G^{m,n}_{u,v}(i,j)
			      \assign\int_u^v
			      \Bigl(J^T(\partial_{x_i}b_m^j)(r,X_r)
			      -J^T(\partial_{x_i}b_m^j)(u,X_u)\Bigr)\,dG_r^{n,i}$.
			      Then
		      \begin{align*}
			      \lim_{m,n\to\infty}\Bigl(
			      \|\mathbb G^{m,n}-\mathbb G^b\|_{1-\alpha/2,2}
			      +\|\mathbb G^{m,n}-\mathbb G^b
				      \mid\mathcal F\|_{3/2-\alpha,\infty}
			      \Bigr)=0.
		      \end{align*}

	\end{enumerate}
	We also call $X$ a rough weak solution whenever such
	a tuple exists.
\end{definition}

We only state the most relevant results from {\cite{kremp_rough_2025}}
concerning the equivalence between weak and martingale solutions in the
following theorem.

\begin{theorem}
	\label{thm:weak-martingale-equivalence}{{Under \AssumptionAMixed{},}
	fix $(s,x)\in[0,T]\times\mathbb R^d$. Then the
	following statements hold.}
	\begin{enumerate}
		\item Let $X$ be a rough weak solution in the sense of
		      Definition~\ref{def:rough-weak-solution}, starting from $(s,x)$ on
		      an arbitrary stochastic basis. Then the law of $X$ on the canonical
		      path space is the unique solution $\mathbb P^{s,x}$ of the
		      martingale problem from
		      Theorem~\ref{thm:martingale-problem-wellposedness}. In particular,
		      rough weak solutions are unique in law.

		\item Let $\mathbb{P}^{s, x}$ be the solution to the martingale problem
		      from Theorem~\ref{thm:martingale-problem-wellposedness}, and let $X$
		      denote the canonical process under $\mathbb{P}^{s, x}$. Then there
		      exist a stochastic basis $(\Omega, \mathcal{F},
		      (\mathcal{F}_t)_{t\in [s, T]}, \mathbb{P})$ and a weak solution
		      $\tilde{X}$ in the sense of
		      Definition~\ref{def:rough-weak-solution} such that
	\begin{enumerate}
		\item $\mathcal{L} (\tilde{X}) =\mathcal{L} (X)$;

		\item $\mathcal{F}_t =\mathcal{F}_t^{\tilde{X}}$ for all $t \in [s, T]$,
		      where $(\mathcal{F}_t^{\tilde{X}})_{t \in [s, T]}$ denotes the
		      completed, right-continuous natural filtration generated by
		      $\tilde{X}$.
	\end{enumerate}
	\end{enumerate}
\end{theorem}

\begin{remark}
	\label{rmk:weak-sol}The first implication identifies the law of an arbitrary
	rough weak solution with the canonical martingale-solution law. The second
	implication constructs a particular rough weak solution on the completed,
	right-continuous natural filtration of its forward process. It does not
	assert that the filtration of an arbitrary rough weak solution may be
	replaced by its natural filtration. Since $X$ and $\tilde X$ have the same
	law, the Markov property and the martingale property of
	\eqref{eq:martingale-problem-process} carry over from $X$ under
	$\mathbb{P}^{s, x}$ to $\tilde X$ under $\mathbb{P}$.
\end{remark}

\begin{proof}
	The first statement follows from
	{\cite[Theorem~5.7]{kremp_rough_2025}} and the second statement from
	{\cite[Theorem~5.10]{kremp_rough_2025}}.
\end{proof}

We conclude this subsection with a useful It{\^o} formula for $u (X)$, where $u$
is the solution to the paracontrolled PDE \eqref{eq:Kolmogorov-backward} and
$X$ is a rough weak solution to \eqref{eq:singular-SDE}. Note that this falls
outside the scope of the classical It{\^o} formula, both because $X$ is no
longer a semimartingale and because $u$ does not, in general, have sufficient
regularity. Fortunately, $X$ still belongs to the class of so-called Dirichlet
processes.

\begin{definition}
	\label{def:dirichlet-process}
	Let $(\Omega,\mathcal F,(\mathcal F_t)_{t\in[s,T]},\mathbb P)$ be a
	filtered probability space. A continuous adapted $\mathbb R^m$-valued
	process $D$ is called a \emph{Dirichlet process}, respectively a
	\emph{weak Dirichlet process}, if it admits a decomposition
	\[
		D=D_s+M+A,
	\]
	where $M$ is a continuous local martingale with $M_s=0$ and $A$ is a
	continuous adapted process with $A_s=0$ and zero quadratic variation,
	respectively such that $\langle A,N\rangle=0$ for every continuous local
	martingale $N$. In either case, the decomposition is unique.
\end{definition}

\begin{remark}
	The rough weak solution $X$ is a Dirichlet process; see
	{\cite[Remark~3.5]{kremp_rough_2025}}.
\end{remark}

If the filtration of the rough weak solution is the natural filtration
generated by $X$, as in Theorem~\ref{thm:weak-martingale-equivalence}, we obtain
the following It{\^o} formula. Here and below, the natural filtration is understood to be
completed and made right-continuous.

\begin{proposition}
		\label{prop:ito-formula}
		{Under \AssumptionAMixed{}, let}
		$f \in C_T L^{\infty} (\mathbb{R}^d)$ and
		$u_T \in \mathcal{C}^{- \alpha + 2} (\mathbb{R}^d)$. Let
	$X$ be a rough weak solution to \eqref{eq:singular-SDE}, with associated
	Brownian motion $W$, on the completed, right-continuous natural filtration
	generated by $X$, and let $u$
	be the solution to the paracontrolled PDE
	\eqref{eq:Kolmogorov-backward}. Then the following It{\^o} formula holds:
		\begin{equation}
			\label{eq:ito-formula} u (t, X_t^{s, x}) = u (r, X_r^{s, x}) + \int_r^t
			\nabla u (\tau, X_{\tau}^{s, x}) \cdot dW_{\tau} + \int_r^t
		f (\tau, X_{\tau}^{s, x})  \hspace{0.17em} d \tau, \qquad s \le r \le t
		\le T.
	\end{equation}
\end{proposition}

\begin{proof}
	We apply {\cite[Proposition~3.10]{gozzi_weak_2006}} to obtain that $u (t,
		X_t^{s, x})$ is again a weak Dirichlet process with decomposition
		\begin{equation}
			\label{eq:dirichlet-decomposition} u (t, X_t^{s, x}) = u (r, X_r^{s, x}) +
			\int_r^t \nabla u (\tau, X_{\tau}^{s, x}) \cdot dW_{\tau} +
		A_{r, t},
	\end{equation}
	where $A$ is a zero-energy process, i.e. $\langle A, N \rangle = 0$ for any
	continuous $(\mathcal{F}_t^X)$-local martingale $N$. By property
	\eqref{eq:martingale-problem-process} of the martingale solution (see also
	Remark~\ref{rmk:weak-sol}), we know that
	\[ \left( u (t, X_t^{s, x}) - u (s, X_s^{s, x}) - \int_s^t f (\tau,
		X_{\tau}^{s, x}) \hspace{0.17em} d \tau \right)_{t \in [s, T]} \]
	is a martingale. By uniqueness of the decomposition of Dirichlet processes,
	it follows that
	\[ A_{r, t} = \int_r^t f (\tau, X_{\tau}^{s, x})  \hspace{0.17em} d \tau .
	\]
	Plugging this into \eqref{eq:dirichlet-decomposition} yields
	\eqref{eq:ito-formula}.
\end{proof}

\subsection{Weak well-posedness and a priori estimates}
\label{subsec:FBSDE-weak-wellposedness}

We first introduce the solution spaces for the backward equation. Given a
stochastic basis $\mathcal S=(\Omega,\mathcal F,\mathbb F,\mathbb P)$, where
$\mathbb F=(\mathcal F_t)_{t\in[s,T]}$, let $\mathbb L^\infty(\mathcal S)$
be the space of real-valued continuous adapted processes $Y$ such that
\[
  \|Y\|_{\mathbb L^\infty(\mathcal S)}
  \assign\left\|\sup_{t\in[s,T]}|Y_t|\right\|_{L^\infty(\mathbb P)}
  <\infty,
\]
and let $\mathbb S^2(\mathcal S)$ denote the space of real-valued continuous
adapted processes $Y$ such that
\[
  \|Y\|_{\mathbb S^2(\mathcal S)}
  \assign\left(\mathbb E\left[\sup_{t\in[s,T]}|Y_t|^2\right]\right)^{1/2}
  <\infty.
\]
Moreover, for $k\in\mathbb N$, let $\mathbb H^2(\mathcal S;\mathbb R^k)$ be
the space of predictable $\mathbb R^k$-valued processes $Z$ such that
\[
  \|Z\|_{\mathbb H^2(\mathcal S;\mathbb R^k)}
  \assign\left(\mathbb E\left[\int_s^T|Z_r|^2\,dr\right]\right)^{1/2}
  <\infty.
\]
We write $\mathbb H^2(\mathcal S)$ when the state space is clear from the
context.

We can now define a (rough) weak solution of the singular FBSDE.
\begin{definition}
  \label{def:FBSDE-rough-weak-solution}
  Fix $s\in[0,T]$ and $x\in\mathbb R^d$. We call $(X,Y,Z)$ a (rough) weak solution to the
  FBSDE \eqref{eq:singular-SDE}--\eqref{eq:quadratic-BSDE} with initial
  condition $(s,x)$ if there exists a
  stochastic basis $\mathcal{S}= (\Omega, \mathcal{F},
  (\mathcal{F}_t^X)_{t \in [s, T]}, \mathbb{P})$, where
  $(\mathcal F_t^X)_{t\in[s,T]}$ denotes the completed, right-continuous
  natural filtration of $X$, such that $X$ is a (rough) weak solution in the
  sense of Definition~\ref{def:rough-weak-solution}, starting from $X_s=x$.
  Moreover, $(Y, Z)$ belongs to
  $\mathbb L^\infty(\mathcal S)\times\mathbb H^2(\mathcal S;\mathbb R^d)$,
  and
  $(X, Y, Z)$ satisfies the
  BSDE \eqref{eq:quadratic-BSDE} in integral form.
\end{definition}

We next formulate the relevant notions of uniqueness for such weak solutions.
\begin{definition}
  \label{def:FBSDE-uniqueness}
  Fix $(s,x)\in[0,T]\times\mathbb R^d$ and set
  \[
    \mathcal E_{s,T}
    :=C([s,T];\mathbb R^d)\times C([s,T];\mathbb R)
      \times L^2([s,T];\mathbb R^d)
      \times C([s,T];\mathbb R^d),
  \]
  endowed with its product Borel sigma-field.
  \begin{enumerate}
    \item The singular FBSDE is \tmtextbf{unique in law} at $(s,x)$ if any
    two weak solutions $(\mathcal S,X,Y,Z)$ and
    $(\mathcal S',X',Y',Z')$ with associated Brownian motions $W$ and $W'$
    and initial condition $(s,x)$ satisfy
    \[
      \mathcal L_{\mathbb P}(X,Y,[Z],W-W_s)
      =\mathcal L_{\mathbb P'}(X',Y',[Z'],W'-W'_s)
      \quad\text{on }\mathcal E_{s,T},
    \]
    where $Z$ and $Z'$ are identified with their equivalence classes in
    $L^2([s,T];\mathbb R^d)$.

    \item \tmtextbf{Fixed-$X$ uniqueness} holds at $(s,x)$ if, for every fixed
    natural-filtration rough weak solution $(\mathcal S,X)$ of the forward
    equation starting from $(s,x)$, any two backward solutions $(Y,Z)$ and
    $(Y',Z')$ on $\mathcal S$ satisfy $Y=Y'$ indistinguishably and
    $Z=Z'$ $\mathrm dt\otimes\mathbb P$-almost surely.
  \end{enumerate}
  If uniqueness in law and fixed-$X$ uniqueness both hold, we say that the
  singular FBSDE has a {\tmstrong{unique solution}} at $(s,x)$.
\end{definition}

\begin{remark}
Related notions of weak solutions for FBSDEs appear in
{\cite{antonelli_weak_2003,ma_weak_2008}}. In those works, the weak
formulation is motivated by the coupled forward--backward dynamics, whereas
in our setting the forward equation is decoupled and the weak formulation is
needed because of the singular drift in the SDE. Antonelli and Ma do not
include $Z$ in their uniqueness-in-law requirement, whereas Ma, Zhang and
Zheng include its finite-dimensional distributions in their martingale-problem
formulation. Here we include $Z$ as an
$L^2([s,T];\mathbb R^d)$-valued random variable, together with the Brownian
increment $W-W_s$. 
\end{remark}

The preceding definition concerns uniqueness after the initial condition
$(s,x)$ has been fixed. We next introduce a notion that links different 
rough weak solution even for different initial conditions.
\begin{definition}
  \label{def:FBSDE-universal-decoupling-field}
  We say that the singular FBSDE admits a \tmtextbf{universal decoupling
  field} if there exists a deterministic Borel function
  \[
    u:[0,T]\times\mathbb R^d\longrightarrow\mathbb R
  \]
  such that, for every $(s,x)\in[0,T]\times\mathbb R^d$ and every (rough)
  weak solution
  $(\mathcal S^{s,x},X^{s,x},Y^{s,x},Z^{s,x})$ with initial condition
  $(s,x)$, the processes $Y^{s,x}$ and
  $u(\cdot,X^{s,x}_\cdot)$ are indistinguishable; that is,
  \[
    Y_t^{s,x}=u(t,X_t^{s,x}),\qquad t\in[s,T],
    \quad\mathbb P^{s,x}\text{-a.s.}
  \]
  Provided that a weak solution exists for every initial condition, such a
  field is necessarily unique: taking $t=s$ and using $X_s^{s,x}=x$ yields
  $u(s,x)=Y_s^{s,x}$ for every $(s,x)$.
\end{definition}

{Throughout this section, we work under \AssumptionB{}.
The regularity of the terminal condition will be specified separately in each
result. We often use the following Lipschitz condition as an intermediate
step, and therefore record it here.}

\phantomsection\label{assumption:B-Lip}
{\tmstrong{Assumption B-Lip: }} For some $C>0$ and all $t,x,y,y',z,z'$,
\[ |H(t,x,y,z)-H(t,x,y',z')|
   \leq C\left(|y-y'|+|z-z'|\right), \qquad
   |H(t,x,0,0)|\leq C. \]

We start with an a priori estimate for quadratic BSDEs, which will be important
later. The result and its proof follow
{\cite[Theorem~7.2.1]{zhang_backward_2017}}. We reproduce the argument to make
explicit that all constants are independent of the initial condition $(s,x)$
and of the chosen forward rough weak solution.

\begin{theorem}[A priori estimates]
  \label{thm:quadratic-BSDE-apriori-estimates}
  Let $\bar h\in L^\infty(\mathbb R^d)$, $s \in [0, T]$ and $x \in \R^d$.
  {Under Assumptions \assAMixed{} and \assB{},} write $C_0$ for the constant in the
  quadratic-growth bound \eqref{assump:quadratic growth}, and
  {let $(X^{s,x},Y^{s,x},Z^{s,x})$ be a rough weak solution
  to the singular FBSDE.} Then
  \[
    \|Y^{s,x}\|_{L^\infty(\Omega\times[s,T])}
    \leq K_Y
    \assign e^{2(1+C_0)T}
      \bigl(\|\bar h\|_{L^\infty}+2\bigr)
  \]
  and
  \[
    \sup_{\tau\in\mathcal T_{s,T}}
    \left\|
      \mathbb E^{\mathbb P^{s,x}}\left[
        \left.\int_\tau^T|Z_r^{s,x}|^2\,dr
        \right|\mathcal F_\tau^{s,x}
      \right]
    \right\|_{L^\infty(\mathbb P^{s,x})}
    \leq
    \frac{e^{2(1+C_0)(K_Y+\|\bar h\|_{L^\infty})}}{2(1+C_0)}
    +2T(1+C_0)e^{4(1+C_0)K_Y}(1+K_Y),
  \]
  where $\mathcal T_{s,T}$ denotes the set of
  $(\mathcal F_t^{s,x})$-stopping times with values in $[s,T]$.
  In particular, $Z^{s,x}\cdot W^{s,x}$ is a BMO martingale. Both bounds are
  independent of $(s,x)$ and of the chosen
  natural-filtration rough weak solution.
\end{theorem}

\begin{proof}
  Fix $(s,x)$ and a natural-filtration rough weak solution. Then
  \[
    Y_t^{s,x}=\bar h(X_T^{s,x})
    +\int_t^T H(r,X_r^{s,x},Y_r^{s,x},Z_r^{s,x})\,dr
    -\int_t^TZ_r^{s,x}\,dW_r^{s,x},
  \]
  where
  \[
    |\bar h(X_T^{s,x})|\leq\|\bar h\|_{L^\infty},\qquad
    |H(t,X_t^{s,x},y,z)|\leq C_0(1+|y|+|z|^2)
  \]
  with the same deterministic constant $C_0$ for every $(s,x)$ and every
  natural-filtration rough weak solution. \\
  We first estimate $Y^{s,x}$. Fix $t_0\in[s,T]$ and set
  \[
    \sigma\assign\inf\{t\geq t_0:|Y_t^{s,x}|\leq1\}\wedge T.
  \]
  Then $\sigma=t_0$ on $\{|Y_{t_0}^{s,x}|\leq1\}$, while on
  $\{|Y_{t_0}^{s,x}|>1\}$ one has $|Y_t^{s,x}|>1$ for
  $t_0\leq t<\sigma$ and
  \[
    |Y_\sigma^{s,x}|\leq1\vee\|\bar h\|_{L^\infty}.
  \]
  Let $\psi\geq0$ be smooth, equal to $|y|$ for $|y|\geq1$, bounded by $1$
  on $[-1,1]$, and chosen so that $|y|\leq\psi(y)+1$. Set
  \[
    \lambda:=2(1+C_0),\qquad
    \varphi(t,y):=\lambda e^{\lambda t}(\psi(y)+1),
  \]
  and, on $[t_0,\sigma]$,
  \[
    \widetilde Y_t^{s,x}\assign e^{\varphi(t,Y_t^{s,x})},\qquad
    \widetilde Z_t^{s,x}\assign
      \partial_y\varphi(t,Y_t^{s,x})\widetilde Y_t^{s,x}Z_t^{s,x}.
  \]
  It\^o's formula and the quadratic-growth estimate give, whenever
  $|Y_t^{s,x}|>1$,
  \begin{align*}
    d\widetilde Y_t^{s,x}
    ={}&\widetilde Y_t^{s,x}\bigg[
      \partial_t\varphi
      -\partial_y\varphi
        H(t,X_t^{s,x},Y_t^{s,x},Z_t^{s,x})
      +\frac12\partial_{yy}\varphi|Z_t^{s,x}|^2
      +\frac12 |\partial_y\varphi|^2|Z_t^{s,x}|^2
    \bigg]dt+\widetilde Z_t^{s,x}\,dW_t^{s,x}\\
    \geq{}&\widetilde Z_t^{s,x}\,dW_t^{s,x}.
  \end{align*}
  Indeed, on $|y|\geq1$ the drift is bounded from below by
  \[
    \lambda^2e^{\lambda t}(1+|y|)
    -C_0\lambda e^{\lambda t}(1+|y|+|z|^2)
    +\frac12\lambda^2e^{2\lambda t}|z|^2\geq0.
  \]
  Since $Y^{s,x}$ is bounded and $Z^{s,x}\in\mathbb H^2$, the stochastic
  integral above is a true martingale. Thus $\widetilde Y^{s,x}$ is a
  submartingale on $[t_0,\sigma]$. On $\{|Y_{t_0}^{s,x}|>1\}$,
  \[
    e^{\lambda e^{\lambda t_0}|Y_{t_0}^{s,x}|}
    \leq\widetilde Y_{t_0}^{s,x}
    \leq\mathbb E_{t_0}^{\mathbb P^{s,x}}[\widetilde Y_\sigma^{s,x}]
    \leq
    e^{\lambda e^{\lambda T}(\|\bar h\|_{L^\infty}+2)}.
  \]
  Together with the trivial estimate on $\{|Y_{t_0}^{s,x}|\leq1\}$, this
  yields
  \begin{equation}
    |Y_{t_0}^{s,x}|\leq
    K_Y=e^{\lambda T}(\|\bar h\|_{L^\infty}+2),
    \qquad t_0\in[s,T].
    \label{eq:quadratic-bsde-Y-apriori}
  \end{equation}
  We next estimate $Z^{s,x}$. Applying It\^o's formula to
  $e^{\lambda Y_t^{s,x}}$ gives
  \[
    de^{\lambda Y_t^{s,x}}
    \geq\lambda e^{\lambda Y_t^{s,x}}
      \bigl[|Z_t^{s,x}|^2-\lambda(1+|Y_t^{s,x}|)\bigr]dt
      +\lambda e^{\lambda Y_t^{s,x}}Z_t^{s,x}\,dW_t^{s,x}.
  \]
  Integrating from an arbitrary $\tau\in\mathcal T_{s,T}$ to $T$ and taking
  the conditional expectation yields
  \[
    \lambda e^{-\lambda K_Y}
    \mathbb E_\tau^{\mathbb P^{s,x}}
      \left[\int_\tau^T|Z_r^{s,x}|^2\,dr\right]
    \leq e^{\lambda\|\bar h\|_{L^\infty}}
      +T\lambda^2e^{\lambda K_Y}(1+K_Y).
  \]
  Notice that the bound on $Y$ and $Z$ involve
  only $\lambda$, $C_0$, $T$ and $\|\bar h\|_{L^\infty}$, which are uniform in $(s,x)$ and in the chosen natural-filtration rough
  weak solution.
\end{proof}

We next complement the preceding upper estimate by imposing the lower growth
bound \eqref{eq:quadratic-BSDE-lower-growth} on $H$. This condition forces
$H$ to grow at least quadratically in $z$, while allowing it to decrease in
$y$, but at most linearly. Under this assumption we obtain a relatively sharp
lower bound for $Y^{s,x}$. The estimate is in the spirit of
{\cite[Theorem~2.1]{perkowski_coming_2025}}, where a stochastic-control
argument gives a lower bound that is insensitive to possibly very negative
boundary data. More precisely, the final estimate
\eqref{eq:quadratic-BSDE-lower-transition} below shows that, once the
probability of reaching a region on which the terminal condition is
nonnegative is bounded from below, our bound does not depend on how negative
the terminal condition may be outside that region.

\begin{proposition}[A terminal-independent lower estimate]
  \label{prop:quadratic-BSDE-lower-estimate}
  {{Under Assumptions \assAMixed{} and \assB{},} let
  $\bar h\in L^\infty(\mathbb R^d)$, $(s,x)\in[0,T]\times\mathbb R^d$, and
  let $(X^{s,x},Y^{s,x},Z^{s,x})$ be a rough weak solution of the singular
  FBSDE such that $Y^{s,x}$ is essentially bounded.} Assume in addition
  that, for some $c>0$ and $C\geq0$,
  \begin{equation}
    H(t,x,y,z)\geq c|z|^2-C(1+|y|).
    \label{eq:quadratic-BSDE-lower-growth}
  \end{equation}
  Then, for every $t\in[s,T]$,
  \[
    Y_t^{s,x}\geq
    1-e^{C(T-t)}
    +\frac{e^{C(T-t)}}{2c}
    \log\mathbb E^{\mathbb P^{s,x}}\left[\left.
      e^{2c(\bar h(X_T^{s,x})\wedge0)}
      \right|\mathcal F_t^{s,x}\right]
    \qquad \mathbb P^{s,x}\text{-a.s.}
  \]
  In particular, for the deterministic function $u(s,x)=Y_s^{s,x}$ from
  Theorem~\ref{thm:FBSDE-wellposedness-bounded}, if
  $A\subset\mathbb R^d$ is Borel and $\bar h\geq0$ on $A$, then
  \begin{equation}
    u(s,x)\geq
    1-e^{C(T-s)}
    +\frac{e^{C(T-s)}}{2c}
      \log\mathbb P^{s,x}(X_T^{s,x}\in A).
    \label{eq:quadratic-BSDE-lower-transition}
  \end{equation}
  {In particular, for $d=1$ and $A=[-1,1]$, there exists
  $K=K(T,b,B)>0$ such that, for every $s<T$,
  \[
    u(s,x)\geq
    1-e^{C(T-s)}
    -\frac{K e^{C(T-s)}}{2c}
      \left(1+\frac{\mathbf 1_{\{|x|>1\}}x^2}{T-s}\right).
  \]}
\end{proposition}

\begin{proof}
  Fix $(s,x)$ and a natural-filtration rough weak solution, write
  $\xi=\bar h(X_T^{s,x})$, and omit the superscripts $(s,x)$. Fix
  $t_0\in[s,T]$ and define
  \[
    \tau\assign\inf\{t\in[t_0,T]:Y_t\geq0\}\wedge T.
  \]
  On $\{Y_{t_0}<0\}$, continuity gives $Y_t<0$ for
  $t\in[t_0,\tau)$, and hence
  \[
    H(t,X_t,Y_t,Z_t)\geq c|Z_t|^2+C(Y_t-1),
    \qquad t\in[t_0,\tau).
  \]
  Set
  \[
    a(t)\assign2ce^{-C(T-t)},\qquad
    R_t\assign e^{a(t)(Y_t-1)}.
  \]
  Since $a'(t)=Ca(t)$, It\^o's formula on $[t_0,\tau]$ gives
  \[
    \frac{dR_t}{R_t}
    =\left[
      Ca(t)(Y_t-1)-a(t)H(t,X_t,Y_t,Z_t)
      +\frac12a(t)^2|Z_t|^2
    \right]dt+a(t)Z_t\,dW_t.
  \]
  On $\{Y_{t_0}<0\}$ and before $\tau$, the drift is bounded above by
  \[
    \frac12a(t)(a(t)-2c)|Z_t|^2\leq0,
  \]
  because $a(t)\leq2c$. On $\{Y_{t_0}\geq0\}$ one has
  $\tau=t_0$, so the stopped process is constant. Since $Y$ is bounded,
  $R_{t\wedge\tau}$ is bounded as well. The stochastic integral is therefore
  a true martingale after the usual localization argument, and the stopped
  process is a supermartingale. Consequently,
  \[
    e^{a(t_0)(Y_{t_0}-1)}
    \geq\mathbb E_{t_0}\left[e^{a(\tau)(Y_\tau-1)}\right].
  \]

  On $\{Y_\tau\geq0\}$, which includes $\{\tau<T\}$, one has
  \[
    e^{a(\tau)(Y_\tau-1)}\geq e^{-a(\tau)}
    \geq e^{-2c}\geq e^{2c(\xi\wedge0-1)}.
  \]
  On $\{Y_\tau<0\}$ one necessarily has $\tau=T$ and
  $\xi=Y_T<0$, so
  \[
    e^{a(\tau)(Y_\tau-1)}
    =e^{2c(\xi-1)}=e^{2c(\xi\wedge0-1)}.
  \]
  Hence
  \[
    e^{a(t_0)(Y_{t_0}-1)}
    \geq\mathbb E_{t_0}\left[e^{2c(\xi\wedge0-1)}\right].
  \]
  Taking logarithms and recalling that
  $a(t_0)=2ce^{-C(T-t_0)}$ proves the first claim.

  Finally, if $\bar h\geq0$ on $A$, then
  \[
    \mathbb E^{\mathbb P^{s,x}}\left[
      e^{2c(\bar h(X_T^{s,x})\wedge0)}\right]
    \geq\mathbb P^{s,x}(X_T^{s,x}\in A).
  \]
  Since $\mathcal F_s^{s,x}$ is trivial up to null sets and
  $u(s,x)=Y_s^{s,x}$, this yields
  \eqref{eq:quadratic-BSDE-lower-transition}. The final one-dimensional
  bound follows by applying the Gaussian lower estimate from
  {\cite[Corollary~2.15]{perkowski_coming_2025}}.
\end{proof}

We next establish well-posedness of the singular FBSDE under \AssumptionB{},
assuming only in addition that the terminal condition is bounded, while all
subsequent results require it to be at least bounded and continuous. It may be possible to go beyond bounded terminal
conditions by adapting the classical theory of quadratic BSDEs with unbounded
terminal data; see
{\cite{briand_hu_convex_2008,delbaen_hu_richou_2011}}. In these works, the
explicit quantitative local Lipschitz condition in $z$ is replaced by
convexity or concavity and continuity in $z$. Extending the singular theory in
this direction could also be interesting, but we do not pursue it here.

\begin{theorem}[Well-posedness for bounded terminal condition]
  \label{thm:FBSDE-wellposedness-bounded}
  \label{thm:FBSDE-wellposedness}
  Assume that $\bar h\in L^\infty(\mathbb{R}^d)$ and that Assumptions
  {\assAMixed{}} and \assB{} hold. Then, for every $s\in[0,T]$ and $x\in\mathbb{R}^d$ and
  every natural-filtration rough weak solution
  $(\mathcal S,X^{s,x},G^{s,x},W^{s,x})$ of the forward equation, the
  corresponding BSDE admits a solution $(Y^{s,x},Z^{s,x})$ on the same
  stochastic basis. The singular FBSDE has a unique solution at $(s,x)$ in
  the sense of
  Definition~\ref{def:FBSDE-uniqueness}. In particular, the joint law of
  $(X^{s,x},Y^{s,x},[Z^{s,x}],W^{s,x}-W_s^{s,x})$ is independent of the
  chosen natural-filtration rough weak solution. Moreover, $Y_s^{s,x}$ is
  deterministic, and
  \[
    u(s,x)\assign Y_s^{s,x}
  \]
  defines a bounded deterministic function on
  $[0,T]\times\mathbb R^d$, independent of the chosen natural-filtration
  rough weak solution.
\end{theorem}

\begin{proof}
  Fix $(s,x)\in[0,T]\times\mathbb R^d$ and a natural-filtration rough weak
  solution,
  \[
    (\Omega^{s,x},\mathcal F^{s,x},
    (\mathcal F_t^{X^{s,x}})_{t\in[s,T]},\mathbb P^{s,x},X^{s,x},
    G^{s,x},W^{s,x}).
  \]
  Denote the underlying stochastic basis by $\mathcal S^{s,x}$.
  The martingale representation theorem
  {\cite[Proposition~2.8]{perkowski_coming_2025}} holds on this filtration.\footnote{The cited result is stated and proved on the one-dimensional torus, the same proof however holds on $\mathbb R^d$.}
  We show that on this stochastic basis there exists a unique pair
  $(Y^{s,x},Z^{s,x})\in\mathbb L^\infty\times\mathbb H^2$ solving the BSDE
  \eqref{eq:quadratic-BSDE}.

  We first consider the Lipschitz case. Thus, in addition to \AssumptionB{},
  assume for the moment that \AssumptionBLip{} holds. Set
  $(Y^{s,x,0},Z^{s,x,0})\equiv(0,0)$ and define the Picard iteration by
  \[
    Y_t^{s,x,n+1}
    =\mathbb E^{\mathbb P^{s,x}}\left[\left.
      \bar h(X_T^{s,x})+\int_t^T
      H(r,X_r^{s,x},Y_r^{s,x,n},Z_r^{s,x,n})\,dr
      \,\right|\mathcal F_t^{X^{s,x}}\right].
  \]
  As for $Z^{s,x,n+1}$, notice that
  \[
    Y_t^{s,x,n+1}+\int_s^t
    H(r,X_r^{s,x},Y_r^{s,x,n},Z_r^{s,x,n})\,dr
  \]
  is an $(\mathcal F_t^{X^{s,x}})_{t\in[s,T]}$-martingale. We define
  $Z^{s,x,n+1}$ to be the unique predictable process satisfying the martingale
  representation
  \[
    Y_t^{s,x,n+1}+\int_s^t
    H(r,X_r^{s,x},Y_r^{s,x,n},Z_r^{s,x,n})\,dr
    =Y_s^{s,x,n+1}+\int_s^t Z_r^{s,x,n+1}\,dW_r^{s,x},
    \qquad s\leq t\leq T,
  \]
  which exists by
  {\cite[Proposition~2.8]{perkowski_coming_2025}}. Following the exact argument
  of {\cite[Theorem~2.1 and Corollary~2.1]{el_karoui_backward_1997}}, the map
  which sends $(Y^{s,x,n},Z^{s,x,n})$ to
  $(Y^{s,x,n+1},Z^{s,x,n+1})$ is a contraction on
  $\mathbb H^2(\mathcal S^{s,x};\mathbb R)\times
  \mathbb H^2(\mathcal S^{s,x};\mathbb R^d)$, equipped with a suitable
  exponentially weighted norm. Hence the sequence converges in
  $\mathbb H^2\times\mathbb H^2$ to the solution of the Lipschitz BSDE.
  In particular, after extracting a subsequence, the convergence holds
  $\mathrm dt\otimes\mathbb P^{s,x}$-a.s. \\
  For the general quadratic case we follow
  {\cite[Theorem~7.3.3]{zhang_backward_2017}}. For $n,m,k\in\mathbb N$, set
  \[
    H^n\assign H\wedge n,\qquad
    H^{n,m}\assign H^n\vee(-m),\qquad
    H^{n,m,k}(t,x,y,z)
    \assign \inf_{z'\in\mathbb R^d}
    \left[H^{n,m}(t,x,y,z')+k|z-z'|\right].
  \]
  Then $H^{n,m,k}$ is Lipschitz in $(y,z)$, and it holds the following
  monotone convergence:
  \[
    H^{n,m,k}\nearrow H^{n,m},\qquad
    H^{n,m}\searrow H^n,\qquad
    H^n\nearrow H.
  \]
  Applying the Lipschitz case to $H^{n,m,k}$ gives a unique solution
  $(Y^{s,x,n,m,k},Z^{s,x,n,m,k})$. Following the existence argument in
  {\cite[Theorem~7.3.3]{zhang_backward_2017}}, we may therefore pass
  successively to the limits $k\to\infty$, $m\to\infty$ and $n\to\infty$.
  At each stage the corresponding pairs converge in
  $\mathbb S^2\times\mathbb H^2$, and the final limit will be denoted by
  $(Y^{s,x},Z^{s,x})$. As in the cited proof, one can verify that the limiting
  pair indeed satisfies the integral equation of the BSDE driven by $H$. To see that $Y^{s,x}$ is indeed bounded, observe first that, for fixed
  $n,m$, the bound $-m\leq H^{n,m,k}\leq n$ gives
  \[
    \|Y^{s,x,n,m,k}\|_{\mathbb L^\infty}
    \leq \|\bar h\|_{L^\infty}+(n\vee m)T,
  \]
  and the same bound holds for the limit $Y^{s,x,n,m}$ as $k\to\infty$.
  Moreover, $H^{n,m}$ satisfies the quadratic-growth bound
  \eqref{assump:quadratic growth} with the same constant as $H$. Theorem
  \ref{thm:quadratic-BSDE-apriori-estimates}, applied to
  $(Y^{s,x,n,m},Z^{s,x,n,m})$, therefore gives an a priori bound independent
  of $n,m$, $(s,x)$ and the chosen natural-filtration rough weak solution.
  Passing this bound through the remaining limits proves that
  $(Y^{s,x},Z^{s,x})\in\mathbb L^\infty\times\mathbb H^2$. Fixed-$X$
  uniqueness now follows from the comparison theorem
  {\cite[Theorem~7.3.1]{zhang_backward_2017}}. Existence of at least one
  natural-filtration rough weak solution follows from
  Theorem~\ref{thm:weak-martingale-equivalence}.

  It remains to prove uniqueness in law across different natural-filtration
  rough weak solutions. Let
  $(\mathcal S^i,X^i,G^i,W^i,Y^i,Z^i)$, $i=1,2$, be two solutions with the same
  initial condition $(s,x)$. By
  Theorem~\ref{thm:weak-martingale-equivalence}, both forward processes have
  the same law as the unique martingale-solution law, i.e.
  \[
    \mathcal L_{\mathbb P^1}(X^1)
    =\mathcal L_{\mathbb P^2}(X^2)\assign\mu.
  \]
  Let $\mathbf X$ be the coordinate process on $C([s,T];\mathbb R^d)$, equipped
  with the measure $\mu$ and the usual augmentation of its canonical
  filtration. Put $B^i=W^i-W_s^i$. Since $B^i$ is continuous and
  adapted to the completed natural filtration of $X^i$, it is a
  $C([s,T];\mathbb R^d)$-valued random variable measurable with respect to the
  completed sigma-field generated by $X^i$. The Doob--Dynkin factorization lemma
  {\cite[Lemma~1.14]{kallenberg_foundations_2021}} gives a Borel map
  $\beta_i:C([s,T];\mathbb R^d)\to C([s,T];\mathbb R^d)$ such that
  \[
    B^i=\beta_i(X^i),\qquad \mathbb P^i\text{-a.s.}
  \]
  {For every rational $q\in[s,T]$, the random variable $B_q^i$
  is $\mathcal F_q^{X^i}$-measurable, so it admits an
  $\mathcal F_q^{\mathbf X}$-measurable factorization through $X^i$.
  Hence $\beta_i(\mathbf X)_q$ is
  $\mathcal F_q^{\mathbf X}$-measurable, $\mu$-a.s. Taking a common null set
  over rational $q$, continuity of $\beta_i(\mathbf X)$ and right-continuity
  of the canonical filtration show that $\beta_i(\mathbf X)$ is adapted.}
  On the canonical space, $\beta_i(\mathbf X)$ is a continuous local martingale and
  $\mathbf X-x-\beta_i(\mathbf X)$ has zero quadratic variation. The
  uniqueness of the Dirichlet decomposition of $\mathbf X$ on the canonical
  filtered space yields
  \[
    \beta_1(\mathbf X)=\beta_2(\mathbf X)=:\mathbf B,
    \qquad\mu\text{-a.s.}
  \]
  where $\mathbf B$ is a Brownian motion with respect to the usual augmented
  canonical filtration. \\
  Applying Doob--Dynkin again to $Y^i$ gives a Borel map
  $\Phi_i:C([s,T];\mathbb R^d)\to C([s,T];\mathbb R)$ such that
  \[
    Y^i=\Phi_i(X^i),\qquad \mathbb P^i\text{-a.s.}
  \]
  {The same rational-time argument shows that
  $\Phi_i(\mathbf X)$ is adapted to the usual augmented canonical
  filtration.}
  For $Z^i$ it is slightly more complicated, since we want to factorize $Z^i$ 
  while preserving predictability. Let
  $(\mathcal F_t^{\mathbf X,0})_{t\in[s,T]}$ be the
  raw canonical filtration on $C([s,T];\mathbb R^d)$ generated by
  $\mathbf X$, and let $\mathcal P^{\mathbf X,0}$ be its predictable
  sigma-field. Likewise,
  let $\mathcal P^{X^i,0}$ be the predictable sigma-field associated with the
  raw natural filtration of $X^i$. The usual augmentation does not change
  predictable processes modulo $\mathrm dt\otimes\mathbb P^i$, so we may
  replace $Z^i$ by a $\mathcal P^{X^i,0}$-measurable representative. We define
  \[
    \Theta_i:[s,T]\times\Omega^i
    \longrightarrow[s,T]\times C([s,T];\mathbb R^d),
    \qquad \Theta_i(t,\omega)=(t,X^i(\omega)).
  \]
  Since $\mathcal F_t^{X^i,0}=(X^i)^{-1}(\mathcal F_t^{\mathbf X,0})$,
  comparing the generators of the two predictable sigma-fields gives
  \[
    \mathcal P^{X^i,0}=\Theta_i^{-1}(\mathcal P^{\mathbf X,0}).
  \]
  The Doob--Dynkin factorization lemma, now applied on the space--time
  space, yields a $\mathcal P^{\mathbf X,0}$-measurable map
  $\Psi_i:[s,T]\times C([s,T];\mathbb R^d)\to\mathbb R^d$ such that
  \[
    Z_t^i=\Psi_i(t,X^i), \qquad
    \mathrm dt\otimes\mathbb P^i\text{-a.e.}
  \]
  In particular, $\Psi_i(t,\mathbf X)$ is predictable on the canonical
  space, and equality of the laws of $X^i$ and $\mathbf X$ gives
  \[
    \mathbb E^\mu\left[\int_s^T
      |\Psi_i(t,\mathbf X)|^2\,dt\right]
    =\mathbb E^{\mathbb P^i}\left[\int_s^T|Z_t^i|^2\,dt\right]<\infty.
  \]
  Define
  \[
    \mathbf Y^i=\Phi_i(\mathbf X),\qquad
    \mathbf Z_t^i=\Psi_i(t,\mathbf X).
  \]
  Both pairs solve on the same canonical stochastic basis the BSDE with
  forward process $\mathbf X$ and Brownian motion $\mathbf B$. Indeed, since
  $X^i$ and $\mathbf X$ have the same law and the same maps $\Phi_i$ and
  $\Psi_i$ represent $(Y^i,Z^i)$ and $(\mathbf Y^i,\mathbf Z^i)$, respectively,
  the terminal condition and the Lebesgue time integral in the BSDE transfer
  directly. For the stochastic integral, approximate $\Psi_i$ in
  $L^2(\mathrm dt\otimes\mu)$ by canonical predictable simple processes.
  For a simple process the integral is a finite sum and transfers directly
  under the identities $B^i=\beta_i(X^i)$ and
  $\mathbf B=\beta_i(\mathbf X)$. The It{\^o} isometry then transfers the
  integral for $\mathbf Z^i$ by passage to the limit.

  Fixed-$X$ uniqueness, established above, now gives
  \[
    \mathbf Y^1=\mathbf Y^2\quad\text{indistinguishably},\qquad
    \mathbf Z^1=\mathbf Z^2\quad
    \mathrm dt\otimes\mu\text{-a.s.}
  \]
  Consequently,
  \[
    \mathcal L_{\mathbb P^1}(X^1,Y^1,[Z^1],W^1-W_s^1)
    =\mathcal L_{\mathbb P^2}(X^2,Y^2,[Z^2],W^2-W_s^2)
  \]
  on
  $C([s,T];\mathbb R^d)\times C([s,T];\mathbb R)
  \times L^2([s,T];\mathbb R^d)\times C([s,T];\mathbb R^d)$, which is the
  asserted uniqueness in law.
  Finally, by Theorems~\ref{thm:martingale-problem-wellposedness} and
  \ref{thm:weak-martingale-equivalence}, $X^{s,x}$ is strong Markov. Since we
  work with its completed, right-continuous natural filtration and
  $X_s^{s,x}=x$ is deterministic, the Blumenthal zero--one law applies
  directly and shows that $\mathcal F_s^{X^{s,x}}$ is trivial up to
  $\mathbb P^{s,x}$-null sets. Since $Y_s^{s,x}$ is measurable with respect to
  this sigma-field,
  it is almost surely constant. The equality of laws just proved shows that
  this constant does not depend on the chosen natural-filtration rough weak
  solution, and the a priori
  estimate in Theorem~\ref{thm:quadratic-BSDE-apriori-estimates} makes the
  resulting function $u$ bounded.
\end{proof}

\begin{remark}
  The preceding theorem defines the deterministic value field
  $u(s,x)=Y_s^{s,x}$, but does not assert the existence of common decoupling fields $u$
  and $v$ such that
  \[
    Y_t^{s,x}=u(t,X_t^{s,x}),\qquad
    Z_t^{s,x}=v(t,X_t^{s,x})
  \]
  for all initial conditions. The essential first step towards such a result
  is to realize the family $(\mathbb P^{s,x})_{s,x}$ on a common
  time-augmented canonical space as a measurable strong Markov family and to
  construct the Brownian martingale part $W-W_s$ (equivalently, the rough
  drift part $G$) as a common continuous additive functional of this family.
  If this can be done, then the Picard martingales in the Lipschitz case become
  common additive martingale functionals, and one can proceed as in
  {\cite[Theorem~4.1]{el_karoui_backward_1997}}, applying
  {\cite[Theorem~6.27]{cinlar_semimartingales_1980}} to obtain common fields
  $u$ and $v$. In the quadratic case, the same monotone approximation should
  give a common $u$, while obtaining a common $v$ in the limit may require
  additional care. We do not pursue this extension here.
\end{remark}

\subsection{A singular non-linear Feynman--Kac result}
\label{subsec:FBSDE-viscosity-identification}

We now establish the connection between the singular HJB equation
\eqref{eq:section4-singular-HJB} and the singular FBSDE
\eqref{eq:singular-SDE}--\eqref{eq:quadratic-BSDE}. We start with the easier
direction: Proposition~\ref{prop:converse-FK} shows that a paracontrolled
solution of the singular HJB, evaluated along a rough weak solution of the
forward equation, gives a weak solution of the singular FBSDE. We refer the reader to
Appendix~\ref{app:paracontrolled-wellposedness} for conditions ensuring the
existence and uniqueness of such a paracontrolled solution.

Conversely, Theorem~\ref{thm:nonlinear-FK} shows that, when $H$ is jointly
continuous, a bounded continuous universal decoupling field is a viscosity
solution of the singular HJB. These properties can in many cases be verified
directly.
In particular, Theorem~\ref{thm:FBSDE-quadratic-continuous-terminal} treats bounded continuous
terminal data, which lie beyond the regular-terminal paracontrolled result of
Theorem~\ref{thm:quadratic-paracontrolled-global}, and identifies $u$ as the
unique viscosity solution of the singular HJB.

\begin{proposition}
  \label{prop:converse-FK}
  {Under Assumptions \assAMixed{}, \assB{}, and \assPContinuous{},} let
  $h$ be a bounded paracontrolled solution to
  \eqref{eq:section4-singular-HJB}, and let $X^{s,x}$ be a
  natural-filtration rough weak solution to \eqref{eq:singular-SDE}. Then
  the triplet
  \[ (X_t^{s,x},h(t,X_t^{s,x}),\nabla h(t,X_t^{s,x})) \]
  is a rough weak solution to the singular FBSDE
  \eqref{eq:singular-SDE}--\eqref{eq:quadratic-BSDE}, and the solution is
  unique in the sense of Definition \ref{def:FBSDE-uniqueness}.
\end{proposition}

\begin{proof}
  Set
  \[
    f(t,x):=-H\bigl(t,x,h(t,x),\nabla h(t,x)\bigr).
  \]
  The boundedness and time continuity of $h$ and $\nabla h$, together with
  Assumptions \assB{} and \assPContinuous{}, imply that
  $f\in C_TL^\infty(\mathbb R^d)$. Since $h$ solves
  \[
    \left(\partial_t+\frac12\Delta+b\cdot\nabla\right)h=f,
    \qquad h(T,\cdot)=\bar h,
  \]
  Proposition~\ref{prop:ito-formula} yields
  \[
    h(t,X_t^{s,x})=\bar h(X_T^{s,x})
    +\int_t^T H\bigl(r,X_r^{s,x},h(r,X_r^{s,x}),
      \nabla h(r,X_r^{s,x})\bigr)\,dr
    -\int_t^T\nabla h(r,X_r^{s,x})\cdot dW_r.
  \]
  Thus $(X_t^{s,x},h(t,X_t^{s,x}),\nabla h(t,X_t^{s,x}))$ is a rough weak
  solution to the singular FBSDE
  \eqref{eq:singular-SDE}--\eqref{eq:quadratic-BSDE}. By Theorem
  \ref{thm:FBSDE-wellposedness} it is the unique solution.
\end{proof}

In addition to the PDE proof in
Proposition~\ref{prop:global-lipschitz-paracontrolled} of
Appendix~\ref{app:paracontrolled-wellposedness}, the same paracontrolled
well-posedness result, together with its FBSDE representation, can be obtained
through a BSDE-based argument. This proof is lengthier than its PDE
counterpart, but we believe it is interesting in its own right, as it provides
a better understanding of the fixed-point iteration in the present weak
framework when the terminal condition has additional regularity. For the
precise definition of the space
$\mathscr D_{\bar T,T,b}^{a,\theta,\rho}$, we refer the reader to
Appendix~\ref{app:paracontrolled-wellposedness}.

\begin{proposition}
  \label{prop:FBSDE-lipschitz-regularity}
  Let $4/3<a<\theta<\gamma<2-\alpha$,
  $\rho\in((\theta-1)/2,(\gamma-1)/2)$ and
  $\bar h\in\mathcal C^\gamma$.
  {{Assume that Assumptions \assAMixed{}, \assBLip{}, and
  \assPContinuous{} hold.}}
  Then there exist deterministic
  functions $u$ and $v$, with
  $(u,v)\in\mathscr D_{T,T,b}^{a,\theta,\rho}$, such that, for every
  $(s,x)\in[0,T]\times\mathbb R^d$ and every weak solution
  $(X^{s,x},Y^{s,x},Z^{s,x})$ to the singular FBSDE,
  \[
    Y_t^{s,x}=u(t,X_t^{s,x}),\qquad s\leq t\leq T,
    \quad\mathbb P^{s,x}\text{-a.s.},
  \]
  and
  \[
    Z_t^{s,x}=v(t,X_t^{s,x}),\qquad
    \mathrm dt\otimes\mathbb P^{s,x}\text{-a.s.}
  \]
  In particular, $Z^{s,x}$ admits the continuous version
  $v(\cdot,X^{s,x}_\cdot)$ and is unique in law as a random variable with
  values in $C([s,T];\mathbb R^d)$.
\end{proposition}

\begin{proof}
  Fix $(s,x)\in[0,T]\times\mathbb R^d$ and a natural-filtration rough weak
  solution,
  \[
    (\Omega^{s,x},\mathcal F^{s,x},
    (\mathcal F_t^{X^{s,x}})_{t\in[s,T]},\mathbb P^{s,x},X^{s,x},
    G^{s,x},W^{s,x}).
  \]
  For brevity, write $\mathcal F_t^{s,x}=\mathcal F_t^{X^{s,x}}$.
  Set $(Y^{s,x,0},Z^{s,x,0})\equiv(0,0)$ and define
  $(Y^{s,x,n+1},Z^{s,x,n+1})$ iteratively by
  \[
    Y_t^{s,x,n+1}
    =
    \mathbb E^{\mathbb P^{s,x}}\left[
      \bar h(X_T^{s,x})
      +\int_t^T H(r,X_r^{s,x},Y_r^{s,x,n},Z_r^{s,x,n})\,dr
      \,\middle|\,\mathcal F_t^{s,x}
    \right],
  \]
  with $Z^{s,x,n+1}$ obtained from the martingale representation theorem on
  $(\mathcal F_t^{X^{s,x}})_{t\in[s,T]}$
  ({\cite[Proposition~2.8]{perkowski_coming_2025}}).

  We construct the corresponding deterministic PDE iterates simultaneously.
  Set $u^0\equiv0$ and $v^0\equiv0$. Given $u^n$ and
  $v^n=\partial_xu^n$, let $u^{n+1}$ be the paracontrolled solution to
  \[
    \partial_tu^{n+1}+\frac12\Delta u^{n+1}+b\cdot\nabla u^{n+1}=-H^n,
    \qquad u^{n+1}(T,\cdot)=\bar h,
  \]
  where
  \[
    H^n(t,x)=H(t,x,u^n(t,x),v^n(t,x)).
  \]
  Proposition~\ref{prop:ito-formula}, applied to the present
  natural-filtration rough weak solution, gives
  \[
    u^{n+1}(t,X_t^{s,x})
    =\bar h(X_T^{s,x})
    +\int_t^T H^n(r,X_r^{s,x})\,dr
    -\int_t^T \partial_xu^{n+1}(r,X_r^{s,x})\,dW_r^{s,x}.
  \]
  Comparison with the Picard equation and uniqueness of the linear BSDE show
  inductively that, for every $n\in\mathbb N$,
  \[
    Y_t^{s,x,n}=u^n(t,X_t^{s,x}),\qquad
    Z_t^{s,x,n}=v^n(t,X_t^{s,x}),
  \]
  where the first identity holds for every $t$, $\mathbb P^{s,x}$-a.s., and
  the second holds $\mathrm dt\otimes\mathbb P^{s,x}$-a.s. Notice that the
  functions $(u^n,v^n)$ depend only on the deterministic PDE data and are
  therefore independent of $(s,x)$ and of the chosen natural-filtration rough
  weak solution.

  Following the fixed-point and convergence argument in
  {\cite[Theorem~2.1 and Corollary~2.1]{el_karoui_backward_1997}}, the sequence
  $(Y^{s,x,n},Z^{s,x,n})_{n\in\mathbb N}$ converges in
  $\mathbb S^2\times\mathbb H^2$ to the unique solution
  $(Y^{s,x},Z^{s,x})$ of the BSDE. It remains to identify the deterministic
  functions and their regularity.

  Estimate (3.12) in Proposition 3.9 of
  {\cite{cannizzaro_multidimensional_2018}}, applied on
  $[T-\bar T,T]$, gives, in our notation,
  \[
    \|(u^n,\partial_xu^n)\|_{\mathscr D_{\bar T,T,b}^{a,\theta,\rho}}
    \leq
    C\|H^{n-1}\|_{C_TL^\infty}
    \left(1+\|b\|_{C_T\mathcal C^{-\alpha}}\right)
    +\|\bar h\|_{\mathcal C^\gamma}
    +(1+\|b\|_{\mathcal X^\alpha})^2
    \left(1+\bar T^\kappa
    \|(u^n,\partial_xu^n)\|_{\mathscr D_{\bar T,T,b}^{a,\theta,\rho}}\right),
  \]
  where the $\mathscr D_{\bar T,T,b}^{a,\theta,\rho}$-norm is defined through the
  paracontrolled decomposition
  $u^n=\partial_xu^n\para J^T(b)+(u^n)^\sharp$. The corresponding time
  regularity
  $\|u^n\|_{C^{\frac{\theta}{2}}_{[T-\bar T,T]}L^\infty}$ follows by the same
  proof, using point 2 of Corollary 2.5 in
  {\cite{cannizzaro_multidimensional_2018}}. Choosing $\bar T$ small enough
  and then using the interpolation argument from
  {\cite[Lemma~2.11]{KPZreloaded}}, we obtain, for any
  $4/3<a'<a$, $\theta'<\theta$ and $\rho'<\rho$,
  \begin{equation}
    \begin{split}
    &\|u^n\|_{C_{[T-\bar T,T]}\mathcal C^{\theta'}}
      +\|u^n\|_{C^{\frac{\theta'}{2}}_{[T-\bar T,T]}L^\infty}
      +\|\partial_xu^n\|_{C^{\rho'}_{[T-\bar T,T]}L^\infty}
      +\|\partial_xu^n\|_{C_{[T-\bar T,T]}\mathcal C^{a'-1}} \\
    &\qquad\leq
    C\bar T^{\kappa'}\|H^{n-1}\|_{C_TL^\infty}
    \left(1+\|b\|_{C_T\mathcal C^{-\alpha}}\right)
    +\|\bar h\|_{\mathcal C^\gamma}
    +(1+\|b\|_{\mathcal X^\alpha})^2 .
    \end{split}
    \label{eq:un-bound}
  \end{equation}
  Since $H$ is Lipschitz,
  \[
    \|H^{n-1}\|_{C_TL^\infty}
    \leq
    \|H^0\|_{C_TL^\infty}
    +2C_H\|u^{n-1}\|_{C_{[T-\bar T,T]}\mathcal C^{\theta'}}.
  \]
  Thus, for suitable constants $A,B>0$,
  \[
    \begin{split}
    &\|u^n\|_{C_{[T-\bar T,T]}\mathcal C^{\theta'}}
      +\|u^n\|_{C^{\frac{\theta'}{2}}_{[T-\bar T,T]}L^\infty}
      +\|\partial_xu^n\|_{C^{\rho'}_{[T-\bar T,T]}L^\infty}
      +\|\partial_xu^n\|_{C_{[T-\bar T,T]}\mathcal C^{a'-1}} \\
    &\qquad\leq
      A+\bar T^{\kappa'}B
      \|u^{n-1}\|_{C_{[T-\bar T,T]}\mathcal C^{\theta'}} .
    \end{split}
  \]
  Taking $\bar T$ so small that $\bar T^{\kappa'}B<1$ and iterating gives a
  uniform bound for $(u^n,\partial_xu^n)$ on $[T-\bar T,T]$. Since $B$ is
  independent of the terminal value, the same argument extends the bound to
  the whole interval $[0,T]$.

  By the generalized Arzel{\`a}--Ascoli theorem
  {\cite[p.~236]{kelley_general_1975}}, after passing to a common subsequence,
  $u^n$ and $v^n=\partial_xu^n$ converge locally uniformly to deterministic
  functions $u$ and $v$. The uniform H{\"o}lder bounds pass to the limit, and
  therefore
  \[
    u\in C_T\mathcal C^\theta\cap C_T^{\frac{\theta}{2}}L^\infty,
    \qquad
    v\in C_T\mathcal C^{a-1}\cap C_T^\rho L^\infty .
  \]
  Since both $u^n\to u$ and $\partial_xu^n\to v$ locally uniformly, the
  standard stability of derivatives under locally uniform convergence gives
  $v=\partial_xu$.

  Since $X^{s,x}$ has continuous paths, its range on $[s,T]$ is compact
  $\mathbb P^{s,x}$-a.s. The local uniform convergence along the common
  subsequence therefore gives
  \[
    \sup_{t\in[s,T]}|u^n(t,X_t^{s,x})-u(t,X_t^{s,x})|\longrightarrow0,
    \qquad
    \sup_{t\in[s,T]}|v^n(t,X_t^{s,x})-v(t,X_t^{s,x})|\longrightarrow0,
  \]
  $\mathbb P^{s,x}$-a.s. along that subsequence. On the other hand, convergence
  in $\mathbb S^2\times\mathbb H^2$ implies convergence in probability of the
  $Y$-processes and convergence in
  $\mathrm dt\otimes\mathbb P^{s,x}$-measure of the $Z$-processes. Uniqueness
  of limits in measure therefore yields
  \[
    Y_t^{s,x}=u(t,X_t^{s,x}),\qquad
    Z_t^{s,x}=v(t,X_t^{s,x}),
  \]
  where the first identity holds for every $t$, $\mathbb P^{s,x}$-a.s., and
  the second holds $\mathrm dt\otimes\mathbb P^{s,x}$-a.s. Since $v$ and
  $X^{s,x}$ are continuous, we choose $v(t,X_t^{s,x})$ as the continuous
  version of $Z^{s,x}$. Since the deterministic subsequence and its limits
  were obtained entirely from the PDE iterates, the same functions $u$ and
  $v$ work for every $(s,x)$ and every natural-filtration rough weak
  solution.

  Finally, let $\bar u$ be the paracontrolled solution to
  \[
    \partial_t\bar u+\frac12\Delta\bar u+b\cdot\nabla\bar u
    =-H(t,x,u,v),
    \qquad
    \bar u(T,\cdot)=\bar h .
  \]
  Apply Proposition~\ref{prop:ito-formula} to $\bar u(t,X_t^{s,x})$. The
  representations already proved show that the resulting pair solves the same
  BSDE as $(Y^{s,x},Z^{s,x})$. Uniqueness of the BSDE therefore gives
  \[
    \bar u(t,X_t^{s,x})=u(t,X_t^{s,x}),
    \qquad
    \partial_x\bar u(t,X_t^{s,x})=v(t,X_t^{s,x}).
  \]
  Setting $t=s$ and using $X_s^{s,x}=x$ yields
  $u(s,x)=\bar u(s,x)$. Since $(s,x)$ was arbitrary, it holds $u=\bar u$. Since
  $v=\partial_xu$, this proves that
  $(u,\nabla u)\in\mathscr D_{T,T,b}^{a,\theta,\rho}$ is the paracontrolled
  solution of the singular HJB equation. Finally, uniqueness in law of
  $X^{s,x}$ and the
  deterministic continuous representation
  $Z^{s,x}=v(\cdot,X^{s,x}_\cdot)$ imply uniqueness in law of the continuous
  version of $Z^{s,x}$.
\end{proof}

\begin{theorem}[Nonlinear Feynman--Kac formula]
  \label{thm:nonlinear-FK}
  {Assume that Assumptions \assAMixed{} and \assB{} hold} and that $H$ is jointly
  continuous. Let $\bar h\in L^\infty(\mathbb R^d)$. Suppose that the singular
  FBSDE admits a bounded continuous universal decoupling field $u$ in the sense
  of Definition~\ref{def:FBSDE-universal-decoupling-field}. Then $u$ is a
  viscosity solution of
  \[
    \left(\partial_t+\frac12\Delta+b\cdot\nabla\right)u
    =-H(t,x,u,\nabla u),
    \qquad u(T,\cdot)=\bar h.
  \]
\end{theorem}

\begin{proof}
  We first prove the subsolution inequality. Let $\phi\in\mathcal T^b$ solve
  \[
    \left(\partial_t+\frac12\Delta+b\cdot\nabla\right)\phi=g,
  \]
  and suppose that $u-\phi$ has a local maximum equal to zero at
  $(t_0,x_0)\in[0,T)\times\mathbb R^d$. We must prove
  \[
    g(t_0,x_0)\geq
    -H\bigl(t_0,x_0,u(t_0,x_0),\nabla\phi(t_0,x_0)\bigr).
  \]
  Assume otherwise. Since $u(t_0,x_0)=\phi(t_0,x_0)$, there exists $\eta>0$
  such that
  \[
    -g(t_0,x_0)
    -H\bigl(t_0,x_0,\phi(t_0,x_0),\nabla\phi(t_0,x_0)\bigr)>2\eta.
  \]
  By the joint continuity of $H$ and the continuity of $g$, $\phi$ and
  $\nabla\phi$, after shrinking $\delta>0$ if necessary, the cylinder
  \[
    Q=[t_0,t_0+\delta]\times\overline B_\delta(x_0)
    \subset[0,T]\times\mathbb R^d
  \]
  satisfies
  \[
    \phi(t,x)\geq u(t,x)
  \]
  and
  \begin{equation}
    -g(t,x)-H\bigl(t,x,\phi(t,x),\nabla\phi(t,x)\bigr)\geq\eta
    \qquad\text{for all }(t,x)\in Q.
    \label{eq:nonlinear-FK-contact-gap}
  \end{equation}

  Let $(X,Y,Z)$ be a natural-filtration rough weak solution starting from
  $(t_0,x_0)$, with associated Brownian motion $W$ and probability measure
  $\mathbb P^{t_0,x_0}$. Define
  \[
    \tau:=\inf\{t\geq t_0:|X_t-x_0|\geq\delta\}\wedge(t_0+\delta).
  \]
  Since $u$ is a universal decoupling field,
  \[
    Y_{t\wedge\tau}=u(t\wedge\tau,X_{t\wedge\tau}),
    \qquad t_0\leq t\leq t_0+\delta.
  \]
  On this interval, set
  \[
    \Delta Y_t:=\phi(t\wedge\tau,X_{t\wedge\tau})-Y_{t\wedge\tau},
    \qquad
    \Delta Z_t:=\mathbf{1}_{\{t\leq\tau\}}
      \bigl(\nabla\phi(t,X_t)-Z_t\bigr).
  \]
  Due to the local touching property of $\phi$ and $u$ being the coupling field, we have
  \begin{equation}
    \Delta Y_t\geq0,\qquad \Delta Y_{t_0}=0.
    \label{eq:nonlinear-FK-touching-process}
  \end{equation}
  We introduce
  \[
    R_r:=-g(r,X_r)
      -H\bigl(r,X_r,\phi(r,X_r),\nabla\phi(r,X_r)\bigr),
  \]
  which by \eqref{eq:nonlinear-FK-contact-gap}, it holds $R_r\geq\eta$ for $r\leq\tau$.
  Proposition~\ref{prop:ito-formula} and the stopped BSDE yield, for
  $t\in[t_0,t_0+\delta]$,
  \begin{align}
    \Delta Y_t
    &=\Delta Y_{t_0+\delta}+\int_t^{t_0+\delta}\mathbf{1}_{\{r\leq\tau\}}
      \bigl[-g(r,X_r)-H(r,X_r,Y_r,Z_r)\bigr]\,dr
      -\int_t^{t_0+\delta}\Delta Z_r\,dW_r \label{eq:nonlinear-FK-difference-BSDE}\\
    &=\Delta Y_{t_0+\delta}+\int_t^{t_0+\delta}\mathbf{1}_{\{r\leq\tau\}}
      \Bigl[R_r+H\bigl(r,X_r,\phi(r,X_r),\nabla\phi(r,X_r)\bigr)
      -H(r,X_r,Y_r,Z_r)\Bigr]\,dr
      -\int_t^{t_0+\delta}\Delta Z_r\,dW_r.\notag
  \end{align}
  For $r\leq\tau$, we can linearize the difference in the second line as
  \begin{align}
    H\bigl(r,X_r,\phi(r,X_r),\nabla\phi(r,X_r)\bigr)
      -H(r,X_r,Y_r,Z_r)=a_r\Delta Y_r+c_r\cdot\Delta Z_r,
    \label{eq:nonlinear-FK-linearization}
  \end{align}
  where due to the (local) Lipschitz continuity of $H$ in $(y,z)$, and the boundedness of $Y$, $\phi(r,X_r)$ and
  $\nabla\phi(r,X_r)$ on the stopped interval, we have the following bounds on $a_r$ and $c_r$:
  \[
    |a_r|\leq C,
    \qquad
    |c_r|\leq C(1+|Z_r|).
  \]
  Consequently, \eqref{eq:nonlinear-FK-difference-BSDE} becomes
  \begin{align}
    \Delta Y_t
    =\Delta Y_{t_0+\delta}
      +\int_t^{t_0+\delta}\mathbf{1}_{\{r\leq\tau\}}
        \bigl(R_r+a_r\Delta Y_r+c_r\cdot\Delta Z_r\bigr)\,dr-\int_t^{t_0+\delta}\Delta Z_r\,dW_r.
    \label{eq:nonlinear-FK-linear-BSDE}
  \end{align}
  Theorem~\ref{thm:quadratic-BSDE-apriori-estimates} gives that $Z\cdot W$ is
  a BMO martingale. The bound on $c$ therefore implies that $\int_{t_0}^{\cdot}\mathbf{1}_{\{r\leq\tau\}}c_r\,dW_r$
  is also BMO. Hence its stochastic exponential is a uniformly integrable
  martingale and defines an equivalent probability measure $\mathbb Q$ on
  $\mathcal F_{t_0+\delta}^{t_0,x_0}$ by
  \[
    \frac{d\mathbb Q}{d\mathbb P^{t_0,x_0}}
    :=\mathcal E\left(
      \int_{t_0}^{\cdot}\mathbf{1}_{\{r\leq\tau\}}c_r\,dW_r
    \right)_{t_0+\delta}.
  \]
  By Girsanov's theorem, $W_t^{\mathbb Q}:=W_t-
    \int_{t_0}^t\mathbf{1}_{\{r\leq\tau\}}c_r\,dr$
  becomes a Brownian motion under $\mathbb Q$, and the term involving
  $c_r\cdot\Delta Z_r$ disappears from
  \eqref{eq:nonlinear-FK-linear-BSDE}. With
  \[
    \Gamma_r:=\exp\left(
      \int_{t_0}^r\mathbf{1}_{\{q\leq\tau\}}a_q\,dq
    \right),
  \]
  the representation formula for the resulting linear BSDE gives
  \[
    \Delta Y_{t_0}
    =\mathbb E^{\mathbb Q}\left[
      \left.
      \Gamma_{t_0+\delta}\Delta Y_{t_0+\delta}
      +\int_{t_0}^{t_0+\delta}
        \Gamma_r\mathbf{1}_{\{r\leq\tau\}}R_r\,dr
      \right|\mathcal F_{t_0}^{t_0,x_0}
    \right].
  \]
  Since $\Delta Y_{t_0+\delta}\geq0$, $R_r\geq\eta$ before $\tau$, and
  $\Gamma_r\geq e^{-C\delta}$, it follows that
  \[
    \Delta Y_{t_0}\geq
    \eta e^{-C\delta}\,
    \mathbb E^{\mathbb Q}\left[
      \left.\tau-t_0\right|\mathcal F_{t_0}^{t_0,x_0}
    \right].
  \]
  The continuity of $X$ and the identity $X_{t_0}=x_0$ imply
  $\tau>t_0$, $\mathbb P^{t_0,x_0}$-almost surely, and hence also
  $\mathbb Q$-almost surely. Thus $\Delta Y_{t_0}>0$, contradicting
  \eqref{eq:nonlinear-FK-touching-process}. This proves the subsolution
  inequality. \\
  The supersolution inequality follows in the same way. \\
  Finally, taking the initial condition $(T,x)$ in the universal decoupling
  identity gives
  \[
    u(T,x)=Y_T^{T,x}=\bar h(x).
  \]
  Therefore $u$ is a viscosity solution of
  \eqref{eq:section4-singular-HJB}.
\end{proof}

By combining Theorem~\ref{thm:nonlinear-FK} and
Proposition~\ref{prop:converse-FK}, we immediately obtain the following
result, which can be seen as the counterpart of
Theorem~\ref{paravicosity} under Assumptions \assB{} and \assPContinuous{},
together with joint continuity of $H$, instead of \AssumptionC{}.

\begin{corollary}
  \label{cor:paracontrolled-viscosity-B}
  {Under Assumptions \assAMixed{}, \assB{}, and \assPContinuous{},} assume
  additionally that $H$ is jointly continuous. Then every paracontrolled solution of
  \eqref{eq:section4-singular-HJB} is a viscosity solution.
\end{corollary}

Stability results for quadratic BSDEs allow us to treat singular HJB equations
with terminal conditions that are not sufficiently regular for the
paracontrolled theory of
Theorem~\ref{thm:quadratic-paracontrolled-global}. Although this regime is not
covered by the paracontrolled solution theory, the equation can still be
understood as admitting a unique viscosity solution through its FBSDE
representation. This leads to the following result.

\begin{theorem}[Quadratic case with continuous bounded terminal condition]
  \label{thm:FBSDE-quadratic-continuous-terminal}
  {Assume that Assumptions \assAMixed{}, \assB{}, \assP{}, and
  \assPContinuous{} hold.} Let
  $\bar h\in C_b(\mathbb R^d)$. Then there exists a deterministic function
  \[
    u\in C_b([0,T]\times\mathbb R^d),
  \]
  independent of $(s,x)$ and of the chosen natural-filtration rough weak
  solution, such that, for every $(s,x)$ and every corresponding solution
  $(Y^{s,x},Z^{s,x})$,
  \[
    Y_t^{s,x}=u(t,X_t^{s,x}),\qquad s\leq t\leq T,
    \quad\mathbb P^{s,x}\text{-a.s.}
  \]
  In particular, $u(T,\cdot)=\bar h$. Moreover, if
  $\bar h\in\operatorname{BUC}(\mathbb R^d)$, then
  $u\in\operatorname{BUC}([0,T]\times\mathbb R^d)$.
  In either case, $u$ is the unique bounded viscosity solution to the singular HJB
  equation \eqref{eq:section4-singular-HJB}.
\end{theorem}

\begin{proof}
  As in the proof of Proposition~\ref{prop:control-bounded-terminal}, choose bounded smooth functions
  $(\bar h_n^-)_{n\in\mathbb N}$ and
  $(\bar h_n^+)_{n\in\mathbb N}$ such that
  \[
    \bar h_1^-\leq\cdots\leq\bar h_n^-\leq\bar h
    \leq\bar h_n^+\leq\cdots\leq\bar h_1^+,
  \]
  and
  \[
    \bar h_n^-(x)\uparrow\bar h(x),\qquad
    \bar h_n^+(x)\downarrow\bar h(x),
    \qquad x\in\mathbb R^d.
  \]
  The approximations may be chosen to converge locally uniformly; no global
  uniform approximation of the bounded continuous function $\bar h$ is
  needed.

  Let $(Y^{s,x,n,\pm},Z^{s,x,n,\pm})$ denote the BSDE solutions with terminal
  conditions $\bar h_n^\pm(X_T^{s,x})$, and set
  \[
    u_n^\pm(t,x):=Y_t^{t,x,n,\pm}.
  \]
  Theorem~\ref{thm:quadratic-paracontrolled-global}, followed by
  Proposition~\ref{prop:converse-FK}, shows that each $u_n^\pm$ is a
  continuous deterministic function, independent of the chosen
  natural-filtration rough weak solution. The comparison theorem for
  quadratic BSDEs gives
  \[
    u_n^-\leq u_{n+1}^-\leq u\leq u_{n+1}^+\leq u_n^+.
  \]

  Fix $(t,x)$ and a natural-filtration rough weak solution starting from
  $(t,x)$. Since
  \[
    \bar h_n^\pm(X_T^{t,x})\longrightarrow\bar h(X_T^{t,x})
    \quad\mathbb P^{t,x}\text{-a.s.},
  \]
  and these terminal variables are uniformly bounded, an application of
  {\cite[Theorem~7.3.4]{zhang_backward_2017}} with the generator fixed shows
  that
  \[
    (Y^{t,x,n,\pm},Z^{t,x,n,\pm})\longrightarrow
    (Y^{t,x},Z^{t,x})
  \]
  in $\mathbb S^p\times\mathbb H^p$ for every $p\geq1$. In particular,
  \[
    u_n^-(t,x)\uparrow u(t,x),\qquad
    u_n^+(t,x)\downarrow u(t,x).
  \]
  Thus $u$, being an increasing pointwise limit of continuous functions, is
  lower semicontinuous. At the same time, as a decreasing pointwise limit of
  continuous functions, it is upper semicontinuous. Hence $u$ is continuous;
  its boundedness follows from
  Theorem~\ref{thm:FBSDE-wellposedness-bounded}.

  If $\bar h\in\operatorname{BUC}(\mathbb R^d)$, the approximations above
  can be chosen such that
  $\|\bar h_n^\pm-\bar h\|_\infty\to0$. Let $L_y$ be the Lipschitz constant
  of $H$ in $y$. Fix $(s,x)$, one of the signs $\pm$, and set
  \[
    \delta_n^\pm:=\|\bar h_n^\pm-\bar h\|_\infty,\qquad
    q_r:=\delta_n^\pm e^{L_y(T-r)},\qquad r\in[s,T].
  \]
  Write
  $\xi_n^\pm:=\bar h_n^\pm(X_T^{s,x})$ and
  $\xi:=\bar h(X_T^{s,x})$. Then
  $\xi_n^\pm\leq\xi+\delta_n^\pm$. Moreover,
  $(Y^{s,x}+q,Z^{s,x})$ solves the BSDE with terminal condition
  $\xi+\delta_n^\pm$ and generator
  \[
    \widetilde H_r(y,z)
    :=H(r,X_r^{s,x},y-q_r,z)+L_yq_r.
  \]
  The Lipschitz property in $y$ gives
  $\widetilde H_r(y,z)\geq H(r,X_r^{s,x},y,z)$. Hence the quadratic
  comparison theorem
  {\cite[Theorem~7.3.1]{zhang_backward_2017}} yields
  \[
    Y_r^{s,x,n,\pm}\leq Y_r^{s,x}+q_r.
  \]
  Similarly, $\xi-\delta_n^\pm\leq\xi_n^\pm$, and
  $(Y^{s,x}-q,Z^{s,x})$ solves the BSDE with terminal condition
  $\xi-\delta_n^\pm$ and generator
  \[
    \widehat H_r(y,z)
    :=H(r,X_r^{s,x},y+q_r,z)-L_yq_r
    \leq H(r,X_r^{s,x},y,z).
  \]
  Another application of comparison gives
  \[
    Y_r^{s,x}-q_r\leq Y_r^{s,x,n,\pm}.
  \]
  Taking $r=s$ and then the supremum over $(s,x)$ therefore gives
  \[
    \|u_n^\pm-u\|_\infty
    \leq e^{L_yT}\|\bar h_n^\pm-\bar h\|_\infty
    \longrightarrow0.
  \]
  Each $u_n^\pm$ is uniformly continuous by
  Theorem~\ref{thm:quadratic-paracontrolled-global}. Hence their
  uniform limit $u$ is uniformly continuous.

  Finally, fix $(s,x)$ and a corresponding forward rough weak solution.
  Comparison and the Markovian representations of the approximating BSDEs
  yield, for every $t\in[s,T]$,
  \[
    u_n^-(t,X_t^{s,x})\leq Y_t^{s,x}\leq
    u_n^+(t,X_t^{s,x}),\qquad\mathbb P^{s,x}\text{-a.s.}
  \]
  Letting $n\to\infty$ gives
  $Y_t^{s,x}=u(t,X_t^{s,x})$ for every fixed $t$. Taking a common null set
  over rational times and using the continuity of $Y^{s,x}$, $u$ and
  $X^{s,x}$ proves the identity simultaneously for all $t\in[s,T]$.
  Taking $s=T$ gives $u(T,\cdot)=\bar h$.

  Theorem~\ref{thm:nonlinear-FK} now shows that $u$ is a viscosity solution
  to \eqref{eq:section4-singular-HJB}. Uniqueness follows by the same argument
  as in Proposition~\ref{prop:control-bounded-terminal}. Indeed, if $\widetilde u$ is any other bounded viscosity
  solution with terminal condition $\bar h$,
  Proposition~\ref{smoothcomparison}, applied to the paracontrolled
  solutions $u_n^-$ and $u_n^+$, gives
  \[
    u_n^-\leq\widetilde u\leq u_n^+.
  \]
  Letting $n\to\infty$ and using the pointwise monotone convergence of
  $u_n^\pm$ to $u$ yields $\widetilde u=u$.
\end{proof}

\appendix
\section{Appendix}

\subsection{The linear equation}

Let us give some heuristics on paracontrolled solutions; a reader interested in a more technical proof can consult the next Section. For an enhanced drift $b$, we want to make sense of
\begin{equation}
	\left( \partial_t + \frac{1}{2} \partial_{x x} + b \partial_x \right)
	\varphi = g, \quad \varphi (T, \cdummy) = \overline{\varphi} (\cdummy)
	\label{appendixpde},
\end{equation}
where for simplicity we let $d = 1$ (the higher dimensional case being
similar). Using the same notation as in Definition \ref{enhanceddrift} for the
Duhamel operator, the solution $\varphi$ formally satisfies
\begin{equation}
	\varphi = e^{(T - t) \frac{\Delta}{2}} \overline{\varphi} + J^T (b
	\cdummy \partial_x \varphi - g) = e^{(T - t) \frac{\Delta}{2}}
	\overline{\varphi} + J^T (b \para \partial_x \varphi + b \odot
	\partial_x \varphi + b \arap \partial_x \varphi - g) . \label{fidecomp}
\end{equation}
By power counting, the expected regularity (in space) of $\varphi$ should be
$\mathcal{C}^{2 - \alpha}$, so the resonant term $b \odot \partial_x \varphi$
is in principle not well defined, since $- \alpha + (1 - \alpha) = 1 - 2
	\alpha \leqslant 0$ and we cannot apply Proposition \ref{paraestimates}.

However, if sense could be made of it, the worst term on the right-hand side
of \eqref{fidecomp} would be the one coming from $b \arap \partial_x \varphi
	\in \mathcal{C}^{- \alpha}$, i.e.
\[ \varphi = J^T (b \arap \partial_x \varphi) + \varphi_1, \quad
	\tmop{with} \varphi_1 \in \mathcal{C} ^{3 - 2 \alpha} . \]
Using that $J^T$ ``commutes'' with the paraproduct (up to terms of
better regularity, see Lemma 3.4 in {\cite{cannizzaro_multidimensional_2018}}), we can write
\begin{equation}
	\begin{array}{lll}
		\varphi & = & J^T (b) \arap \partial_x \varphi + \{
		(J^T (b \arap \partial_x \varphi) - J^T (b) \arap
		\partial_x \varphi) + \varphi_1 \}                                         \\
		        & = & J^T (b) \arap \partial_x \varphi + \varphi_2, \quad
		\tmop{with} \varphi_2 \in \mathcal{C}^{3 - 2 \alpha} .
	\end{array} \label{paraanzatz}
\end{equation}
If $\varphi$ is assumed to have this special structure (referred to usually as
\tmtextit{paracontrolled ansatz}), we can define the ill posed term $b \odot
	\partial_x \varphi$ as follows. First, notice that
\[ \begin{array}{lll}
		b \odot \partial_x \varphi & = & b \odot \partial_x (\partial_x \varphi
		\para J^T (b) + \varphi_2)                                                               \\
		                           & = & b \odot (\partial_{x x} \varphi \para J^T (b) +
		\partial_x \varphi \para \partial_x J^T (b) + \partial_x
		\varphi_2)                                                                                         \\
		                           & = & b \odot (\partial_x \varphi \para \partial_x J^T (b)) +
		b \odot (\partial_{x x} \varphi \para J^T (b) + \partial_x
		\varphi_2),
	\end{array} \]
where in the last sum, the second term is well defined (and has regularity $2
	- 3 \alpha > 0$, since $\alpha \in (1 / 2, 2 / 3)$).

For the first term, i.e. $b \odot (\partial_x \varphi \para J^T
	(\partial_x b))$ ($\partial_x$ and $J^T$ commute as Fourier
multipliers for each fixed time), we require the following commutator
estimate, a proof of which can be consulted in Lemma 2.4 of
	{\cite{paradistrib}}.

\begin{proposition}
	\tmtextbf{(Commutator estimate)}\label{commutator} Let $\lambda, \beta,
		\gamma \in \mathbb{R}$ be such that $\lambda \in (0, 1)$, $\lambda + \beta +
		\gamma > 0$ and $\beta + \gamma < 0$. Then, for smooth $f, g$and $h$, the
	operator
	\[ C (f, g, h) = (f \para g) \odot h - f (g \odot h) \]
	satisfies
	\[ \| C (f, g, h) \|_{\mathcal{C}^{\lambda + \beta + \gamma}} \leqslant \| f
		\|_{\mathcal{C}^{\lambda}} \| g \|_{\mathcal{C}^{\beta}} \| h
		\|_{\mathcal{C}^{\gamma}},  \]
	and therefore admits a unique extension to a bounded trilinear operator on
	$\mathcal{C}^{\lambda} \times \mathcal{C}^{\beta} \times
		\mathcal{C}^{\gamma}$.
\end{proposition}

The hypothesis of Proposition \ref{commutator} are satisfied if we choose
\[ \left\{ \begin{array}{l}
		f = \partial_x \varphi, g = J^T (\partial_x b) \infixand h =
		b, \\
		\lambda = 1 - \alpha, \beta = 1 - \alpha \infixand \gamma = - \alpha .
	\end{array} \right. \]
Remembering Definition \ref{enhanceddrift}, we can understand the second term
in the commutator as
\[ '' \partial_x \varphi (J^T (\partial_x b) \odot b)'' \equiv
	\partial_x \varphi \cdummy B \in \mathcal{C}^{2 - 3 \alpha}, \]
and then define
\[ \begin{array}{lll}
		b \odot (\partial_x \varphi \para J^T (\partial_x b)) & \assign
		                                                                & C (\partial_x \varphi, J^T (\partial_x b), b) + \partial_x
		\varphi \cdummy B.
	\end{array} \]
Putting everything together,
\begin{equation}
	\begin{array}{lllll}
		\varphi = \partial_x \varphi \para J^T (b) + \varphi_2 &
		\Longrightarrow                                                  & b \cdummy \partial_x \varphi & : = & b \para \partial_x
		\varphi + b \arap \partial_x \varphi + b \odot \partial_x \varphi_2                                                                                                                      \\
		                                                                 &                              &     &                    & {{}+ b\odot\bigl(\partial_{xx}\varphi\para J^T(b)\bigr)}
		+ C (\partial_x \varphi, J^T (\partial_x b), b) +
		\partial_x \varphi \cdummy B,
	\end{array} \label{extproduct}
\end{equation}
We summarise the discussion in the following

\begin{proposition}
	\label{productprop}Consider the Banach space
	\[ \mathcal{D}_T^{\alpha} \assign \left\{ \varphi \in C_T \mathcal{C}^{2 -
			\alpha} : \exists \varphi_2 \in C_T \mathcal{C}^{3 - 2 \alpha} \quad
		\tmop{with} \quad \varphi = \partial_x \varphi \para J^T (b) +
		\varphi_2 \right\}, \]
	with the norm
	\[ \| \varphi \|_{\mathcal{D}_T^{\alpha}} \assign \| \varphi \|_{C_T
		\mathcal{C}^{2 - \alpha}} + \| \varphi_2 \|_{C_T \mathcal{C}^{3 - 2
			\alpha}} .  \]
	Then, there exists a linear map
	\[ b \cdummy  \partial_x : \mathcal{D}_T^{\alpha} \longrightarrow C_T
		\mathcal{C}^{- \alpha}, \]
	which coincides with the usual product when $b$ is smooth, and such that
	\[ \| b \cdummy \partial_x \varphi - \partial_x \varphi \para b \|_{C_T
		\mathcal{C}^{1 - 2 \alpha}} \leqslant (1 + \| (b, B)
		\|^2_{\mathcal{X}^{\alpha}}) \| \varphi \|_{\mathcal{D}^{\alpha}_T} . \]
	In particular,
	\[ \| b \cdummy \partial_x \varphi \|_{C_T \mathcal{C}^{- \alpha}} \leqslant
		(1 + \| (b, B) \|^2_{\mathcal{X}^{\alpha}}) \| \varphi
		\|_{\mathcal{D}^{\alpha}_T} . \]
	\begin{proof}
		Follows from \eqref{extproduct} (remember that the resonant product is the
		only ill-defined term), Propositions \ref{paraestimates} and
		\ref{commutator}, and Schauder estimates ({\cite{KPZreloaded}}, Lemma
		2.9).
	\end{proof}
\end{proposition}

At this point, one would like to set up a Picard iteration for the map $\Phi :
	\mathcal{D}^{\alpha}_T \rightarrow \mathcal{D}_T^{\alpha}$ given by
\[ \Phi (\varphi) = e^{(T - t) \frac{\Delta}{2}} \overline{\varphi} +
	J^T (b \cdummy \partial_x \varphi - g) . \]
If we had been dealing with the elliptic equation (replacing every instance of
$J^T$ by $(\Delta / 2)^{- 1}$) we could conclude in this way, but in
the parabolic case one last technical step is necessary; a close inspection
of ({\cite{cannizzaro_multidimensional_2018}}, Lemma 3.4) reveals that \eqref{paraanzatz} only holds as
long as $\partial_x \varphi$ has some space-time H{\"o}lder regularity.

To get around this problem, one can either work with ``parabolic''
paracontrolled spaces (as done in {\cite{cannizzaro_multidimensional_2018}}), or modify the paraproduct
so that it is better behaved with respect to the heat operator (as in
	{\cite{KPZreloaded}}). Once this detail is taken care of, the main idea
remains the same, and one can prove

\begin{theorem}
	$\left( \cite{cannizzaro_multidimensional_2018}, \tmop{Proposition} 3.9 \right)$ For any $4 / 3 <
		\gamma < \theta < 2 - \alpha$, there exists a Banach space $(\mathcal{D}_T,
		\| \cdummy \|_{\mathcal{D}_T})$ such that the following hold:
	\begin{enumerate}
		\item $\mathcal{D}_T \subset C_T \mathcal{C}^{\gamma} \cap
			      C^{\frac{\gamma}{2}} L^{\infty}$.

		\item There is linear operator $b \cdummy \partial_x : \mathcal{D}_T
			      \longrightarrow C_T \mathcal{C}^{- \alpha}$ satisfying
		      \[ \max \{ \| b \cdummy \partial_x \varphi \|_{C_T \mathcal{C}^{-
				      \alpha}}, \| b \cdummy \partial_x \varphi - \partial_x \varphi \para b
			      \|_{C_T \mathcal{C}^{2 \gamma - 3}} \} \leqslant (1 + \| (b, B)
			      \|^2_{\mathcal{X}^{\alpha}}) \| \varphi \|_{\mathcal{D}_T} . \]
		\item If $\overline{\varphi} \in \mathcal{C}^{\theta}$ and $g \in C_T
			      L^{\infty}$, the map $\Phi : \mathcal{D}_T \longrightarrow \mathcal{D}_T $
		      given by
		      \[ \Phi (\varphi) = e^{(T - t) \frac{\Delta}{2}} \overline{\varphi} +
			      J^T (b \cdummy \partial_x \varphi - g), \]
		      has a unique fixed point, which coincides with the classical solution of
		      \eqref{appendixpde} when $b$, $g$ and $\overline{\varphi}$ are smooth.
	\end{enumerate}
\end{theorem}

It is also possible to see that the solution map is continuous with respect to
the input data (i.e. $b$, $g$ and $\overline{\varphi}$).

\begin{theorem}
	$\left( \cite{cannizzaro_multidimensional_2018}, \tmop{Theorem} 3.10 \right)$\label{solmap} Let $4 / 3
		< \gamma < \theta < 2 - \alpha$ and consider the map which assigns to each
	$\left( \overline{u}, f, \eta \right) \in \mathcal{C}^2(\mathbb R^d)
		\times C_T\mathcal{C}^2(\mathbb R^d)
		\times C_T\mathcal{C}^{\infty}(\mathbb R^d;\mathbb R^d)$ the
	(classical) solution $S_c
		\left( \overline{u}, f, \eta \right) = u \in C ^{1, 2} ([0, T] \times
		\mathbb{R}^d)$ of
	\[ \left( \partial_t + \frac{\Delta}{2} + \eta\cdot\nabla \right) u
		= f, \quad u (T, \cdummy) = \overline{u} . \]
	Then, there exists a locally Lipschitz continuous map $S_r :
		\mathcal{C}^{\theta}(\mathbb R^d) \times C_T L^{\infty}(\mathbb R^d)
		\times \mathcal{X}^{\alpha} \rightarrow C_T
		\mathcal{C}^{\gamma}(\mathbb R^d)$ such that, for every $\left(
		\overline{u}, f, \eta \right) \in \mathcal{C}^2(\mathbb R^d)
		\times C_T\mathcal{C}^2(\mathbb R^d)
		\times C_T\mathcal{C}^{\infty}(\mathbb R^d;\mathbb R^d)$,
	\[ S_c \left( \overline{u}, f, \eta \right) = S_r \left( \overline{u}, f,
		\left(\eta,\left(J^T(\partial_{x_j}\eta^i)\odot\eta^j\right)_{i,j=1}^d\right)
		\right) . \]
\end{theorem}

\subsection{Paracontrolled well-posedness for the singular HJB equation}
\label{app:paracontrolled-wellposedness}

We collect here the paracontrolled well-posedness results for the singular HJB
equation with regular terminal condition under different growth assumptions
on the generator. We use a slight variant of the paracontrolled space
introduced in {\cite{cannizzaro_multidimensional_2018}}.

Fix $\bar T\in(0,T]$, set $I_{\bar T}:=[T-\bar T,T]$, and let
$4/3<a<\theta<\gamma<2-\alpha$ and
$\rho\in((\theta-1)/2,(\gamma-1)/2)$. We denote by
$\mathscr D_{\bar T,T,b}^{a,\theta,\rho}$ the space of pairs $(u,u')$ such that
\[
  u\in C_{I_{\bar T}}\mathcal C^\theta,
  \qquad
  \nabla u\in C^\rho_{I_{\bar T}}L^\infty,
  \qquad
  u'\in C_{I_{\bar T}}\mathcal C^{a-1},
\]
and
\[
  u^\sharp:=u-u'\para J^T(b)
  \in C_{a-1,I_{\bar T}}\mathcal C^{2a-1},
\]
where
\[
  \|u^\sharp\|_{C_{a-1,I_{\bar T}}\mathcal C^{2a-1}}
  :=\sup_{t\in I_{\bar T}}(T-t)^{(a-1)/2}
    \|u^\sharp(t)\|_{\mathcal C^{2a-1}}.
\]
We equip this space with the norm
\begin{align*}
  \|(u,u')\|_{\mathscr D_{\bar T,T,b}^{a,\theta,\rho}}:={}&
  \|u\|_{C_{I_{\bar T}}\mathcal C^\theta}
  +\|\nabla u\|_{C^\rho_{I_{\bar T}}L^\infty}
  +\|u'\|_{C_{I_{\bar T}}\mathcal C^{a-1}}
  +\|u^\sharp\|_{C_{a-1,I_{\bar T}}\mathcal C^{2a-1}}.
\end{align*}
This differs slightly from the definition in
{\cite{cannizzaro_multidimensional_2018}} only in that we do not subtract the
contribution of the forcing term in the definition of $u^\sharp$, since the
forcing terms considered below are not singular.

\begin{proposition}[Globally Lipschitz generators]
  \label{prop:global-lipschitz-paracontrolled}
  Let $4/3<a<\theta<\gamma<2-\alpha$,
  $\rho\in((\theta-1)/2,(\gamma-1)/2)$ and
  $\bar h\in\mathcal C^\gamma$.
  {Suppose Assumptions \assBLip{} and \assPContinuous{} hold.}
  Then the singular HJB equation
  \eqref{eq:section4-singular-HJB} has a unique global paracontrolled solution
  $(u,u')\in\mathscr D_{T,T,b}^{a,\theta,\rho}$.
\end{proposition}

\begin{proof}
  For $F\in C_{I_{\bar T}}L^\infty$, denote by
  $\mathcal S_{\bar h}(F)$ the unique solution of the linear paracontrolled
  equation
  \begin{equation}
    \left(\partial_t+\frac12\Delta+b\cdot\nabla\right)\widetilde u=F,
    \qquad
    \widetilde u(T,\cdot)=\bar h.
    \label{eq:linear-paracontrolled-equation}
  \end{equation}
  The well-posedness of this equation is given by
  {\cite[Theorem~3.10]{cannizzaro_multidimensional_2018}} in the
  paracontrolled space defined above.

  Proposition~3.9 of
  {\cite{cannizzaro_multidimensional_2018}}, applied to the solution
  $(\widetilde u,\widetilde u')
  =(\mathcal S_{\bar h}(F),\nabla\mathcal S_{\bar h}(F))$ of
  \eqref{eq:linear-paracontrolled-equation}, gives some $\kappa>0$,
  depending only on $a,\theta,\rho$ and $\gamma$, such that, for every
  $F,\widetilde F\in C_{I_{\bar T}}L^\infty$, it holds
  \begin{align}
    \|(\mathcal S_{\bar h}(F)-\mathcal S_{\bar h}(\widetilde F),\nabla\mathcal S_{\bar h}(F)-\nabla\mathcal S_{\bar h}(\widetilde F))\|_{\mathscr D_{\bar T,T,b}^{a,\theta,\rho}}&\leq C\bar T^\kappa(1+\|b\|_{\mathcal X^\alpha})^2\|F-\widetilde F\|_{C_{I_{\bar T}}L^\infty}.
    \label{eq:linear-paracontrolled-short-time}
  \end{align}
  \begin{align*}
    \|(\mathcal S_{\bar h}(F),\nabla\mathcal S_{\bar h}(F))\|
      _{\mathscr D_{\bar T,T,b}^{a,\theta,\rho}}
    \leq{}& C\|F\|_{C_{I_{\bar T}}L^\infty}
      \left(1+\|b\|_{C_T\mathcal C^{-\alpha}}\right)
      +\|\bar h\|_{\mathcal C^\gamma}\\
    &+C\bar T^\kappa(1+\|b\|_{\mathcal X^\alpha})^2
      \left(1+
      \|(\mathcal S_{\bar h}(F),\nabla\mathcal S_{\bar h}(F))\|
        _{\mathscr D_{\bar T,T,b}^{a,\theta,\rho}}\right).
  \end{align*}

  Choose $\bar T$ small enough so that
  $C\bar T^\kappa(1+\|b\|_{\mathcal X^\alpha})^2\leq1/2$. The last term in
  the second estimate can then be absorbed into the left-hand side, which
  gives\footnote{Our definition of $u^\sharp$ differs from that in
  \cite{cannizzaro_multidimensional_2018} only by the term $J^T(F)$. By the
  Schauder estimate in
  \cite[Corollary~2.5]{cannizzaro_multidimensional_2018},
  $\|J^T(F)\|_{C_{a-1,I_{\bar T}}\mathcal C^{2a-1}}
  \lesssim \bar T^{(2-a)/2}\|F\|_{C_{I_{\bar T}}L^\infty}$, so this
  difference is harmless for bounded forcing terms.}
  \begin{align*}
    \|(\mathcal S_{\bar h}(F),\nabla\mathcal S_{\bar h}(F))\|
      _{\mathscr D_{\bar T,T,b}^{a,\theta,\rho}}
    \leq2C\|F\|_{C_{I_{\bar T}}L^\infty}
      \left(1+\|b\|_{C_T\mathcal C^{-\alpha}}\right)
      +2\|\bar h\|_{\mathcal C^\gamma}
    +2C\bar T^\kappa(1+\|b\|_{\mathcal X^\alpha})^2.
  \end{align*}

  We carry out the fixed point on the closed affine subspace
  \[
    \left\{(u,u')\in\mathscr D_{\bar T,T,b}^{a,\theta,\rho}:
      u'=\nabla u,\ (u(T),u'(T))=(\bar h,\nabla\bar h)\right\}.
  \]
  For $(u,u')$ in this space, define
  \[
    \Phi(u,u'):=\left(
      \mathcal S_{\bar h}(-H(\cdot,\cdot,u,u')),
      \nabla\mathcal S_{\bar h}(-H(\cdot,\cdot,u,u'))
    \right).
  \]
  The continuity assumption on $H$ and \AssumptionBLip{} imply that
  $H(\cdot,\cdot,u,u')\in C_{I_{\bar T}}L^\infty$. Hence the linear theory
  shows that $\Phi$ maps this affine space into itself. If $L_H$ denotes the
  global Lipschitz constant of $H$, then
  \eqref{eq:linear-paracontrolled-short-time} gives
  \[
    \|\Phi(u,u')-\Phi(v,v')\|
      _{\mathscr D_{\bar T,T,b}^{a,\theta,\rho}}
    \leq C\bar T^\kappa(1+\|b\|_{\mathcal X^\alpha})^2L_H
      \|(u,u')-(v,v')\|_{\mathscr D_{\bar T,T,b}^{a,\theta,\rho}}.
  \]
  Shrinking $\bar T$ further, if necessary, so that
  $C\bar T^\kappa(1+\|b\|_{\mathcal X^\alpha})^2L_H<1$, the map $\Phi$ is a
  strict contraction on this space.
  Banach's fixed-point theorem gives a unique paracontrolled solution on
  $[T-\bar T,T]$.

  Notice that the choice of the interval length $\bar T$ does not depend on
  the terminal condition. As explained in the proof of \cite[Theorem 3.8]{kremp_multidimensional_2022}, we may iterate the above argument with the
  same $\bar T$ on $[T-2\bar T,T-\bar T]$, and then iteratively on consecutive
  intervals until reaching time $0$. This gives the unique global
  paracontrolled solution.
\end{proof}

\begin{proposition}[Locally Lipschitz generators]
  \label{prop:local-lipschitz-paracontrolled}
  Let $4/3<a<\theta<\gamma<2-\alpha$,
  $\rho\in((\theta-1)/2,(\gamma-1)/2)$ and
  $\bar h\in\mathcal C^\gamma$. Suppose that
  $H$ is locally Lipschitz in $(y,z)$, uniformly in $(t,x)$, and bounded on
  bounded subsets of $(y,z)$, uniformly in $(t,x)$.
  {Assume moreover that \AssumptionPContinuous{} holds.}
  Then there exists $\bar T>0$ such that the singular HJB
  equation \eqref{eq:section4-singular-HJB} admits a unique paracontrolled
  solution $u$ on $[T-\bar T,T]$, with
  $(u,\nabla u)\in\mathscr D_{\bar T,T,b}^{a,\theta,\rho}$.
  If, in addition, there exists $M<\infty$ such that every paracontrolled
  solution satisfies, on its interval of existence,
  \begin{equation}
    \sup_t\left(\|u(t)\|_{L^\infty}
      +\|\nabla u(t)\|_{L^\infty}\right)\leq M,
    \label{eq:local-lipschitz-apriori-bound}
  \end{equation}
  then the solution extends uniquely to a global paracontrolled solution on
  $[0,T]$.
\end{proposition}

\begin{proof}
  Fix
  \[
    R>\|\bar h\|_{L^\infty}+\|\nabla\bar h\|_{L^\infty},
  \]
  and define $\pi_R^y:\mathbb R\to[-R,R]$ and
  $\pi_R^z:\mathbb R^d\to\overline B_R(0)$ by
  \[
    \pi_R^y(y):=
    \begin{cases}
      y,&|y|\leq R,\\
      \displaystyle R\frac{y}{|y|},&|y|>R,
    \end{cases}
    \qquad
    \pi_R^z(z):=
    \begin{cases}
      z,&|z|\leq R,\\
      \displaystyle R\frac{z}{|z|},&|z|>R.
    \end{cases}
  \]
  These are the metric projections onto the closed convex sets $[-R,R]$ and
  $\overline B_R(0)$, respectively. In particular, they are $1$-Lipschitz
  and equal the identity on their respective target sets. Define
  \[
    H_R(t,x,y,z)
    :=H\bigl(t,x,\pi_R^y(y),\pi_R^z(z)\bigr).
  \]
  The local assumptions on $H$ imply that $H_R$ is globally Lipschitz in
  $(y,z)$ and that $H_R(\cdot,\cdot,0,0)$ is bounded. The required
  $L^\infty$-continuity in time is also inherited by $H_R$. Moreover,
  \[
    H_R(t,x,y,z)=H(t,x,y,z)
    \quad\text{whenever } |y|\leq R\text{ and }|z|\leq R.
  \]
  Proposition~\ref{prop:global-lipschitz-paracontrolled} therefore gives a
  unique global paracontrolled solution $u_R$ of
  \[
    \left(\partial_t+\frac12\Delta+b\cdot\nabla\right)u_R
    =-H_R(t,x,u_R,\nabla u_R),
    \qquad u_R(T,\cdot)=\bar h.
  \]

  Since the terminal datum lies strictly inside the cutoff region and $u_R$
  and $\nabla u_R$ are continuous in time with values in
  $L^\infty$, there exists $\bar T>0$ such that
  \[
    \sup_{t\in[T-\bar T,T]}
    \left(\|u_R(t)\|_{L^\infty}
    +\|\nabla u_R(t)\|_{L^\infty}\right)<R.
  \]
  On this interval the cutoff is inactive, so $u_R$ solves the equation with
  generator $H$, which proves local existence.

  If $u$ and $v$ are two local paracontrolled solutions, their norms and the
  norms of their gradients are bounded on their common interval. Choosing a
  larger cutoff radius if necessary, both solve there the equation with the
  same globally Lipschitz truncated generator. The uniqueness argument from
  Proposition~\ref{prop:global-lipschitz-paracontrolled} therefore yields
  $u=v$.

  Suppose now that the a priori estimate
  \eqref{eq:local-lipschitz-apriori-bound} is available, and choose from the
  outset $R>M$.
  Let $(s_*,T]$ be the maximal interval on which the cutoff is inactive.
  On this interval the a priori estimate gives
  \[
    \sup_t\left(\|u_R(t)\|_{L^\infty}
      +\|\nabla u_R(t)\|_{L^\infty}\right)\leq M<R.
  \]
  If $s_*>0$, time continuity shows that the cutoff remains inactive on a
  strictly larger interval to the left of $s_*$, contradicting maximality.
  Hence $s_*=0$, and $u_R$ is a global solution of the original equation.
  Global uniqueness follows from the local uniqueness just proved.
\end{proof}

\begin{theorem}[Quadratic generators]
  \label{thm:quadratic-paracontrolled-global}
  Let $4/3<a<\theta<\gamma<2-\alpha$,
  $\rho\in((\theta-1)/2,(\gamma-1)/2)$ and
  $\bar h\in\mathcal C^\gamma$.
{Assume that Assumptions \assP{} and \assPContinuous{} hold,
  and suppose in addition that either Assumptions \assAMixed{} and \assB{} or
  Assumptions \assA{} and \assC{} hold.} Then
  \eqref{eq:section4-singular-HJB} admits a unique global paracontrolled
  solution $(u,\nabla u)\in\mathscr D_{T,T,b}^{a,\theta,\rho}$.
\end{theorem}

\begin{proof}
  Under \AssumptionB{}, the generator $H$ is locally Lipschitz in $(y,z)$,
  uniformly in $(t,x)$, and bounded on bounded subsets of $(y,z)$. Under
  Assumptions \assC{} and \assP{}, the same properties follow from the independence of
  $H$ from $y$, the local Lipschitz condition in $z$ from \AssumptionP{}, and
  the uniform bound on $H(t,x,0)$ from \AssumptionC{}. Thus, in either case,
  Proposition~\ref{prop:local-lipschitz-paracontrolled} gives a unique
  paracontrolled solution $u$ on $[T-\bar T,T]$ for some $\bar T>0$.

  If \AssumptionB{} holds, Proposition~\ref{prop:converse-FK} identifies
  $(X^{s,x},u(\cdot,X^{s,x}),\nabla u(\cdot,X^{s,x}))$ with a solution of the
  singular FBSDE. Theorem~\ref{thm:quadratic-BSDE-apriori-estimates} then
  yields a bound on $u$, uniformly over its interval of existence, by taking
  $s=t$ and using $X_t^{t,x}=x$. Since $u$ is continuous,
  Theorem~\ref{thm:nonlinear-FK} shows that it is a viscosity solution.
  If \AssumptionC{} holds, Corollary~\ref{maxpple} gives the uniform bound on
  $u$, while Theorem~\ref{paravicosity} shows directly that $u$ is a viscosity
  solution.

 {In either case, the preceding uniform bound on $u$, together
  with Corollary~\ref{cor:singular-viscosity-lipschitz-estimate}, gives a
  uniform bound on the spatial Lipschitz constant of $u$, and hence on
  $\nabla u$.}
  We may consequently invoke the second part of
  Proposition~\ref{prop:local-lipschitz-paracontrolled}, which extends the
  local solution uniquely to $[0,T]$.
\end{proof}

\section*{Acknowledgements}
The first, second, and third author acknowledge support from DFG CRC/TRR 388 ``Rough Analysis, Stochastic Dynamics and Related Fields"  - Project ID 516748464. The second and fourth authors are grateful for funding by the Deutsche Forschungsgemeinschaft (DFG, German Research Foundation) through EXC 2046, Berlin Mathematical School, and through IRTG 2544 ``Stochastic Analysis in Interaction" - Project ID 410208580. This material is partly based upon work supported by the National Science Foundation under Grant No. DMS-2424139, while the second author was in residence at the Simons Laufer Mathematical Sciences Institute in Berkeley, California, during the Fall 2025 semester.

\section*{AI usage statement}

Generative AI tools, specifically ChatGPT (versions GPT-5.4–GPT-5.6), were used during the preparation of this manuscript to assist with language editing, exposition, literature searches, and feedback on mathematical arguments. In addition, ChatGPT GPT-5.6 proposed using for Theorem 5.4 the proof strategy developed in \cite{LeyNguyen2017,PorrettaPriola2013}, the resulting argument was subsequently adapted to the present setting, developed, verified, and written by the authors. All mathematical claims, proofs, references, and final content were independently checked and approved by the authors, who take full responsibility for the correctness and content of the manuscript.

\end{document}